\documentclass[leqno,onefignum,onetabnum,eqref]{siamltex}
\usepackage[T1]{fontenc}
\usepackage[latin9]{inputenc}
\usepackage{bbm}
\usepackage{geometry}
\usepackage{amsmath,graphicx}

\usepackage{epsfig,amssymb,subfigure,version,fancybox,mathrsfs,amscd}
\usepackage{cases,bm,multirow}
\usepackage{algorithm,algorithmic}
\usepackage{epstopdf}
\usepackage{url,xcolor}
\usepackage[colorlinks,anchorcolor=blue,citecolor=blue]{hyperref}
\newtheorem{thm}{\bf Theorem}
\newtheorem{lem}{\bf Lemma}[section]

\graphicspath{{images/}}
\newtheorem{assump}{\bf Assumption}

\newcommand{\tabincell}[2]{
\begin{tabular}{@{}#1@{}}#2\end{tabular}
}

\title{
Overlapping Domain Decomposition\\ Methods for Ptychographic Imaging
}
\author{ Huibin Chang\thanks{Corresponding author. School of Mathematical Sciences, Tianjin Normal University, Tianjin, China, {\tt E-mail:changhuibin@gmail.com}} 
	\and Roland Glowinski\thanks{Department of Mathematics, University of Houston, Houston, USA }
	\and Stefano Marchesini\thanks{Computational Research Division, Lawrence Berkeley National Laboratory, Berkeley, USA} 
	\and 	Xue-Cheng Tai\thanks{Department of Mathematics, Hong Kong Baptist University, Kowloon Tong, Hong Kong}
	\and  Yang Wang\thanks{Department of Mathematics,	The Hong Kong University of Science and Technology,  Kowloon, Hong Kong
	}
	\and Tieyong Zeng\thanks{Department of Mathematics, The Chinese University of Hong Kong, Hong Kong}}
\begin{document}

\maketitle
\slugger{sisc}{xxxx}{xx}{x}{x--x}

\begin{abstract}
In ptychography experiments,   redundant scanning is usually required  to guarantee the stable recovery, such that  a huge amount of frames are generated,  and thus it  poses a great demand of parallel computing in order to solve this large-scale inverse problem.  
In this paper, we propose the overlapping {\bf D}omain {\bf D}ecomposition {\bf M}ethod{\bf s} (DDMs) to solve the nonconvex optimization problem in ptychographic imaging. They  decouple the problem defined on the whole domain into subproblems only defined  on the subdomains with synchronizing information in the overlapping regions of these subdomains, thus leading to highly parallel algorithms with good load balance.  More specifically,  for the nonblind recovery (with known probe in advance), by enforcing the continuity of the overlapping regions for the image (sample),   the nonlinear optimization model is established based on a novel {\bf s}mooth-{\bf t}runcated {\bf a}mplitude-{\bf G}aussian {\bf m}etric (ST-AGM). Such metric allows for fast calculation  of the proximal mapping with closed form, and meanwhile provides the possibility for the convergence guarantee of the first-order nonconvex optimization algorithm due to its Lipschitz smoothness. Then  the  {\bf A}lternating {\bf D}irection {\bf M}ethod of {\bf M}ultipliers (ADMM) is utilized to generate an efficient {\bf O}verlapping {\bf D}omain {\bf D}ecomposition based {\bf P}tychography  algorithm (OD${}^2$P) for the two-subdomain {\bf d}omain {\bf d}ecomposition (DD), where all subproblems can be computed with close-form solutions.  Due to the  Lipschitz continuity for the gradient of the objective function with  ST-AGM, the convergence of the proposed OD${}^2$P is derived  under  mild conditions.     
Moreover, it is extended to more general case including multiple-subdomain DD and blind recovery. 
Numerical experiments are further conducted to show the performance of proposed algorithms, demonstrating good convergence speed and robustness  to the noise.  { Especially, we report the virtual wall-clock time of proposed algorithm up to 10 processors, which shows  potential for upcoming  massively  parallel  computations.}
\end{abstract}




\section{Introduction}


X-ray ptychography  \cite{rodenburg2004phase}  enables imaging  with the large field of view and nano-scale resolution, by combining the technique of coherent diffraction imaging  and scanning transmission X-ray microscopy. It has been applied to vast research areas, that helps to reveal the structures and chemical contrast of nano-scale materials. 
Thus it is very meaningful to develop efficient algorithms for ptychographic imaging. As a special case for the phase retrieval problem (the inverse quadratic problem) arising in wide areas including optics, crystallography, materials sciences etc.,  various algorithms developed in \cite{marchesini2007invited, Shechtman2014} and references therein  can be exploited to tackle this problem.  Especially, for ptychography, recent algorithms  \cite{chapman1996phase,rodenburg2004phase,guizar2008phase,thibault2009probe,maiden2009improved,wen2012,osti_1324480,hesse2015proximal,chang2018Blind,fannjiang2019fixed} have been developed in order to further improve the performance and work in various settings, e.g. partial coherence \cite{chang2018partially} and blind recovery \cite{thibault2009probe, maiden2009improved} for unknown probe.

As an important property of ptychography, a huge amount of  the frames are usually collected such that the iterative imaging algorithms usually requires a large memory footprint and high computational cost. Therefore, it poses a great demand to develop parallel algorithms \cite{nashed2014parallel,guizar2014high,marchesini2016sharp, enfedaque2019high}. Namely, the asynchronous parallel algorithm \cite{nashed2014parallel} was performed on partitions of entire dataset asynchronously to get the sub-images, and then merge these sub-images  to the whole image. However, due to possible drifts  and phase ambiguities (linear phase ambiguity \cite{guizar2011phase} for blind recovery) for the recovered sub-images, high computation cost was required for pairwise  registration and correction of phase ambiguities.  To further improve the   asynchronous algorithm, the authors \cite{guizar2014high,nashed2014parallel} also provided   synchronous parallel algorithms, where either the boundary layers of the sub-images were synchronized by averaging all  local neighborhood \cite{nashed2014parallel}, or the entire image and probe were shared across all subdomains \cite{guizar2014high,marchesini2016sharp}.  
Hybrid parallel combined with the alternating projection   algorithm was implemented on GPU workstation \cite{enfedaque2019high}, and similarly to \cite{guizar2014high,marchesini2016sharp}, each subdomain has the individual copy of the whole image and probe leading to high communication cost.

Domain decomposition methods (DDMs) have played a great role in solving partial differential equations numerically \cite{lions1988schwarz,chan1994domain,glowinski1988domain}. Particularly, it allows for highly parallel computing with good load balance,  by decomposing the equations on whole domain to the problems on relatively small subdomains with information synchronization on the partition interfaces. It has also been successfully extended to large-scale image restoration, image reconstruction and other inverse problems, e.g. \cite{xu2010two,fornasier2010convergent,Chang2015sims,1930-8337_2015_1_163,lee2017primal,jiang2018,langer2019overlapping,lee2019finite,chen2019parallel,TOURNIER201988}  and references therein.  To the best of our knowledge, most of existing studies focused on the linear inverse problems with convex minimization modeling.  In this paper, we will explore the potential using DDMs to realize parallel computing for ptychographic imaging, which essentially corresponds to a nonconvex quadratic inverse problem.

However, it seems quite difficult to design corresponding DDMs for the ptychographic imaging. On one hand, the forward operator for ptychography involves the Fourier transform (integral operator) of the multiplication of the probe and sliding patch of the image, such that direct non-overlapping decomposition of  the image into subdomains  cannot split the measured frames completely, leading to both high computational and communication cost (see a related work of the DDM for nonlocal operator \cite{chang2014domain}). On the other hand, there is lack of convergence studies of DDMs for nonconvex minimization problems.  Especially, the convergence for the related first-order operator splitting algorithms seems difficult to guarantee, since the convention objective function \cite{osti_1324480,chang2016Total} is nonsmooth, and the smoothed version \cite{chang2018Blind} based operator-splitting algorithm cannot get rid of additional inner loop.

We aim to design more efficient DDMs for ptychographic imaging with convergence guarantee. Simply one can see that the non-overlapping domain decomposition (DD) of the image (sample) cannot reduce both the computation and communication  costs compared to the overlapping DD due to the Fourier transform. Therefore, we will consider the overlapping DD motivated by the overlapping scan style of ptychography experiments, such that the decomposed frames on the subdomains are completely decoupled. By enforcing the continuity of the image on joint layers of overlapping subdomains, we propose the  {\bf O}verlapping {\bf DD} based {\bf P}tychography algorithm (OD$^2$P) in  the case of two-subdomain DD, utilizing the  Alternating Direction Methods of Multipliers (ADMM \cite{Glowinski1989,Wu&Tai2010,boyd2011distributed}) in order to solve the constraint optimization problem with a modified metric. It is naturally  extended to the case of the multiple-subdomain DD.  As the size of the probe is much smaller than that of the image, the DDM for the blind recovery to recover the image and probe jointly is further proposed by sharing the joint probe across all subdomains, combined with the overlapping DD of the image.   Comparing to the existing parallel imaging algorithms, there is no need to correct the phase ambiguity, and moreover, a clear mathematical framework is given with lower communication cost and strict theoretical convergence under quite mild conditions.  The main contributions of this paper are summarized below:
\begin{itemize}
    \item  A  novel smooth truncated  metric is introduced to measure the recovered intensity and target,  which allows for fast calculation of the proximal mapping with closed form, and meanwhile  guarantees the convergence of the first-order optimization algorithm due to its Lipschitz smoothness.

    \item An overlapping DD is introduced to solve ptychography reconstruction, where a narrow layer of the whole sample is shared with adjacent subdomains.  We further solve the ptychography model on the overlapping DD by using ADMM, which is  essentially a fully splitting algorithm on the sense that all subproblems have closed-form solutions. Furthermore, with mild condition of the stepsizes,  the convergence is derived. To the best of our knowledge, it is the first time to prove the convergence for DDMs in nonconvex problems.
    \item The DDMs are further extended to multiple-subdomain DD and blind recovery.  For blind ptychography, we further propose  a new variational model with the compact support set of Fourier transform of the probe, to get rid of grid pathology for grid scanning. 
    Especially, we report the virtual wall-clock time of proposed algorithm for multiple DD, { which shows high speedup efficiency up to 10 processors and  potential for upcoming  massively parallel computations.}
\end{itemize}

The rest of paper is given below. In section \ref{sec2},  the proposed OD$^2$P based on the two-subdomain DD is presented for nonblind recovery with convergence guarantee under mild conditions of the stepsizes.  It is further extended to the case of multiple-subdomain DD and blind recovery in section \ref{sec3}. Numerous experiments are conducted in section \ref{sec4}. Finally, conclusions and future work are given in section \ref{sec5}.

\section{DDM for nonblind ptychography: Two-subdomain case}
\label{sec2}
In a standard ptychography experiment,  a localized coherent X-ray probe (or illumination) $w$  scans through an image (or sample) $u$, while the detector
collects a sequence of phaseless intensities  in the far field.
Throughout this paper, we consider the following discrete setting.
The variable \noindent$u\in\mathbb C^n$ (the complex Euclidean space) corresponds to a 2D image (sample) with $\sqrt{n}\times\sqrt{n}$ pixels, and $w\in\mathbb C^{\bar m}$  is a localized 2D  probe   with  $\sqrt{\bar m}\times \sqrt{\bar m}$ pixels, where both $u$ and $w$ are written as  a vector by a lexicographical order. A stack of phaseless measurements ~$\tilde f_j\in \mathbb R_+^{\bar m}~\forall 0\leq j\leq J-1$ ($\mathbb R_+^{\bar m}$ is the Euclidean space with non-negative elements) is collected with $
\tilde f_j=|\mathcal F(w\circ \mathcal S_j u)|^2,
$
where notations $|\cdot|, ()^2$ represent the element-wise absolute, and square values of a vector respectively, 
 $\circ$ denotes the element-wise multiplication, and $\mathcal F$ denotes the normalized discrete Fourier transformation. The operator $\mathcal S_j\in \mathbb R^{\bar m\times n}$ is denoted by a binary  matrix that specifies a small window  with the index $j$ and  size  $\bar m$ over the entire image $u$. The general blind ptychography (BP) problem can then be expressed as follows:
\begin{equation}\label{PtychoPR}
{\text{\qquad To find ~}w \in \mathbb C^{\bar m}\text{~and~} u\in \mathbb C^n,~~} \text{such that}~~|\mathcal F(w\circ \mathcal S_j u)|^2=\tilde f_j\ \forall 0\leq j\leq J-1.
\end{equation}

In this section, we assume that the probe $w$ is known in advance. 
Denote the linear operator $A:\mathbb C^{n}\rightarrow \mathbb C^m$ {( $m=J\times \bar m$)} as
$Au:=\big((\mathcal F (w\circ \mathcal S_0 u))^T,(\mathcal F (w\circ \mathcal S_1 u))^T
,\cdots, (\mathcal F (w\circ \mathcal S_{J-1} u))^T\big)^T\ \ \forall \ u\in \mathbb C^n,$
 such that the nonblind ptychography reconstruction can be modeled below:
\begin{equation}\label{eq:ptychoModel}
\text{To find the image $u\in\mathbb C^n$,\ \ s.t. \ }|Au|^2=f,
\end{equation}
with  the phaseless data $f:=(\tilde f^T_0,\tilde f^T_1,\cdots,\tilde f^T_{J-1})^T\in \mathbb R^m$ and the given probe $w.$

\subsection{Overlapping domain decomposition}
Denote the  whole region $\Omega:=\{0, 1, 2, \cdots, n-1\}$ in the discrete setting. There exists the two-subdomain overlapping DD $\{\Omega_d\}_{d=1}^2$,  such that  $\Omega=\bigcup_{d=1}^{2} \Omega_d$ with $\Omega_d:=\{l^d_0, l^d_1,\cdots, l^d_{n_d-1}\}$, and the overlapping region is denoted as $\Omega_{1,2}:=\Omega_1\cap\Omega_2=\{\hat l_0, \hat l_1, \cdots, \hat l_{\hat n-1}\}.$
Here we consider a special overlapping DD. As shown in Fig. \ref{fig1} (b) (totally $4\times 4$ frames) for grid scan,  each subdomain contains a complete  scanning frame, i.e. the right (or left) boundaries of frames near the left(or right) subdomain  interface lies exactly on the right(or left) boundary of subdomains. 
Denote the restriction operators $R_1, R_2$ as $R_d u=u|_{\Omega_d}$ and $R_{1,2}u=u|_{\Omega_{1,2}}$, i.e. $(R_d u)(j)=u(l^d_j)\forall \ 0\leq j\leq n_d-1,$
$(R_{1,2}u)(j)=u(\hat l_j)\forall 0\leq j\leq \hat n-1.$
Then two groups of localized shift operators can be introduced $\{\mathcal S_{j_d}^d\}_{j_d=0}^{J_d-1}, d=1,2$, with $\sum_d J_d=J.$ 
Therefore, we can denote the linear operators $A_1, A_2$  on the subdomains as
\begin{equation}
A_d u_d:=\big((\mathcal F(w\circ\mathcal S_0^d u_d))^T, (\mathcal F(w\circ\mathcal S_1^du_d))\big)^T,\cdots, (\mathcal F(w\circ\mathcal S_{J_d-1}^du_d))^T)^T\ \ d=1,2.
\label{eq:Ad_def}
\end{equation}
Naturally, the measurement $f$ is also decomposed to two non-overlapping parts $f_1, f_2$, i.e. $f_d:=|A_d R_d u|^2$ assuming that $(u,f)$ follows \eqref{eq:ptychoModel}.

\begin{figure}[ht]
\begin{center}
\subfigure[]{\includegraphics[width=.32\textwidth]{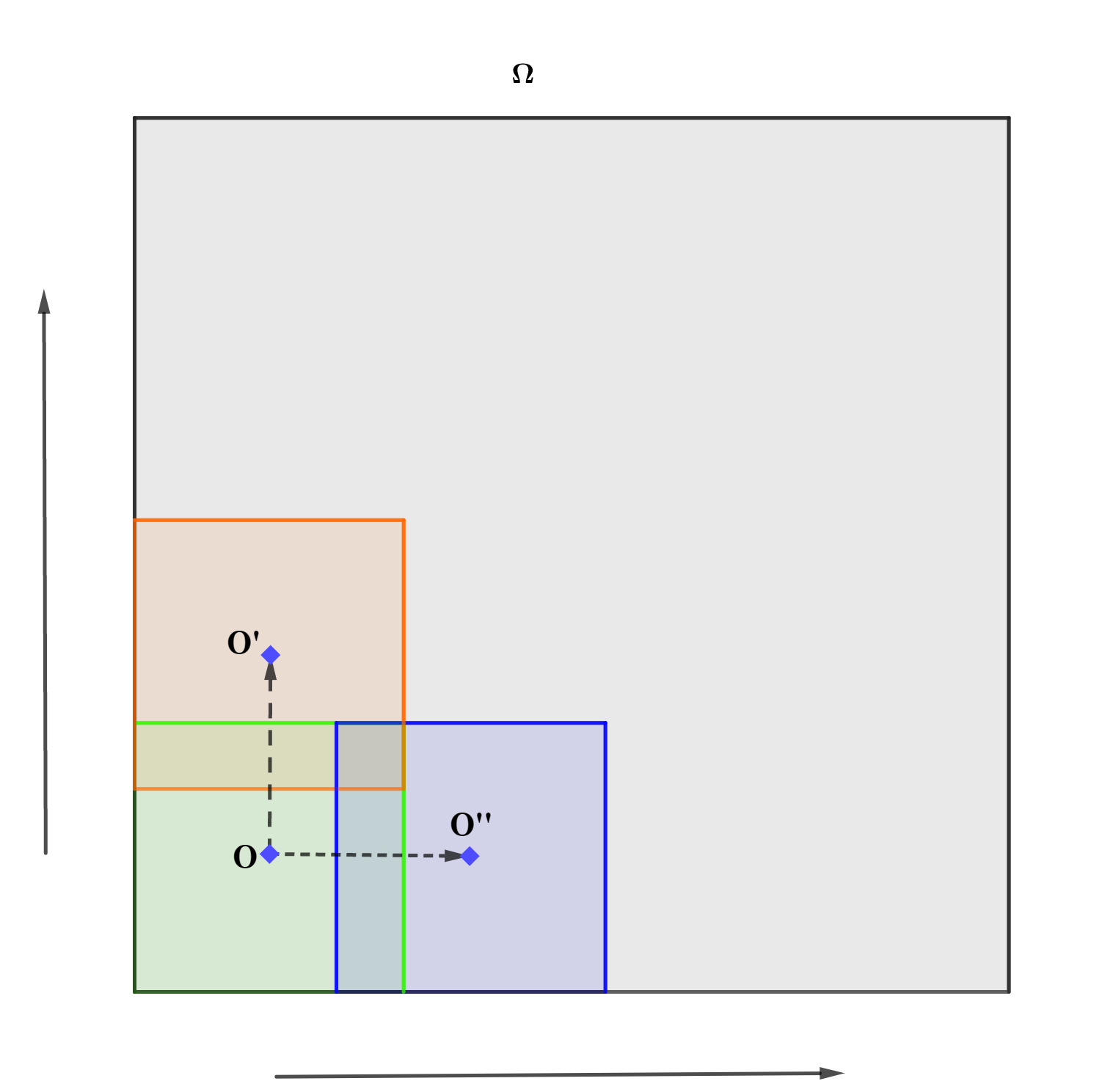}}
\subfigure[]{\includegraphics[width=.3\textwidth]{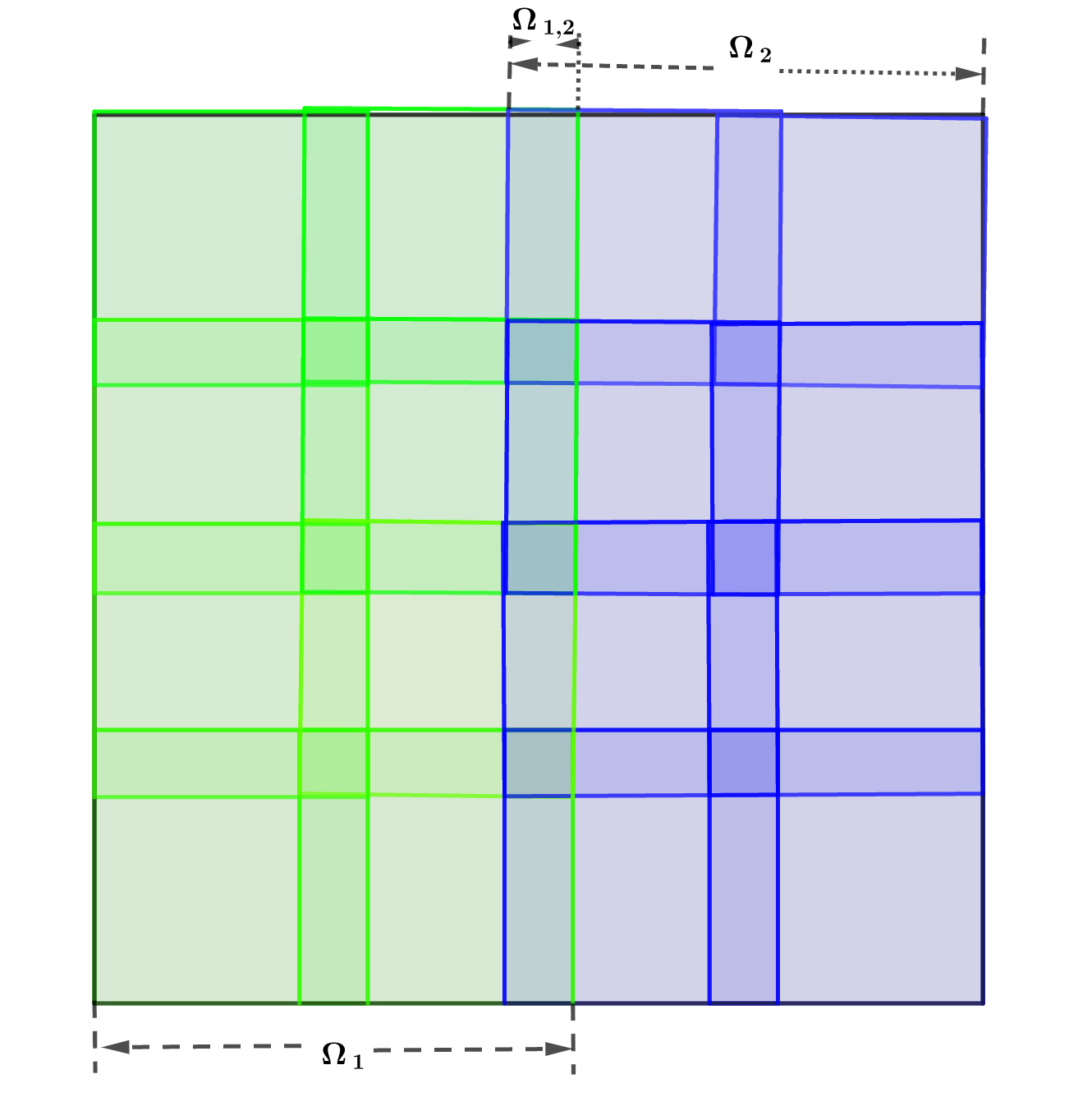}}
\end{center}
    \caption{(a) Ptychography scan in the domain $\Omega$ (grid scan): the starting scan centers at point {\bf O}, and then move up (or to the right) with the center point {\bf O'} (or {\bf O"} ); (b) Two-subdomain DD (totally $4\times 4$ frames): The subdomains $\Omega_1, \Omega_2$ are generated by two $4\times 2$-scans, and the overlapping region $\Omega_{1,2}=\Omega_1\cap\Omega_2$. }
\label{fig1}
\end{figure}

We remark that as shown in Fig. \ref{fig1}, the width of the overlapping region $\Omega_{1,2}$ is $\sqrt{\bar m}-n_{stepsize}$, with the scan stepsize $n_{stepsize}$ pixels (assuming the stepsizes are same for $x, y-$directions). Therefore, letting the stepsize $n_{stepsize}<\sqrt{\bar m}$ guarantees the sufficient overlap of the two-subdomain DD, that immediately demonstrates that   sufficient redundancy of the scan for ptychography leads to enough overlap of the DD. We also remark that  one can first divide the measurements with non-overlapping partitions as shown in Fig. \ref{fig1} (b), that will automatically generate the overlapping DD.
  
\subsection{A new nonlinear optimization model}\label{sec-ST-AGM}
For simplicity,   we consider the decomposition with only two subdomains, and rewrite the problem \eqref{eq:ptychoModel} on the whole region to the following equations
\begin{equation}\label{eq:forward}
\begin{split}
|A_1 u_1|^2=f_1,~~|A_2 u_2|^2=f_2,
\end{split}
\end{equation}
with the continuity of the overlapping regions
\begin{equation}\label{eq:continuity}
\mathcal {\mathbf  \pi}_{1,2}u_1=\mathcal {\mathbf  \pi}_{2,1}u_2,
\end{equation}
where the operators $\pi_{1,2}$ (restriction from $\Omega_1$ into $\Omega_{1,2}$) and $\pi_{2,1}$ (restriction from $\Omega_2$ into $\Omega_{1,2}$) are denoted as
$\pi_{1,2}u_1:= R_{1,2} R_1^T u_1,$ and $\pi_{2,1}u_2:= R_{1,2} R_2^T u_2.$ 
Remarkably, the continuity condition is essential for the perfect reconstruction of the ptychography. Without this condition, the redundancy of the boundary layer may not be sufficient such that there will be visible artifacts in the final reconstruction results and meanwhile phase ambiguities need to be corrected. 

We will discuss how to solve the above problem efficiently. One may solve the problem using a Schwarz alternating method \cite{lions1988schwarz}. However, directly solving this multivariable quadratic system is generally NP-complete. Alternatively, we seek  the solution by  nonlinear optimization. With the popular amplitude-Gaussian metric (AGM) \cite{wen2012,fannjiang2019fixed} to measure the fitting errors, one can readily establish the following  nonlinear optimization model 
\begin{equation*}
\min_{u_1, u_2}\sum\nolimits_{d=1}^2\tfrac12\left( \big\|| A_d u_d|-\sqrt{f_d}\big\|^2\right),\ \ s.t. \ \mathcal {\mathbf  \pi}_{1,2}u_1=\mathcal {\mathbf  \pi}_{2,1}u_2,
\end{equation*} 
where $\sqrt{\cdot}$ denotes the element-wise square root of a vector, and $\|\cdot\|$ denotes the standard $\ell^{2}$ norm in real or complex Euclidean space.  Due to lack of smoothness,  one may introduce a modified metric \cite{guizar2008phase,chang2018Blind,gao2019solving} with additional penalization factor in order to gain the convergence guarantee. 
However, an inner loop has to be introduced \cite{chang2018Blind}, since the proximal mapping of the modified metric does not have closed form solution. 

Readily one sees that the AGM is smooth  except at the origin. Instead of smoothing the function by adding a global penalization factor \cite{guizar2008phase,chang2018Blind,gao2019solving}, a more simple but effective way is to  cut off the AGM near the origin, and then add back a smooth function which can keep the global minimizer unchanged. At the same time, it should not  increase  computational cost compared with AGM.   Hence, a novel smooth truncated AGM (ST-AGM) $\mathcal G_\epsilon(\cdot; f)$ with truncation parameter $\epsilon>0$ is designed below:
\begin{equation}\label{eq:STAGM}
\mathcal G_\epsilon(z; f):=\sum\nolimits_j g_\epsilon(z(j); f(j)),
\end{equation}
where $\forall\ \  x\in \mathbb C, b\in\mathbb R^+,$
\begin{equation}\label{eq:obj-1}
g_{\epsilon}(x; b):=
\left\{
\begin{split}
&\frac{1-\epsilon}{2}\left( b-\tfrac{1}{\epsilon}{|x|^2}\right), \text{\ \ if\ } |x|<\epsilon \sqrt{b};\\
&\frac{1}{2} \big||x|-\sqrt{b}\big|^2, \text{\ \ \ \ \ otherwise.}
\end{split}
\right.
\end{equation}

One readily sees that the new metric is a mixture of the AGM and a quadratic function (smooth connected). Please also see Fig. \ref{fig0} for the landscape of ST-AGM in  1-dimension: Among four different plots with $\epsilon=0, 0.1, 0.5, 1.5$, one can observe that  the minimizer does not change w.r.t. different $\epsilon$ ($0\leq \epsilon<1$); However, if $\epsilon>1$, the global minimizer moves to the origin. Hence, to keep the minimizer unchanged, we set  $\epsilon\in (0,1)$  hereafter.

Denote the corresponding proximal mapping of $g_\epsilon$ as
$
\mathrm{Prox}_{\mathcal G_\epsilon;\lambda}(y):=\arg\min_z \mathcal G_{\epsilon}(z; f) +\tfrac{\lambda}{2}\|z-y\|^2.
$
After simple calculations,  one derives the following closed form solution with $\epsilon\in (0,1)$:
\begin{equation}
\ \ \ \ (\mathrm{Prox}_{\mathcal G_\epsilon;\lambda}(y))(j)=\mathrm{sign}(y(j))\times
\left\{
\begin{split}
&\max\Big\{0, \tfrac{\lambda|y(j)|}{\lambda-\tfrac{1-\epsilon}{\epsilon}}\Big\},\text{\ \ \ \ \ if } |y(j)|<\epsilon-\tfrac{1-\epsilon}{\lambda}\sqrt{f(j)};\\
&\tfrac{\sqrt{f(j)}+\lambda|y(j)|}{1+\lambda},\text{\ \ \ \ \ otherwise. \ }\\
\end{split}
\right. 
\label{eq:Prox_sol}
\end{equation}
\begin{figure}[h!]
\begin{center}
\includegraphics[width=.45\textwidth]{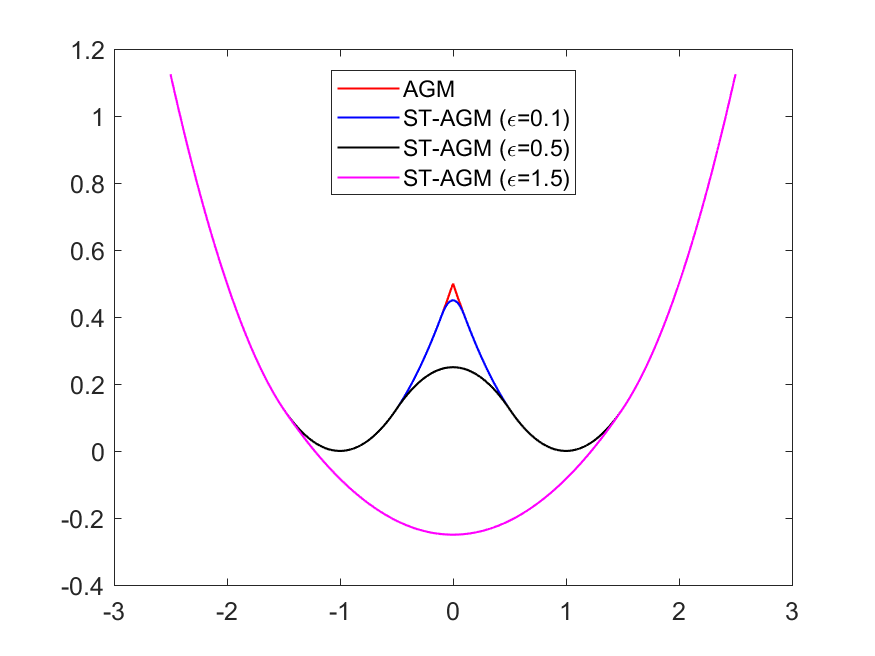}
\end{center}
    \caption{Landscapes of ST-AGM as $g_\epsilon(x, 1)\ \forall x\in\mathbb R$ ($\epsilon=0, 0.1, 0.5, 1.5$). If $\epsilon=0$, ST-AGM is exactly the same as AGM.
}
\label{fig0}
\end{figure}
One can further show that $\mathcal G_\epsilon(z; f)$ is Lipschitz differentiable by the following lemma.
\begin{lem}
Denote $L:=\frac{2}{\epsilon}-1$.
The function $\mathcal G_\epsilon(\cdot; f)$ is Lipschitz differentiable, i.e. 
\begin{equation}\label{eq:Lips}
 \|\nabla_z \mathcal G_{\epsilon}(z; f)-\nabla_z \mathcal G_{\epsilon}(y; f)\|
 \leq L\| z-y \|  \ \ \  \forall \ z, y\in \mathbb C^m.
\end{equation}
\end{lem}

\begin{proof}
We will  show the Lipschitz continuity of $\nabla_x g_\epsilon(\cdot; b)$.  
Denoting $\mathcal K_1:=\{x\in\mathbb C:\ |x|<\epsilon \sqrt{b}\}$ and $\mathcal K_2:=\mathbb C\setminus \mathcal K_1.$
Readily one knows that $g_\epsilon(x;b)\in \mathcal C^1$ if $\epsilon>0,$  where the gradient is given below
\begin{equation}\label{eq:objgrad}
\nabla_x g_\epsilon(x; b)=\left\{
\begin{split}
& (1-\tfrac{1}{\epsilon}) x, \text{\ \ if\ } x\in \mathcal K_1;\\
& (1-\tfrac{\sqrt{b}}{|x|})x, \text{\ otherwise.}
\end{split}
\right.
\end{equation}
It is not difficult to see that
\begin{equation}\label{eq:L1}
|\nabla_x g_\epsilon(x_1; b)-\nabla_x g_\epsilon(x_2; b)|=  |\tfrac{1}{\epsilon}-1|\ |x_1-x_2|\ \ \ \forall x_1, x_2\in \mathcal K_1.
\end{equation}
One can also get
\begin{equation}\label{eq:L2}
\begin{split}
&|\nabla_x g_\epsilon(x_1; b)-\nabla_x g_\epsilon(x_2; b)|
\leq
|(1-\tfrac{\sqrt{b}}{|x_1|})(x_1-x_2)|+|(\tfrac{\sqrt{b}}{|x_1|}-\tfrac{\sqrt{b}}{|x_2|})x_2|\\
\leq &|1-\tfrac{\sqrt{b}}{|x_1|}| \ |x_1-x_2|+\tfrac{\sqrt{b}}{|x_1|}\big||x_1|-|x_2|\big|
\leq (|\tfrac{1}{\epsilon}-1|+\tfrac{1}{\epsilon})|x_1-x_2|\ \ \ \forall x_1, x_2\in \mathcal K_2.
\end{split}
\end{equation}
One can further derive
\begin{equation}\label{eq:L3}
\begin{split}
&|\nabla_x g_\epsilon(x_1; b)-\nabla_x g_\epsilon(x_2; b)|
=|(1-\tfrac{1}{\epsilon}) x_1-(1-\tfrac{\sqrt{b}}{|x_2|})x_2|\\
\leq &|1-\tfrac{1}{\epsilon}| \ |x_1-x_2|+|x_2||\tfrac{1}{\epsilon}-\tfrac{\sqrt{b}}{|x_2|}|
\leq (|\tfrac{1}{\epsilon}-1|+\tfrac{1}{\epsilon})|x_1-x_2| \ \ \ \forall x_1\in\mathcal K_1, x_2\in \mathcal K_2.
\end{split}
\end{equation}
By \eqref{eq:L1}-\eqref{eq:L3}, one gets
\begin{equation}\label{eq:L4}
|\nabla_x g_\epsilon(x_1; b)-\nabla_x g_\epsilon(x_2; b)| \leq    (|\tfrac{1}{\epsilon}-1|+\tfrac{1}{\epsilon})|x_1-x_2|\ \ \ \forall \ x_1, x_2\in\mathbb C,
\end{equation}
that demonstrates the Lipschitz continuity of $\nabla_x g_\epsilon(\cdot; b).$

Finally one can derive
\[
\begin{split}
&\|\nabla_z \mathcal G_{\epsilon}(z; f)-\nabla_z \mathcal G_{\epsilon}(y; f)\|
=\sqrt{\sum\nolimits_j \|\nabla_{x} g_{\epsilon}(z(j); f(j))-\nabla_{x}  g_{\epsilon}(y(j); f(j))\|^2}\\
\leq&(|\tfrac{1}{\epsilon}-1|+\tfrac{1}{\epsilon})\| z-y \| \ \ \forall \ z,y\in\mathbb C^m,
\end{split}
\]
where the last relation is based on \eqref{eq:L4}. That concludes  this lemma.
\end{proof}

{We  remark that  when considering the following optimization problem w.r.t. the solution on the entire region
\[
\min\nolimits_u \mathcal G_\epsilon(Au;f),
\]
one can readily get the first-order optimality condition by \eqref{eq:objgrad}
\[
A^*\big((\bm 1-\tfrac{\sqrt{f}}{\max\{\epsilon \sqrt{f}, |Au|\}})\circ Au\big)=0,
\]
where we assume that $f>0$ as \cite{osti_1324480} (one can always remove the zero intensity values). It does not have singularity at the origin, and  is exactly the first-order optimality condition for the AGM based optimization problem in the case of $\epsilon=0$. Moreover, one can notice that the stationary points of the ST-AGM defined in  \eqref{eq:obj-1}  keep unchanged compared with those of AGM, and new metric only differs from the AGM  near the origin (local maximum) by setting the truncation parameter   $\epsilon\in(0,1)$ in \eqref{eq:obj-1}. Therefore, it should not affect the performance of new algorithm developed in the paper. 
  }

Hereafter, we consider the nonlinear optimization problem with ST-AGM. In order to enable the parallel computing of $u_1$ and $u_2$, we introduce an auxiliary variable $v$ which is only defined { in the overlapping region $\Omega_{1,2}$,} and then are concerned with the following model:
\begin{equation}
\begin{split}
    &\min\nolimits_{u_1, u_2, v}\sum\nolimits_{d=1}^2\mathcal G_\epsilon(A_d u_d;f_d)\ \text{~~~~s.t.}\ \mathcal {\mathbf  \pi}_{d,3-d}u_d-v=0,~ d=1,2.
\end{split}
\label{model:twosubdomain}
\end{equation}

\subsection{{\bf O}verlapping DD based {\bf p}tychography algorithm (OD${}^2$P)}
As one of the most popular first-order operator-splitting algorithm, the ADMM  \cite{Glowinski1989,Wu&Tai2010,boyd2011distributed} has been applied to ptychography reconstruction, showing competitive performances for Poisson noisy measurements as well as  large scan stepsizes  \cite{wen2012,chang2018Blind, fannjiang2019fixed}. We will also solve the proposed model \eqref{model:twosubdomain} using ADMM.

{
In order to develop an iterative scheme without inner loop as well as with fast convergence for large-step scan,} two  auxiliary variables $z_1, z_2$ are introduced below:
\begin{equation}
\begin{split}
    &\min\nolimits_{u_1, u_2, v, z_1, z_2}\sum\nolimits_{d=1}^2\mathcal G_\epsilon(z_d;f_d)\ \\
s.t.& \  \ \mathcal {\mathbf  \pi}_{d,3-d}u_d-v=0, \ \ \ A_d u_d-z_d=0,\ d=1, 2,
\end{split}
\label{eq:model-constraint}
\end{equation}
and then introduce the corresponding augmented Lagrangian for \eqref{eq:model-constraint} as follows:
\begin{equation}\label{eq:augLag}
\begin{split}
&\mathcal L(u_1, u_2, v, z_1, z_2,\Lambda_{1,2},\Lambda_{2,1}, \Gamma_1, \Gamma_2):=\sum\nolimits_{d=1}^2\mathcal G_\epsilon(z_d;f_d)\ \\
&\ \ \ \ + r\sum\nolimits_{d=1}^2\left( \Re(\langle \Lambda_{d,3-d}, \mathcal {\mathbf  \pi}_{d,3-d}u_d-v\rangle)+\tfrac{1}{2}\|\mathcal {\mathbf  \pi}_{d,3-d}u_d-v\|^2\right)\\
&\ \  \  +\eta\sum\nolimits_{d=1}^2\left(  \Re(\langle \Gamma_d, \mathcal A_d u_d-z_d \rangle) +\tfrac{1}{2}\|A_d u_d-z_d \|^2\right),
\end{split}
\end{equation}
{ with $\Re(\cdot)$ denoting the real part of a complex number.}
Then we are concerned with the following saddle point problem below:
\begin{equation}\label{eq:SaddlePoint}
\max\nolimits_{\Lambda_{1,2},\Lambda_{2,1},\Gamma_1,\Gamma_2}\qquad\min\nolimits_{u_1,u_2,v,z_1,z_2}\mathcal L(u_1, u_2, v, z_1, z_2,\Lambda_{1,2},\Lambda_{2,1}, \Gamma_1, \Gamma_2).
\end{equation}
To solve the saddle point problem \eqref{eq:SaddlePoint}, an iterative scheme using  the alternating minimization is given below: To  get $(u_1^{n+1}, u_2^{n+1}, v^{n+1}, z_1^{n+1}, z_2^{n+1}, \Gamma_1^{n+1}, \Gamma_2^{n+1}, \Lambda_{1,2}^{n+1}, \Lambda_{2,1}^{n+1})$ by the following steps:
\begin{equation}\label{eq:ADMM}
\left\{
\begin{split}
\text{Step 1.\ \ }&(v^{n+1}, z^{n+1}_1, z^{n+1}_2)=\arg\min\nolimits_{v, z_1, z_2} \mathcal L(u^{n}_1, u^{n}_2, v, z_1, z_2,\Lambda^{n}_{1,2},\Lambda^{n}_{2,1}, \Gamma^{n}_1, \Gamma_2^{n})
\\ &\hskip 2.7cm=\arg\min_{v, z_1, z_2} \sum\nolimits_{d=1}^2 \mathcal G_\epsilon(z_d;f_d)+\tfrac{\eta}{2}\sum\nolimits_{d=1}^2 \|z_d-(\Gamma^{n}_d+A_d u_d^{n}) \|^2\\
&\hskip 4.5cm  +\tfrac{r}{2}\sum\nolimits_{d=1}^2\left(\|v-(\mathcal {\mathbf  \pi}_{d,3-d}u^{n}_d+\Lambda^{n}_{d,3-d})\|^2\right);\\
\text{Step 2.\ \ }&\Gamma_d^{n+1}=\Gamma^n_d+ A_d u^{n}_d-z^{n+1}_d, d=1,2;\\
\text{Step 3.\ \ }&(u_1^{n+1}, u_2^{n+1})=\arg\min\nolimits_{u_1,u_2} \mathcal L(u_1, u_2, v^{n+1}, z_1^{n+1}, z_2^{n+1},\Lambda^n_{1,2},\Lambda^n_{2,1}, \Gamma^{n+1}_1, \Gamma^{n+1}_2)\\
 &\hskip 1cm =\arg\min\nolimits_{u_1, u_2}\tfrac{\eta}{2}\sum\nolimits_{d=1}^2 \|A_d u_d+\Gamma_d^{n+1}-z_d^{n+1} \|^2+\tfrac{r}{2}\sum\nolimits_{d=1}^2\|\mathcal {\mathbf  \pi}_{d,3-d}u_d-v^{n+1}+\Lambda^n_{d,3-d}\|^2;\\
 \text{Step 4.\ \ }&\Lambda_{d,3-d}^{n+1}= \Lambda^n_{d,3-d}+{\mathbf  \pi}_{d,3-d}u^{n+1}_d-v^{n+1}, d=1,2;\\
\end{split}
\right.
\end{equation}
with the approximation solutions $(u_1^{n}, u_2^{n}, v^{n}, z_1^{n}, z_2^{n}, \Gamma_1^{n}, \Gamma_2^{n}, \Lambda_{1,2}^{n}, \Lambda_{2,1}^{n})$ in the $n^{th}$ step.

One readily knows that the objective functions for all subproblems (sub-optimization problem) of proposed ADMMs are differentiable, such that we have
\begin{equation}\label{eq:Stationary1}
\left\{
\begin{split}
&0=\nabla_{z_d}\mathcal G_\epsilon(z_d^{n+1}; f_d)- \eta \Gamma_d^{n+1},\ \ d=1,2;\\
&0=v^{n+1}-\tfrac{1}{2}\sum\nolimits_{d=1}^2({\mathbf  \pi}_{d,3-d}u^{n}_d+\Lambda^{n}_{d,3-d});\\
&0=(r{\mathbf  \pi}_{d,3-d}^T{\mathbf  \pi}_{d,3-d}+\eta A_d^*A_d) u_d^{n+1}-\big(r{\mathbf  \pi}_{d,3-d}^T(v^{n+1}-\Lambda^{n}_{d,3-d})+\eta A_d^*(z^{n+1}_d-\Gamma^{n+1}_d)\big), d=1,2; \\
\end{split}
\right.
\end{equation}
where the third equation is derived by taking partial gradient of the augmented Lagrangian and the multiplier update in \eqref{eq:ADMM}.
By  \eqref{eq:Prox_sol} and the first two equations in \eqref{eq:Stationary1}, one can get closed forms of $z_d^{n+1}$ and $v^{n+1}$.  In order to get the solution for $u_d^{n+1}, d=1,2$, we further discuss how to solve the linear systems in the last two equations in \eqref{eq:Stationary1}. With the definition of $A_d$ in \eqref{eq:Ad_def}, 
one can get 
\begin{equation}\label{eq:AtA}
A_d^* A_d=\mathrm{diag}\big(\sum\nolimits_{j_d=0}^{J_d-1} (\mathcal S_{j_d}^d)^T |w|^2\big),~d=1,2.
\end{equation}
One can also get 
$
\pi_{d,3-d}^T\pi_{d,3-d}=R_dR_{d,3-d}^TR_{d,3-d}R_d^T=\mathrm{diag}(\bm\sigma_d),
~d=1,2,
$
with  $\bm\sigma_d\in \mathbb R^{n_d}$  denoted as 
\begin{equation}
\bm\sigma_d(j)=\left\{
\begin{split}
&1,\ \ \text{if \ } l_{j}^d\in \Omega_{1,2};\\
&0,\ \ \text{otherwise,}
\end{split}
\right.
\ \ \ \forall ~0\leq j\leq n_d-1.\ \
\end{equation}
Therefore, $u_d^{n+1} (d=1,2)$ can be determined explicitly as below: $\forall~0\leq j_d\leq n_d-1,$
\begin{equation}
u_d^{n+1}(j_d)=\left\{
\begin{split}
&\tfrac{ (r{\mathbf  \pi}_{d,3-d}^T(v^{n+1}-\Lambda^{n}_{d,3-d})+\eta A_d^*(z^{n+1}_d-\Gamma^{n+1}_d)) (j_d) }{r+\eta \sum\nolimits_{j_d=0}^{J_d-1} ((\mathcal S_{j_d}^d)^T |w|^2)(j_d) }, \text{~if~} l_{j_d}^d\in \Omega_{1,2};\\
&\tfrac{ (A_d^*(z^{n+1}_d-\Gamma^{n+1}_d)) (j_d) }{\sum\nolimits_{j_d=0}^{J_d-1} ((\mathcal S_{j_d}^d)^T |w|^2)(j_d) }, \text{~\quad\qquad\qquad\qquad\quad~otherwise.}
\end{split}
\right.\label{eq:Alg1-solver-u}
\end{equation}

The overall overlapping DD based ptychography  algorithm (OD${}^2$P) is given as below:
\vskip .2in
\begin{minipage}{.95\textwidth}
\begin{center}
\rule{\textwidth}{1mm}
\vskip .1in
\centering{Algorithm 1: {\bf O}verlapping {\bf DD} based {\bf P}tychography  algorithm (OD${}^2$P)}
\vskip .1in
\rule{\textwidth}{.8mm}
\begin{itemize}
\item[Step 0.] Initialize $u_1^0=\bm 1, u_2^0=\bm 1, v^0$, $z_1^0=A_1 u_1^0, z_2^0=A_2 u_2^0,$ and multipliers $\Gamma_1^0=\bm 0, \Gamma_2^0=\bm 0,\Lambda^0_{1,2}=\Lambda^0_{2,1}=\bm 0.$  Set $n:=0.$

\item[Step 1.] Update $z^{n+1}_1$ and $z_2^{n+1}$ in parallel by
\begin{equation}\label{eq:updateZ}
\begin{split}
&z_d^{n+1}=\mathrm{Prox}_{\mathcal G_\epsilon;\eta}(\Gamma^{n}_d+A_d u_d^{n}), d=1,2;\ \ \
\end{split}
\end{equation}
and  update $v^{n+1}$ by
\begin{equation}\label{eq:updateV}
v^{n+1}=\tfrac{1}{2}\sum\nolimits_{d=1}^2({{\mathbf  \pi}_{d,3-d}u^{n}_d+\Lambda^n_{d,3-d}});
\end{equation}

\item[Step 2.] Update multipliers $\Gamma_1^{n+1}, \Gamma_2^{n+1}$ in parallel by
\begin{equation}\label{eq:updateM-1}
\begin{split}
&\Gamma_d^{n+1}=\Gamma^n_d+ A_d u^{n}_d-z^{n+1}_d,\ \  d=1,2; \\
\end{split}
\end{equation}

\item[Step 3.] Update $u^{n+1}_1, u^{n+1}_2$ in parallel by \eqref{eq:Alg1-solver-u}.

\item[Step 4.] Update multipliers $\Lambda_{1,2}^{n+1}, \Lambda_{2,1}^{n+1}$  by
\begin{equation}\label{eq:updateM-2}
\begin{split}
& \Lambda_{d,3-d}^{n+1}= \Lambda^n_{d,3-d}+{\mathbf  \pi}_{d,3-d}u^{n+1}_d-v^{n+1},\ \ d=1,2.\\
\end{split}
\end{equation}
\item[Step 5.] If satisfying the stopping condition, then stop and output $u_1^{n+1}$ and $u_2^{n+1}$ as the final solution; otherwise set $n:=n+1$, and go to Step 1.
\end{itemize}
\vskip .1in
\rule{\textwidth}{1mm}
\end{center}
\end{minipage}

\subsection{Convergence analysis}

First by the last  equation in \eqref{eq:Stationary1} and the multipliers update in \eqref{eq:ADMM}, one can derive
\begin{equation}\label{eq:Stationary2}
\begin{split}
0&={\eta}A_d^*(A_d u_d^{n+1}+\Gamma_d^{n+1}-z_d^{n+1})+r\mathcal {\mathbf  \pi}_{d,3-d}^T\Lambda^{n+1}_{d,3-d}, d=1,2.\\
\end{split}
\end{equation}

\begin{assump}\label{assump1}
The matrices $A_d^*A_d~(d=1,2)$  are non-singular, i.e. there exists strictly positive constant $\hat c_d (d=1,2)$, s.t.
$
\|A_d h_d\|\geq \hat c_d \|h_d\|\ \ \ \forall h_d\in\mathbb C^{n_d}.
$
\end{assump}

{
The assumption is  to guarantee the boundedness of the iterative sequence { $\{u_d^n\}$, if $\{A_d u_d^n\}$} is bounded.  
By \eqref{eq:AtA}, the assumption holds with sufficiently small scan stepsize, i.e.  each scan position and its neighbours have sufficient overlapping (one can also refer to Remark 4.1 of \cite{chang2018Blind}). 
}

Due to \eqref{eq:Lips}, one can prove that the function $\mathcal G_\epsilon(\cdot;f)$ has the descent property \cite{chang2018Blind}:
\begin{equation}\label{eq:descent}
\mathcal G_\epsilon(z;f)-\mathcal G_\epsilon(\hat z;f)-\Re(\langle \nabla_z \mathcal G_\epsilon(\hat z; f), z-\hat z\rangle)\leq \  \tfrac{L}{2}\|z-\hat z\|^2\ \ \ \forall \ z, \hat z\in \mathbb C^m.
\end{equation}

\begin{lem}\label{le:2}
Denote $E(z)=\frac{1}{2}\|Tz- f\|^{2}+M(z)$, with convex function $M$ and linear mapping $T$.  Letting $z^*$ be the global minimizer, i.e. $E(z^*)=\arg\min_z E(z)$,  then we have
\begin{displaymath}\label{eq:convexDescent}
E(z)-E(z^*)\geq\tfrac12\|T(z-z^*)\|^{2}.
\end{displaymath}
\end{lem}
The proof of  above lemma is trivial, and therefore we omit the details. Denote $X^n:=(u^{n}_1, u^{n}_2, v^n, z_1^{n}, z_2^{n},$ $\Lambda^{n}_{1,2}, \Lambda^{n}_{2,1}, \Gamma^{n}_1, \Gamma_2^{n})$.
{ Then we can estimate the lower bound of the changes for the objective function values in the following lemma, that can further show  the non-increasing of the sequence $\{\mathcal L(X^n)\}$ with proper parameters $r$ and  $\eta$ (see Lemma \ref{lem:bound}).}
\begin{lem}\label{lem:non-increasing}
\begin{equation*}
\begin{split}
&\mathcal L(X^n)-\mathcal L(X^{n+1})\\
\geq&\left(\tfrac{r}{2}-\tfrac{2\eta^2}{r}\|A_1^*A_1\|^2\right)\ \|\pi_{1,2} (u_1^n-u_1^{n+1})\|^2+\left(\tfrac{r}{2}-\tfrac{2\eta^2}{r}\|A_2^*A_2\|^2\right)\  \|\pi_{2,1} (u_2^n-u_2^{n+1})\|^2\\
&+\left(\tfrac{\eta-3L}{2}-\tfrac{L^2}{\eta}-\tfrac{2(L+\eta)^2}{r^2} \|\pi_{1,2}A_1^*\|^2\right)\ \|z_1^n-z_1^{n+1}\|^2+\left(\tfrac{\eta-3L}{2}-\tfrac{L^2}{\eta}-\tfrac{2(L+\eta)^2}{r^2} \|\pi_{2,1}A_2^*\|^2\right)\ \|z_2^n-z_2^{n+1}\|^2\\
&+r\|v^n-v^{n+1}\|^2+\tfrac{\eta}{2}\sum\nolimits_{d=1}^2 \|A_d (u_d^n-u_d^{n+1})\|^2\\
\end{split}
\end{equation*}
\end{lem}

\begin{proof}
In order to estimate the decay of the augmented Lagrangian, the following relation \cite{chang2018Blind} is given below:
\begin{equation}\label{eq:QuaDiff}
\|z\|^2-\|y\|^2=\|z-y\|^2+2\Re\langle y, z-y\rangle\ \ \ \forall \ z, y\in\mathbb C^m.
\end{equation}

First for the $z_d-$subproblems,  by \eqref{eq:QuaDiff}, one readily has
\begin{equation*}
\begin{split}
&\mathcal L(X^n)-\mathcal L(u^{n}_1, u^{n}_2, v^n, z_1^{n+1}, z_2^{n+1},\Lambda^{n}_{1,2},\Lambda^{n}_{2,1}, \Gamma^{n}_1, \Gamma_2^{n})\\
=&\sum\nolimits_d \mathcal G_\epsilon(z^n_d;f_d)+\tfrac{\eta}{2} \|z^n_d-(\Gamma^{n}_d+A_d u_d^{n}) \|^2-\left(\mathcal G_\epsilon(z_d^{n+1};f_d)+\tfrac{\eta}{2} \|z_d^{n+1}-(\Gamma^{n}_d+A_d u_d^{n}) \|^2\right)\\
\stackrel{\eqref{eq:QuaDiff}}{=}&\sum\nolimits_d\left(\mathcal G_\epsilon(z^n_d;f_d)-\mathcal G_\epsilon(z^{n+1}_d;f_d)+\tfrac{\eta}{2} \|z^n_d-z_d^{n+1}\|^2+\eta \Re(\langle z_d^{n+1}-(\Gamma^{n}_d+A_d u_d^{n}), z_d^n-z_d^{n+1}\rangle)\right)\\
\stackrel{\eqref{eq:ADMM}}{=}&\sum\nolimits_d\left(\mathcal G_\epsilon(z^n_d;f_d)-\mathcal G_\epsilon(z^{n+1}_d;f_d)+\tfrac{\eta}{2} \|z^n_d-z_d^{n+1}\|^2-\eta \Re(\langle \Gamma^{n+1}_d, z_d^n-z_d^{n+1}\rangle)\right)\\
\stackrel{\eqref{eq:descent}}{\geq}&\sum\nolimits_d\left(\Re(\langle\nabla_z\mathcal G_\epsilon(z_d^n; f_d)-\eta\Gamma^{n+1}_d, z_d^{n}-z_d^{n+1}\rangle)+\tfrac{\eta-L}{2} \|z^n_d-z_d^{n+1}\|^2\right).
\end{split}
\end{equation*}
Further by the first equation in \eqref{eq:Stationary1}, \eqref{eq:Lips} and Cauchy's inequality $\Re(\langle z, y\rangle)\geq -\tfrac{L}{2}\|z\|^2-\tfrac{1}{2L}\|y\|^2$,
one gets
\begin{equation}\label{eq:ZD2-1}
 \begin{split}
&\mathcal L(X^n)-\mathcal L(u^{n}_1, u^{n}_2, v^n, z_1^{n+1}, z_2^{n+1},\Lambda^{n}_{1,2},\Lambda^{n}_{2,1}, \Gamma^{n}_1, \Gamma_2^{n})\geq\tfrac{\eta-3L}{2} \sum\nolimits_d\|z^n_d-z_d^{n+1}\|^2.
\end{split}
\end{equation}
For $v-$subproblem, it is quite standard, and we omit the details. Hence combined with \eqref{eq:ZD2-1}, one gets
\begin{equation}\label{eq:vD-1}
\begin{split}
&\mathcal L(X^n)-\mathcal L(u_1^{n}, u_2^{n}, v^{n+1}, z_1^{n+1}, z_2^{n+1},\Lambda_{1,2}^{n},\Lambda_{2,1}^{n}, \Gamma_1^{n}, \Gamma_2^{n})\geq\tfrac{\eta-3L}{2} \sum\nolimits_d\|z^n_d-z_d^{n+1}\|^2+r\|v^n-v^{n+1}\|^2.
\end{split}
\end{equation}
\vskip .1in
For the multipliers update $\{\Gamma^n_d\}$, one gets
\begin{equation}\label{eq:GD-1}
\begin{split}
&\mathcal L(u_1^{n}, u_2^{n}, v^{n+1}, z_1^{n+1}, z_2^{n+1},\Lambda_{1,2}^{n},\Lambda_{2,1}^{n}, \Gamma_1^{n}, \Gamma_2^{n})-
\mathcal L(u_1^{n}, u_2^{n}, v^{n+1}, z_1^{n+1}, z_2^{n+1},\Lambda_{1,2}^{n},\Lambda_{2,1}^{n}, \Gamma_1^{n+1}, \Gamma_2^{n+1})\\
=&-\eta\sum\nolimits_d\|\Gamma_d^n-\Gamma_d^{n+1}\|^2{\geq}-\tfrac{L^2}{\eta}\sum\nolimits_d\|z_d^{n}-z_d^{n+1}\|^2,
\end{split}
\end{equation}
where the last equation is derived by the first equation in  \eqref{eq:Stationary1} and \eqref{eq:Lips}.
\vskip .1in

For the sequences $\{u_d^n\}$, based on Lemma \ref{le:2}, one gets
\begin{equation}\label{eq:uD-1}
\begin{split}
&\mathcal L(u_1^{n}, u_2^{n}, v^{n+1}, z_1^{n+1}, z_2^{n+1},\Lambda_{1,2}^{n},\Lambda_{2,1}^{n}, \Gamma_1^{n+1}, \Gamma_2^{n+1})-
\mathcal L (u_1^{n+1}, u_2^{n+1}, v^{n+1}, z_1^{n+1}, z_2^{n+1},\Lambda_{1,2}^{n},\Lambda_{2,1}^{n}, \Gamma_1^{n+1}, \Gamma_2^{n+1})\\
\geq &\tfrac{\eta}{2}\sum\nolimits_d \|A_d (u_d^n-u_d^{n+1})\|^2+\tfrac{r}{2}\left(\|\pi_{1,2} (u_1^n-u_1^{n+1})\|^2+ \|\pi_{2,1} (u_2^n-u_2^{n+1})\|^2\right).
\end{split}
\end{equation}
\vskip .1in

For the multipliers $\{\Lambda^n_{1,2}\}$ and $\{\Lambda^n_{2,1}\}$ update, by \eqref{eq:Stationary2}, one has
$\Lambda^{n+1}_{1,2}=\tfrac{\eta}{r}\mathcal {\mathbf  \pi}_{1,2}A_1^*(A_1 u_1^{n+1}+\Gamma_1^{n+1}-z_1^{n+1})\\
                   =\tfrac{\eta}{r}{\mathbf  \pi}_{1,2}A_1^*A_1 {\mathbf  \pi}_{1,2}^T({\mathbf  \pi}_{1,2}u_1^{n+1})+\tfrac{1}{r}\mathcal {\mathbf  \pi}_{1,2}A_1^*\nabla_{z_1}\mathcal G_\epsilon(z_1^{n+1}; f_1)-\tfrac{\eta}{r}\mathcal {\mathbf  \pi}_{1,2}A_1^*z_1^{n+1},$
by the first equation in \eqref{eq:Stationary1} and $A_1^*A_1$ being diagonal matrix.
Then by \eqref{eq:Lips}, one can estimate the successive error of the multiplier $\{\Lambda^{n}_{1,2}\}$ as
$\|\Lambda_{1,2}^{n}-\Lambda_{1,2}^{n+1}\|\leq \tfrac{\eta}{r}\|A_1^*A_1\|\ \|\pi_{1,2}(u_1^n-u_1^{n+1})\|+ \tfrac{L+\eta}{r}\ \|\pi_{1,2}A_1^*\|\ \|z_1^n-z_1^{n+1}\|.
$
 By the inequality $(a+b)^2\leq 2(a^2+b^2)\ \forall\ a, b\in\mathbb R,$ one immediately obtains
\begin{equation}\label{eq:mult1}
\begin{split}
\|\Lambda_{1,2}^{n}-\Lambda_{1,2}^{n+1}\|^2\leq \tfrac{2\eta^2}{r^2}\|A_1^*A_1\|^2\ \|\pi_{1,2}(u_1^n-u_1^{n+1})\|^2+ \tfrac{2(L+\eta)^2}{r^2}\ \|\pi_{1,2}A_1^*\|^2\ \|z_1^n-z_1^{n+1}\|^2,\\
\end{split}
\end{equation}
Similarly, one can get
\begin{equation}\label{eq:mult2}
\begin{split}
\|\Lambda_{2,1}^{n}-\Lambda_{2,1}^{n+1}\|^2\leq \tfrac{2\eta^2}{r^2}\|A_2^*A_2\|^2\ \|\pi_{2,1}(u_2^n-u_2^{n+1})\|^2+ \tfrac{2(L+\eta)^2}{r^2}\ \|\pi_{2,1}A_2^*\|^2\ \|z_2^n-z_2^{n+1}\|^2.
\end{split}
\end{equation}
Combined with \eqref{eq:mult1} and \eqref{eq:mult2}, one gets
\begin{equation}\label{eq:LD1}
\begin{split}
&\mathcal L(u_1^{n+1}, u_2^{n+1}, v^{n+1}, z_1^{n+1}, z_2^{n+1},\Lambda_{1,2}^{n},\Lambda_{2,1}^{n}, \Gamma_1^{n+1}, \Gamma_2^{n+1})-
\mathcal L (X^{n+1})\\
= &-r\|\Lambda_{1,2}^n-\Lambda_{1,2}^{n+1}\|^2-r\|\Lambda_{2,1}^n-\Lambda_{2,1}^{n+1}\|^2\\
\geq&-\tfrac{2\eta^2}{r}\|A_1^*A_1\|^2\ \|\pi_{1,2}(u_1^n-u_1^{n+1})\|^2-\tfrac{2\eta^2}{r}\|A_2^*A_2\|^2\ \|\pi_{2,1}(u_2^n-u_2^{n+1})\|^2\\
    &- \tfrac{2(L+\eta)^2}{r^2}\ \|\pi_{1,2}A_1^*\|^2\ \|z_1^n-z_1^{n+1}\|^2-\tfrac{2(L+\eta)^2}{r^2}\ \|\pi_{2,1}A_2^*\|^2\ \|z_2^n-z_2^{n+1}\|^2.
\end{split}
\end{equation}
Summing up \eqref{eq:vD-1}, \eqref{eq:GD-1}, \eqref{eq:uD-1}, and \eqref{eq:LD1} concludes this lemma.
\end{proof}

\begin{lem}\label{lem:lowerBound}
The augmented Lagrangian is lower bounded, i.e.  $\mathcal L(X^n)>-\infty,$ if
\[
(r, \eta)\in\left\{(r, \eta)\in\mathbb R^2_+:\ \eta-1\geq 0, r-\eta\max\{\|\pi_{1,2}A_1^*\|,\|\pi_{2,1}A_2^*\|\}\geq 0\right\}.\]
\end{lem}

\begin{proof}
First we have
\begin{equation*}
\begin{split}
&\mathcal L(X^{n+1}):=\sum\nolimits_{d}\mathcal G_\epsilon(z_d^{n+1};f_d)\ \\
&\ \ \ \ +\tfrac{r}{2}\|\mathcal {\mathbf  \pi}_{1,2}u_1^{n+1}-v^{n+1}+\Lambda^{n+1}_{1,2}\|^2-\tfrac{r}{2}\|\Lambda^{n+1}_{1,2}\|^2\\
&\ \ \ \ +\tfrac{r}{2}\|\mathcal {\mathbf  \pi}_{2,1}u_2^{n+1}-v^{n+1}+\Lambda^{n+1}_{2,1}\|^2-\tfrac{r}{2}\|\Lambda^{n+1}_{2,1}\|^2\\
&\ \  \  +\eta\sum\nolimits_d\left( \tfrac{1}{2}\|A_d u_d^{n+1}-z_d^{n+1}+\Gamma_d^{n+1}\|^2-\tfrac{1}{2}\|\Gamma_d^{n+1}\|^2\right)\\
\stackrel{\eqref{eq:Stationary1}, \eqref{eq:Stationary2}}{=}&\ \ \ \tfrac{r}{2}\|\mathcal {\mathbf  \pi}_{1,2}u_1^{n+1}-v^{n+1}+\Lambda^{n+1}_{1,2}\|^2-\tfrac{\eta^2}{2r}\|\pi_{1,2} A_1^*(A_1 u_1^{n+1}+\Gamma_1^{n+1}-z_1^{n+1})\|^2\\
& +\tfrac{r}{2}\|\mathcal {\mathbf  \pi}_{2,1}u_2^{n+1}-v^{n+1}+\Lambda^{n+1}_{2,1}\|^2-\tfrac{\eta^2}{2r}\|\pi_{2,1} A_2^*(A_2 u_2^{n+1}+\Gamma_2^{n+1}-z_2^{n+1})\|^2\\
& +\sum\nolimits_d\left( \tfrac{\eta}{2}\|A_d u_d^{n+1}-z_d^{n+1}+\Gamma_d^{n+1}\|^2+\mathcal G_\epsilon(z_d^{n+1};f_d)-\tfrac{1}{2\eta}\|\nabla_{z_d}\mathcal G_\epsilon(z_d^{n+1};f_d)\|^2\right)\\
\hskip -.2in\geq&\ \ \ \tfrac{r}{2}\|\mathcal {\mathbf  \pi}_{1,2}u_1^{n+1}-v^{n+1}+\Lambda^{n+1}_{1,2}\|^2+\tfrac{r}{2}\|\mathcal {\mathbf  \pi}_{2,1}u_2^{n+1}-v^{n+1}+\Lambda^{n+1}_{2,1}\|^2\\
&+\sum\nolimits_d\left( \tfrac{\eta}{2}(1-\tfrac{\eta}{r}\max\{\|\pi_{1,2}A_1^*\|,\|\pi_{2,1}A_2^*\|\}) \ \|A_d u_d^{n+1}-z_d^{n+1}+\Gamma_d^{n+1}\|^2 \right)\\
&+\sum\nolimits_{d,j}\left(g_\epsilon(z_d^{n+1}(j),f_d(j))-\tfrac{1}{2\eta}\|\nabla_{x}g_\epsilon(z_d^{n+1}(j), f_d(j))\|^2\right).
\end{split}
\end{equation*}
By the definition of \eqref{eq:obj-1} and its gradient \eqref{eq:objgrad},
one gets
\begin{equation*}
\begin{split}
&\sum\nolimits_{d,j}\left(g_\epsilon(z_d^{n+1}(j),f_d(j))-\tfrac{1}{2\eta}\|\nabla_{x}g_\epsilon(z_d^{n+1}(j), f_d(j))\|^2\right)\\
=&\sum\nolimits_d\tfrac12(1-\tfrac{1}{\eta})\sum\nolimits_{j_1\in\{j:\ |z_d^{n+1}(j)|>=\epsilon \sqrt{f_d(j)}\}}
\big||z_d^{n+1}(j_1)|-\sqrt{f_d(j_1)}\big|^2\\
&+\sum\nolimits_d\sum\nolimits_{j_2\in\{j:\ |z_d^{n+1}(j)|<\epsilon \sqrt{f_d(j)}\}}
\left(\tfrac{1-\epsilon}{2}f_d(j_2)-\tfrac{1-\epsilon}{2\epsilon}(1+\tfrac{1-\epsilon}{\eta\epsilon})|z_d^{n+1}(j_2)|^2 \right).
\end{split}
\end{equation*}
Hence one gets
\begin{equation}\label{eq:LB}
\begin{split}
&\mathcal L(X^{n+1})\geq
\tfrac{r}{2}\|\mathcal {\mathbf  \pi}_{1,2}u_1^{n+1}-v^{n+1}+\Lambda^{n+1}_{1,2}\|^2+\tfrac{r}{2}\|\mathcal {\mathbf  \pi}_{2,1}u_2^{n+1}-v^{n+1}+\Lambda^{n+1}_{2,1}\|^2\\
&+\sum\nolimits_d\left( \tfrac{\eta}{2}(1-\tfrac{\eta}{r}\max\{\|\pi_{1,2}A_1^*\|,\|\pi_{2,1}A_2^*\|\}) \ \|A_d u_d^{n+1}-z_d^{n+1}+\Gamma_d^{n+1}\|^2 \right)\\
&+\sum\nolimits_d\tfrac12(1-\tfrac{1}{\eta})\sum\nolimits_{j_1\in\{j:\ |z_d^{n+1}(j)|>=\epsilon \sqrt{f_d(j)}\}}
\big||z_d^{n+1}(j_1)|-\sqrt{f_d(j_1)}\big|^2\\
&+\sum\nolimits_d\sum\nolimits_{j_2\in\{j:\ |z_d^{n+1}(j)|<\epsilon \sqrt{f_d(j)}\}}
\left(\tfrac{1-\epsilon}{2}f_d(j_2)-\tfrac{1-\epsilon}{2\epsilon}(1+\tfrac{1-\epsilon}{\eta\epsilon})|z_d^{n+1}(j_2)|^2 \right).
\end{split}
\end{equation}
Readily one knows the last term is bounded since $|z_d^{n+1}(j_2)|<\epsilon\sqrt{f_d(j_2)}$. By setting $\eta\geq 1$ and $r\geq\eta\max\{\|\pi_{1,2}A_1^*\|,\|\pi_{2,1}A_2^*\|\}$, one can immediately conclude this lemma.

\end{proof}

\begin{lem}\label{lem:bound}
Letting $(r, \eta)\in \mathcal K$ with
\begin{equation}
\label{eq:paramsSet}
\begin{split}
&\mathcal K:=\Big\{(r, \eta)\in \mathbb R^2_+:\ \ \eta-1\geq 0, \ \ { r-\eta\max\{2\|A_1^*A_1\|, 2\|A_2^*A_2\|,\|\pi_{1,2}A_1^*\|,\|\pi_{2,1}A_2^*\|\}>0,}\\
&\ \ \qquad\qquad\qquad\qquad \tfrac{\eta-3L}{2}-\tfrac{L^2}{\eta}-\tfrac{2(L+\eta)^2}{r^2} \max\{\|\pi_{1,2}A_1^*\|^2,  \|\pi_{2,1}A_2^*\|^2\}>0 \Big\},
\end{split}
\end{equation}
the following assertions hold:
\begin{itemize}
    \item[(1)] The sequence $\{\mathcal L(X^n)\}$ is non-increasing with the finite lower bound.
    \item[(2)] The following successive error converges to zero, i.e.
\begin{equation}
\begin{split}
\lim_{n\rightarrow+\infty}X^n-X^{n+1}=\bm 0,
\end{split}
\end{equation}
    \item[(3)] The iterative sequence $\{X^n\}$  is bounded  under Assumption \ref{assump1}.
\end{itemize}
\end{lem}

Before giving the detailed proof, we will show that the set $\mathcal K$ is nonempty.  By denoting $$c_0:=\max\{2\|A_1^*A_1\|, 2\|A_2^*A_2\|,\|\pi_{1,2}A_1^*\|,\|\pi_{2,1}A_2^*\|\},\qquad c_1:=2\max\{\|\pi_{1,2}A_1^*\|^2,  \|\pi_{2,1}A_2^*\|^2\},$$
and letting $\eta\geq 1, r>c_0\eta,$
we further have 
$c_1\tfrac{(L+\eta)^2}{r^2}< \tfrac{c_1}{c_0^2}\tfrac{(L+\eta)^2}{\eta^2}=\tfrac{c_1}{c_0^2}{(L/\eta+1)^2}\leq \tfrac{c_1}{c_0^2}{(L+1)^2} $ (the first inequality is based on $r>c_0\eta$, and the second is based on $\eta\geq 1$). Therefore, one just needs to guarantee the following inequality $\tfrac{\eta-3L}{2}-\tfrac{L^2}{\eta}-\tfrac{c_1}{c_0^2}{(L+1)^2}>0,$ with choosing proper $\eta$.
By letting $\eta>6L+2L^2+\tfrac{2c_1}{c_0^2}{(L+1)^2},$  the above inequality holds. Finally setting
$(r,\eta)$ in the nonempty set $\big\{(r,\eta):~~ \tfrac{r}{c_0}>\eta> \max\{6L+2L^2+\tfrac{2c_1}{c_0^2}{(L+1)^2}, 1\} \big\}$, it also belongs to the set $\mathcal K$, that immediately implies  it is nonempty.
Roughly to say, with big enough parameters $\eta$ and $r$ ($r$ is no less than $\eta$ with a constant factor), the above lemma holds.

\begin{proof}
The first two items (Items (1)-(2))  can be easily proved by Lemmas \ref{lem:non-increasing}-\ref{lem:lowerBound}, \eqref{eq:Stationary1}, and \eqref{eq:Stationary2}. Notice that the last inequality of \eqref{eq:paramsSet} is required in order to guarantee the positivity of the coefficients for the terms related with  $\|z_d^n-z_d^{n+1}\|, d=1,2.$

One can readily get the boundedness of $\{z_d^n\}$ by \eqref{eq:LB}. By the first equation of \eqref{eq:Stationary1} and continuity of $\nabla_{z_d}\mathcal G_\epsilon(z_d;f_d)$, the sequences $\{\Gamma_d^n\}$ are also bounded. Since $\{A_d u_d^n-z_d^n+\Gamma_d^n\}$ are bounded by \eqref{eq:LB} and Assumption \ref{assump1},  one can further conclude the boundedness of $\{u_d^n\}.$ By \eqref{eq:Stationary2}, one can get the boundedness of $\{\Lambda_{1,2}^n\}$ and $\{\Lambda_{2,1}^n\}$, as well as the boundedness of $\{v^n\}$. That finishes the proof of Item (3).
\end{proof}

\begin{thm}[subsequence convergence]\label{thm1}
Letting $(r, \eta)\in \mathcal K$, any limit point of $\{X^n\}$ is the stationary point of the saddle point problem \eqref{eq:SaddlePoint}  under Assumption \ref{assump1}.
\end{thm}
\begin{proof}
By Lemma \ref{lem:bound} and \eqref{eq:Stationary1}-\eqref{eq:Stationary2}, one can readily prove that any limit point $X^\star:=(u^{\star}_1, u^{\star}_2, v^\star, z_1^{\star},$ $ z_2^{\star},\Lambda^{\star}_{1,2},\Lambda^{\star}_{2,1}, \Gamma^{\star}_1, \Gamma_2^{\star})$ of the iterative sequence satisfies the following equations
\begin{equation}\label{eq:StationarySaddle-0}
\begin{split}
&0=\nabla_{z_d}\mathcal G_\epsilon(z_d^\star;f_d)- \eta \Gamma_d^\star;\\
&0=v^{\star}-\tfrac{1}{2}\sum\nolimits_{d=1}^2({\mathbf  \pi}_{d,3-d}u^{\star}_d+\Lambda^{\star}_{d,3-d});\\
&0=(r{\mathbf  \pi}_{d,3-d}^T{\mathbf  \pi}_{d,3-d}+\eta A_d^*A_d) u_d^{\star}-\big(r{\mathbf  \pi}_{d,3-d}^T(v^{\star}-\Lambda^{\star}_{d,3-d})+\eta A_d^*(z^{\star}_d-\Gamma^{\star}_d)\big), d=1,2; \\
&0=A_d u_d^\star-z_d^\star;\\ &0=\pi_{d,3-d}u_d^\star-v^\star, d=1,2;\\
\end{split}
\end{equation}
By the last two equations of the above relation, it  can be simplified below:
\begin{equation}\label{eq:StationarySaddle-1}
\begin{split}
&0=\nabla_{z_d}\mathcal G_\epsilon(z_d^\star;f_d)- \eta \Gamma_d^\star,\ d=1,2;\\
&0=\Lambda^{\star}_{1,2}+\Lambda^{\star}_{2,1};\\
&0=r{\mathbf  \pi}_{d,3-d}^T\Lambda^{\star}_{d,3-d}+\eta A_d^*\Gamma^{\star}_d,  d=1,2; \\
&0=A_d u_d^\star-z_d^\star,\ d=1,2;\\
&0=\pi_{d,3-d}u_d^\star-v^\star, d=1,2,\\
\end{split}
\end{equation}
that is exactly the stationary point of \eqref{eq:SaddlePoint}, and therefore it finishes the proof.
\end{proof}

\begin{thm}[global convergence]
Letting $(r, \eta)\in \mathcal K$, the iterative sequence $\{X^n\}$ converges to the stationary point of the saddle point problem \eqref{eq:SaddlePoint} under Assumption \ref{assump1}.
\label{thm2}
\end{thm}
\begin{proof}
It is quite standard to prove that the gradient of the augmented Lagrangian is bounded by the $\|X^n-X^{n+1}\|$,  which simply follows Lemma 4.3 of \cite{chang2018Blind}. One just needs to show that the augmented Lagrangian with the ST-AGM metric is semi-algebraic function. First, the graph of the $g_\epsilon(z_d(j),f_d(j))$ is semi-algebraic (We study this property by separating the real and imaginary parts of the complex-valued variables following \cite{chang2018Blind}), since it can be expressed by
{ 
\[
\begin{split}
&\Big(~\{(z_{d,r}(j),z_{d,i}(j),F(j))\in\mathbb R^3:\ F(j)-\tfrac{1-\epsilon}{2}(f_d(j)-\tfrac{1}{\epsilon}((z_{d,r}(j))^2+(z_{d,i}(j))^2))=0\}\\
&~\cap\{(z_{d,r}(j),z_{d,i}(j),F(j))\in\mathbb R^3:\ (z_{d,r}(j))^2+(z_{d,i}(j))^2-\epsilon^2 f_d(j)<0\}\Big)\ \\
\bigcup &
\Big(~
\{(z_{d,r}(j),z_{d,i}(j),F(j))\in\mathbb R^3:\
\tfrac{1}{2}((z_{d,r}(j))^2+(z_{d,i}(j))^2+f_d(j))-F(j)\geq 0
\}\\
&\!\!\!\!\!\!\!\!\cap
\{(z_{d,r}(j),z_{d,i}(j),F(j))\in\mathbb R^3:\!\!
(\tfrac{1}{2}((z_{d,r}(j))^2+(z_{d,i}(j))^2\!+\!f_d(j))\!-\!F(j))^2-f_d(j)((z_{d,r}(j))^2+(z_{d,i}(j))^2)\!=\!0
\}\\
&\cap
\{(z_{d,r}(j),z_{d,i}(j),F(j))\in\mathbb R^3:\ (z_{d,r}(j))^2+(z_{d,i}(j))^2-\epsilon^2 f_d(j)\geq 0\}
\Big),
\end{split}
\]
}
with $z_d(j):=z_{d,r}(j)+\bm i z_{d,i}(j)$ and $\bm i^2=-1.$
Immediately  one sees that the function $D_j(z_d; f_d):=g_\epsilon(z_d(j);$ $ f_d(j))$ is semi-algebraic. Then ST-AGM $\mathcal G_\epsilon(z_d;f_d)=\sum\nolimits_j D_j(z_d;f_d)$ is also semi-algebraic since the finite sum of the algebraic functions are also semi-algebraic \cite{bolte2014proximal}. Since other terms other than ST-AGM are also semi-algebraic,    the augmented Lagrangian function $\mathcal L(X)$ is semi-algebraic. Thus one concludes this theorem based on Theorem \ref{thm1}, the gradient of augmented Lagrangian bounded by $\|X^n-X^{n+1}\|$ and semi-algebraic property following \cite{bolte2014proximal}.
\end{proof}

We remark that Theorem \ref{thm1} demonstrates that limit point of iterative sequence  satisfies the constraint $\pi_{1,2} u_1=\pi_{2,1}u_2.$ Denote 
\[
\forall 0\leq j\leq n-1,\ \ \ 
u^\star(j):=\left\{
\begin{split}
&(R_1^T u_1^\star)(j),\text{~~if~}j\in \Omega_1\setminus\Omega_{1,2};    \\
& (R_{1,2}^T v^\star)(j),\text{~if~}j\in \Omega_{1,2};\\
& (R_2^T u_2^\star)(j),\text{\ \ \ \  otherwise};    \\
\end{split}
\right.
\]
 $z^\star:=A u^\star,$ and $\Gamma^\star:=(\Gamma_1^T,\Gamma_2^T)^T.$
Further with the given stationary point system \eqref{eq:StationarySaddle-1}, we have
\[
\left\{
\begin{split}
&0=\nabla_{z}\mathcal G_\epsilon(z^\star; f)- \eta \Gamma^\star,\\
&0=Au^\star-z^\star,\\
&0=A^*\Gamma^\star,\\
\end{split}
\right.
\]
Then we get 
$
A^*\nabla_{z}\mathcal G_\epsilon(Au^\star;f)=0,
$
that is exactly the first-order optimality condition for the optimization problem on the whole domain as 
$
\min_u \mathcal G_\epsilon (Au;f).
$
Hence, by merging these two sub-solutions, Theorem \ref{thm2} actually demonstrates the proposed algorithm is able to produce the  stationary point for the original optimization problem defined on the whole domain.

\section{Extensions}
\label{sec3}
In this part, we  mainly discuss further extensions for multiple-subdomain DD (more than two) and blind recovery.
\subsection{Multiple-subdomain DD}
We consider the DD with $D$ subdomains ($D\geq 2$).  Denote the  overlapping DD as $\Omega:=\bigcup_{d=1}^{D} \Omega_d$, and there exists a non-empty index set 
$\widetilde{\mathscr  N}:=\{(i_1,i_2):\ \Omega_{i_1}\bigcap\Omega_{i_2}\not=\emptyset ~\forall 1\leq i_1, i_2\leq D \}$ denoting the index pairs for overlapped  subdomains, and 
$\mathscr N:=\{(i_1,i_2)\in\widetilde{\mathscr N}:~~i_1<i_2\}$.  In a similar manner to the two-subdomain case in the last section, one can denote $A_d, f_d(1\leq d\leq D)$ and the restriction operator $\pi_{i_1,i_2}\ \forall (i_1,i_2)\in\mathscr N$  from the subdomain $\Omega_{i_1}$ to the overlapping region $\Omega_{i_1, i_2}$ between  the subdomains $\Omega_{i_1}$ and $\Omega_{i_2}$,  such that to get the solution $u_d$ on the subdomain $\Omega_d \forall 1\leq d\leq D$,
one has to solve the following equations:
\begin{equation}
|A_d u_d|^2={f_d}\ \forall 1\leq d\leq D;\ \ \\
\pi_{i_1, i_2} u_{i_1}=\pi_{i_2, i_1} u_{i_2}\ \forall (i_1,i_2)\in \mathscr N,
\end{equation}
that exactly extends \eqref{eq:forward} and \eqref{eq:continuity} to multiple subdomains DD. 
The following constraint optimization problem is given by introducing auxiliary variables $\{z_d\}_{d=1}^D$ and $\{v_{i_1,i_2}\}_{(i_1,i_2)\in\mathscr N}$ 
\begin{equation*}
\begin{split}
    &\min\nolimits_{\{u_d\}_{d=1}^{D}, \{v_{i_1,i_2}\}_{(i_1,i_2)\in\mathscr N} }\ \ \  \sum\nolimits_{d=1}^{D}
    \mathcal G_\epsilon(z_d;f_d)\ \\
s.t.& \ \mathcal {\mathbf  \pi}_{i_1,i_2}u_{i_1}-v_{i_1,i_2}=0, \mathcal {\mathbf  \pi}_{i_2,i_1}u_{i_2}-v_{i_1,i_2}=0 \ \ (i_1,i_2)\in \mathscr N,\\
    & \   \ z_d-A_d u_d=0,\ 1\leq d\leq D.
\end{split}
\end{equation*}
Similar to the two-subdomain case, the corresponding augmented Lagrangian is given below:
\begin{equation}\label{eq:augLag-multi-Domain}
\begin{split}
&\mathcal L_m(\{u_d\},\{v_{i_1,i_2}\}_{(i_1,i_2)\in\mathscr N}, \{z_d\},\{\Lambda_{i_1,i_2}\}_{(i_1,i_2)\in\widetilde{\mathscr N}}, \{\Gamma_d\}):=\sum\nolimits_{d=1}^{D}\mathcal G_\epsilon(z_d;f_d)\ \\
&\ \ \ \ + \sum\nolimits_{(i_1,i_2)\in\mathscr N}\Big(r\Re(\langle \Lambda_{i_1,i_2}, \mathcal {\mathbf  \pi}_{i_1,i_2}u_{i_1}-v_{i_1,i_2}\rangle)+\tfrac{r}{2}\|\mathcal {\mathbf  \pi}_{i_1,i_2}u_{i_1}-v_{i_1,i_2}\|^2\Big)\\
&\ \ \ \ + \sum\nolimits_{(i_1,i_2)\in\mathscr N}\Big(r\Re(\langle \Lambda_{i_2,i_1}, \mathcal {\mathbf  \pi}_{i_2,i_1}u_{i_2}-v_{i_1,i_2}\rangle)+\tfrac{r}{2}\|\mathcal {\mathbf  \pi}_{i_2,i_1}u_{i_2}-v_{i_1,i_2}\|^2\Big)\\
&\ \  \  +\eta\sum\nolimits_{d=1}^{D}\left(  \Re(\langle \Gamma_d, \mathcal A_d u_d-z_d \rangle) +\tfrac{1}{2}\|A_d u_d-z_d \|^2\right),
\end{split}
\end{equation}
with multipliers $\{\Lambda_{i_1,i_2}\}_{(i_1,i_2)\in\widetilde{\mathscr N}}$ and $\{\Gamma_d\}_{d=1}^D$.

We only analyze the $u_{d}-$subproblem which is different from the two-subdomain case in Algorithm 1. One has
 \begin{equation*}
 \min\nolimits_{u_{d}} \Big(\sum\nolimits_{(d,i_2)\in\mathscr N}\tfrac{r}{2}\|\Lambda_{d,i_2}+\mathcal {\mathbf  \pi}_{d,i_2}u_{d}-v_{d,i_2}\|^2+\sum\nolimits_{(i_1,d)\in\mathscr N }\tfrac{r}{2}\|\Lambda_{d,i_1}+\mathcal {\mathbf  \pi}_{d,i_1}u_{d}-v_{i_1,d}\|^2\Big)+\tfrac{\eta}{2}\|A_d u_d-z_d+\Gamma_d\|^2,
\end{equation*}
such that the first order optimality condition is given as
\begin{equation*}
\begin{split}
&\Big(\eta A_d^*A_d+r\sum\nolimits_{ (d,i_2), (i_1,d)\in\mathscr N }(\pi_{d,i_1}^T\pi_{d,i_1}+\pi_{d,i_2}^T\pi_{d,i_2})\Big)u_d\\
=&\eta A_d^*(z_d-\Gamma_d)
+r\sum\nolimits_{(d,i_2)\in\mathscr N}\pi_{d,i_2}^T(v_{d,i_2}-\Lambda_{d,i_2})+r\sum\nolimits_{(i_1,d)\in\mathscr N, }\pi_{d,i_1}^T(v_{i_1,d}-\Lambda_{d,i_1}).
\end{split}
\end{equation*}
Similar to the case of two-subdomain, this subproblem has closed-form solution as \eqref{eq:Alg1-solver-u}, such that it can also be  very efficiently solved.

Similar to the two subdomain decomposition, the overall overlapping DD based ptychography  algorithm for  multiple subdomains (OD${}^2$P${_m}$) is given as below:
\begin{center}
\rule{\textwidth}{1mm}
\centering{Algorithm 2: {\bf O}verlapping {\bf DD} based {\bf P}tychography  algorithm for {\bf M}ultiple Subdomains (OD${}^2$P${_m}$)}
\begin{minipage}{1\textwidth}
\rule{\textwidth}{.8mm}
\begin{itemize}
\item[Step 0.] Initialize $u_d^0=\bm 1, \{v_{i_1,i_2}^0\}$, $z_d^0=A_d u_d^0,$ and multipliers $\Gamma_d^0=\bm 0, \Lambda^0_{i_1,i_2}=\bm 0.$  Set $n:=0.$

\item[Step 1.] Update $\{z^{n+1}_d\}_{d=1}^{D}$  by \eqref{eq:updateZ}
and  update $\{v_{i_1,i_2}^{n+1}\}_{(i_1,i_2)\in\mathscr N}$ by
\begin{equation}\label{eq:updateV-multi}
v_{i_1,i_2}^{n+1}=\tfrac{1}{2}({{\mathbf  \pi}_{i_1,i_2}u^{n}_{i_1}+{\mathbf  \pi}_{i_2,i_1}u^{n}_{i_2}+\Lambda^n_{i_1,i_2}+\Lambda^n_{i_2,i_1}});
\end{equation}

\item[Step 2.] Update multipliers in parallel by
\begin{equation}\label{eq:updateM-1-multi}
\begin{split}
&\Gamma_d^{n+1}=\Gamma^n_d+ A_du^{n}_d-z^{n+1}_d\ \forall \ 1\leq d\leq D; \\
\end{split}
\end{equation}

\item[Step 3.] Update $\{u^{n+1}_d\}_{d=1}^{D}$  in parallel by
\begin{equation*}
\begin{split}
&u_d^{n+1}=\Big(\eta A_d^*A_d+r\sum\nolimits_{ (d,i_2), (i_1,d)\in\mathscr N }(\pi_{d,i_1}^T\pi_{d,i_1}+\pi_{d,i_2}^T\pi_{d,i_2})\Big)^{-1}\\
\times&\Big(\eta A_d^*(z^{n+1}_d-\Gamma^{n+1}_d)
+r\sum\nolimits_{(d,i_2)\in\mathscr N}\pi_{d,i_2}^T(v^{n+1}_{d,i_2}-\Lambda^n_{d,i_2})+r\sum\nolimits_{(i_1,d)\in\mathscr N, }\pi_{d,i_1}^T(v^{n+1}_{i_1,d}-\Lambda^n_{d,i_1})\Big);
\end{split}
\end{equation*}

\item[Step 4.] Update multipliers  by
\begin{equation}\label{eq:updateM-2-multi}
\begin{split}
& \Lambda_{i_1,i_2}^{n+1}= \Lambda^n_{i_1,i_2}+{\mathbf  \pi}_{i_1,i_2}u^{n+1}_{i_1}-v_{i_1,i_2}^{n+1},\\
& \Lambda_{i_2,i_1}^{n+1}= \Lambda^n_{i_2,i_1}+{\mathbf  \pi}_{i_2,i_1}u^{n+1}_{i_2}-v_{i_2,i_1}^{n+1},\\
\end{split}
\end{equation}
for all $(i_1,i_2)\in\mathscr N.$
\item[Step 5.] If satisfying the stopping condition, then stop and output $u_d^{n+1}$ as the final solution; otherwise set $n:=n+1$, and go to Step 1.
\end{itemize}
\end{minipage}
\vskip .1in
\rule{\textwidth}{1mm}
\end{center}

\vskip .1in

	Among all the variables update  in Algorithm 2,  most variables (except $v_{i_1,i_2}^{n+1}$ and the multipliers $\Lambda_{i_1,i_2}^{n+1}$) can be computed independently of other variables defined on another subdomain. The information of  overlapping regions $v_{i_1,i_2}^{n+1}$ is synchronized by the overlapping parts of two sub-solutions $u_{i_1}^n, u_{i_2}^n$, as well as the multiplier $\Lambda_{i_1,i_2}$. Hence, in principle the proposed OD${}^2$P$_m$ can be implemented highly parallel, since the communications only happen at the  boundary layers (much smaller sizes compared with the entire image).

For the convergence analysis, it is almost the same as the case of two subdomains and therefore we omit the details and give the convergence theorem directly as below:
\begin{thm}
Letting $(r, \eta)\in \mathcal K_{D}$  with
\begin{equation}\label{eq:paramsSet_multi}
\begin{split}
&\mathcal K_{D}:=\Big\{(r, \eta)\in \mathbb R^2_+:\ \ \eta\geq 1, r-2\eta\max\{\|A_d^*A_d\|^2\}_{d=1}^{ D}>0,\\
&\ \ \qquad\qquad\qquad \tfrac{\eta-3L}{2}-\tfrac{L^2}{\eta}-\tfrac{2(L+\eta)^2}{r^2} \max\nolimits_{(i_1,i_2)\in\mathscr N}\max\{\|\pi_{i_1,i_2}A_{i_1}^*\|^2,  \|\pi_{i_2,i_1}A_{i_2}^*\|^2\}>0 \Big\},
\end{split}
\end{equation}
then under Assumption \ref{assump1} for $D-$subdomains, the iterative sequence  $\widetilde X^n:=\big(\{u^{n}_d\},\{v_{i_1,i_2}^n\}_{(i_1,i_2)\in\mathscr N}, \{z_d^{n}\},$ $\{\Lambda^{n}_{i_1,i_2}\}_{(i_1,i_2)\in\widetilde{\mathscr N}}, \{\Gamma^{n}_d\}_{d=1}^{D}\big)$ generated by OD${}^2$P${}_m$
converges to the stationary point of the saddle point problem of the augmented Lagrangian defined in  \eqref{eq:augLag-multi-Domain}.
\end{thm}

{Similar to the 2-subdomain case, the set \eqref{eq:paramsSet_multi} of the parameters $r,\eta$ is nonempty}.

\subsection{Blind ptychography }

To reduce the grid pathology \cite{chang2018Blind} (ambiguity derived by the multiplication of any periodical function and the true solution) due to grid scan, introduce the support set constraint of the probe, i.e. $\mathcal O:=\{w:\ (\mathcal Fw)(j)=0,\ j\in \mathscr J\},$ {  with the support set $\bar {\mathscr J}$ denoted as the complement of the set $\mathscr J$ (index set for zero values for the Fourier transform of  the probe). } We will explain why it would help to remove the grid pathology intuitively. Assume that one gets a solution multiplied by a non-constant vector $e$, such that
one gets
$
\mathcal F(w\circ S_j (u\circ e))=(\mathcal F(\mathcal S_j e)\ast\mathcal Fw)\ast
\mathcal F(\mathcal S_j u).
$
Obviously, it demonstrates
that
$\mathcal F(\mathcal S_j e)\ast\mathcal Fw$ will enlarge the support set, that may disobey the compact support constraint.

For simplicity, we consider the blind ptychography problem for two-subdomain DD:
\begin{equation*}
\begin{split}
    &\min\nolimits_{\{w, u_1, u_2, v\}}\sum\nolimits_{d=1}^2\mathcal G_\epsilon(B_d(w, u_d);f_d)+\mathbb I_{\mathcal O}(w),\ \ s.t. \  {\pi}_{d,3-d}u_d-v=0, d=1,2, \ \
\end{split}
\end{equation*}
where the bilinear mapping
$B_d(w, u_d)$ is denoted as
$(B_d)_{j_d}(w,u_d):=\mathcal F (w\circ \big(S^d_{j_d}u_d))\ \forall\ 0\leq j_d\leq J_{d}-1$(
$\sum_{d=1}^{2} J_{d}=J$), and the indicator function $\mathbb I_{\mathcal O}$ is defined as
\[
\mathbb I_{\mathcal O}(w):=\left\{
\begin{split}
&0,\qquad\text{if~} w\in \mathcal O;\\
&+\infty, \text{~otherwise.}
\end{split}
\right.
\]
To enable parallel computing, we consider the following constraint optimization problems
\begin{equation*}
\begin{split}
    &\min\nolimits_{\{w, w_1, w_2, u_1, u_2, v, z_1, z_2\}}\sum_{d=1}^2\mathcal G_\epsilon(z_d;f_d)+\mathbb I_{\mathcal O}(w)\ \\
s.t.& \ \mathcal {\pi}_{1,2}u_1-v=0,\ \  \mathcal { \pi}_{2,1}u_2-v=0, \ w_d=w, \ z_d=B_d(w_d, u_d), d=1,2.  \ \
\end{split}
\end{equation*}

{To get rid of heavy notations, we introduce the following definitions.}
By denoting
\begin{equation*}
\bm \pi=
\begin{pmatrix}
\pi_{1,2}&\bm 0\\
0&\pi_{2,1}
\end{pmatrix},
\ \mathbf I_0=
\begin{pmatrix}
I_0\\
I_0
\end{pmatrix},
\
\mathbf I_1=
\begin{pmatrix}
I_1\\
I_1\\
\end{pmatrix},
\
\mathbf B(\mathbf W, \mathbf U)=
\begin{pmatrix}
B_1(\omega_1, u_1)\\
B_2(\omega_2, u_2)
\end{pmatrix}
\end{equation*}
with identity matrices  $I_0, I_1$,
 we rewrite the constrained condition below:
\begin{equation*}
\begin{split}
&{\bm\pi} \mathbf U-\mathbf I_0  v=0,\ \  {\bf W}-{\mathbf I_1} w=0,\ \  \mathbf B(\mathbf W, \mathbf U)-\mathbf Z=0,
\end{split}
\end{equation*}
where $\mathbf U=(u_1^T, u_2^T)^T, \mathbf W=(w_1^T, w_2^T)^T,$ and $\mathbf Z=(z_1^T, z_2^T)^T$.
Hence we consider the following model:
\begin{equation*}
\begin{split}
    &\min_{w, \mathbf W,\mathbf U, v, \mathbf Z}\mathcal {G}_\epsilon(\mathbf Z; \mathbf f)+\mathbb I_{\mathcal O}(w)\ \\
\text{s.t.}&\ \  {\bm\pi} \mathbf U-\mathbf I_0  v=0,\ \ {\bf W}-{\mathbf I_1} w=0,\ \  \mathbf B(\mathbf W, \mathbf U)-\mathbf Z=0,  \ \
\end{split}
\end{equation*}
with $\mathbf f:=(f^T_1, f^T_2)^T.$

Similarly, the augmented Lagrangian can be formulated below:
\begin{equation}\label{eq:augmentedLagrangianBlind}
\begin{split}
    &\mathcal L_b(w,\mathbf W, \mathbf U, v, \mathbf Z, \mathbf \Lambda, \mathbf \Gamma, \mathbf \Delta)=\mathcal { G}_\epsilon(\mathbf Z; \mathbf f)+\mathbb I_{\mathcal O}(w)\\
    &\ \  +r\Re(\langle{\bm\pi} \mathbf U-\mathbf I_0  v, \mathbf \Lambda\rangle)+ \eta\Re(\langle \mathbf B(\mathbf W, \mathbf U)-\mathbf Z, \mathbf \Gamma\rangle)+\mu\Re(\langle {\bf W}-{\mathbf I_1} w, \mathbf\Delta\rangle)\\
    &\ \   +\tfrac{r}{2}\|{\bm\pi} \mathbf U-\mathbf I_0  v\|^2+ \tfrac{\eta}{2}\|\mathbf B(\mathbf W, \mathbf U)-\mathbf Z\|^2+\tfrac{\mu}{2}\|{\bf W}-{\mathbf I_1} w\|^2,
\end{split}
\end{equation}
with multipliers $\mathbf \Lambda, \mathbf \Gamma, \mathbf \Delta$ and positive parameters $r, \eta, \mu.$
Then alternatively one needs to solve the following saddle point problem below
\begin{equation}\label{eq:saddlepointBlind}
\max_{ \mathbf \Lambda, \mathbf \Gamma, \mathbf \Delta}\ \ \min_{w,\mathbf W, \mathbf U, v, \mathbf Z}\mathcal L_b(w,\mathbf W, \mathbf U, v, \mathbf Z, \mathbf \Lambda, \mathbf \Gamma, \mathbf \Delta).
\end{equation}

One can solve the above saddle point problem step by step  using  a variant ADMM with additional proximal term for the $\mathbf U-$subproblem below:
\begin{align}
(w^{n+1},v^{n+1},\mathbf Z^{n+1})=&\arg\min_{w, v, \mathbf Z}     \mathcal L_b(w, \mathbf W^n, \mathbf U^n, v, \mathbf Z, \mathbf \Lambda^n, \mathbf \Gamma^n, \mathbf \Delta^n)\\
\bm \Gamma^{n+1}=&\bm \Gamma^{n}+\mathbf B(\mathbf W^n, \mathbf U^n)-\mathbf Z^{n+1}; \label{eq:admm_Gamma}\\
\mathbf W^{n+1}=&\arg\min_{\mathbf W}     \mathcal L_b( w^{n+1}, \mathbf W, \mathbf U^n, v^{n+1}, \mathbf Z^{n+1}, \mathbf \Lambda^n, \mathbf \Gamma^{n+1}, \mathbf \Delta^{n})\\
\mathbf U^{n+1}=&\arg\min_{\mathbf U}     \mathcal L_b(w^{n+1}, \mathbf W^{n+1},  \mathbf U, v^{n+1}, \mathbf Z^{n+1}, \mathbf \Lambda^n, \mathbf \Gamma^{n+1}, \mathbf \Delta^{n})+\tfrac{\gamma}{2}\|\mathbf U-\mathbf U^n\|^2\\
\mathbf \Delta^{n+1}=&\mathbf \Delta^n+ {\mathbf W^{n+1}}-{\mathbf I_1} w^{n+1};\label{eq:admm_Delta}\\
\mathbf \Lambda^{n+1}=&\mathbf \Lambda^n+{\bm\pi} \mathbf U^{n+1}-\mathbf I_0  v^{n+1}.\label{eq:admm_Lambda}
\end{align}

For the $w-$subproblem,
$
w^{n+1}=\arg\min_w \mathbb I_{\mathcal O}(w)+\tfrac{\mu}{2}\|\mathbf\Delta^n+ \mathbf W^n-\mathbf I_1 w\|^2.
$
Readily one has
\begin{equation}\label{eq:blind_w}
w^{n+1}=\tfrac12\mathcal F^* M \mathcal F\mathbf I_1^T(\mathbf \Delta^n+\mathbf W^n),\\
\end{equation}
where the  matrix $M$ is diagonal with diagonal elements defined as
\begin{equation*}
M(j,j)=\left\{
\begin{split}
&1,\  \ \text{if\ } j\in \mathscr J;\\
&0, \ \ \text{otherwise}.
\end{split}
\right.
\end{equation*}
For the $v-$subproblem,
\begin{equation}\label{eq:blind_v}
v^{n+1}=\arg\min_v  \tfrac{r}{2}\|\mathbf\Lambda^n+\bm \pi \mathbf U^n-{\mathbf I_0} v\|^2=\tfrac{1}{2}\mathbf I_0^T(\mathbf \Lambda^n+\bm \pi \mathbf U^n).
\end{equation}
For the $z-$subproblem,
\begin{equation}\label{eq:blind_Z}
\begin{split}
\mathbf Z^{n+1}&=\arg\min_{\mathbf Z}    \mathcal {  G}_\epsilon(\mathbf Z; \mathbf f)+\tfrac{\eta}{2}\|\mathbf\Gamma^n+\mathbf B(\mathbf W^n, \mathbf U^n)-\mathbf Z\|^2\\
&=\mathrm{Prox}_{\mathcal {  G}_\epsilon, \eta}(\mathbf\Gamma^n+\mathbf B(\mathbf W^n, \mathbf U^n)),
\end{split}
\end{equation}
where $\mathrm{Prox}_{\mathcal {G}_\epsilon, \eta}(\mathbf Z)=(\mathrm{Prox}_{ {\mathcal  G}_\epsilon, \eta}^T(z_1), \mathrm{Prox}_{\mathcal {G}_\epsilon, \eta}^T(z_2))^T\ \ \forall\ \mathbf Z=(z_1^T, z_2^T)^T$.

Define two related  operators by fixing one variable for the bilinear operator ${\mathbf B(\cdot, \cdot)}$ as
$
\mathbf D_{\mathbf W}(\mathbf U):=\mathbf B(\mathbf W,\mathbf U),$ 
$
\mathbf D_{\mathbf U}(\mathbf W):=\mathbf B(\mathbf W,\mathbf U) \ \ \forall \ \mathbf W, \mathbf U.
$
Then  $\mathbf D_{\mathbf W}$ and $\mathbf D_{\mathbf U}$ are  linear, and moreover, they are linear (also bounded) w.r.t. the subscripts $\mathbf W, \mathbf U$ respectively.

For $\mathbf W-$subproblem, one has
\begin{equation}\label{eq:blind_W}
\begin{split}
\mathbf W^{n+1}&=\arg\min_{\mathbf W}   \tfrac{\eta}{2}\|\mathbf B(\mathbf W,\mathbf U^n)-\mathbf Z^{n+1}+\mathbf \Gamma^{n+1}\|^2+\tfrac{\mu}{2}\| \mathbf W-\mathbf I_1 w^{n+1}+\mathbf \Delta^{n}\|^2 \\
&=\big(\eta\mathbf D_{\mathbf U^{n}}^*\mathbf D_{\mathbf U^{n}}+\mu \mathbf I_1\big)^{-1}
\big(\eta\mathbf D_{\mathbf U^{n}}^*(\mathbf Z^{n+1}-\mathbf \Gamma^{n+1})+\mu (\mathbf I_1 w^{n+1}-\mathbf \Delta^n)\big),
\end{split}
\end{equation}
with  $\mathbf D_{\mathbf U^{n}}^*\mathbf D_{\mathbf U^{n}}=\mathrm{diag}\Big(\big((\sum\nolimits_{j_1=0}^{J_1-1} \mathcal S_{j_1}^1 |u_1^{n}|^2)^T, (\sum\nolimits_{j_2=0}^{J_2-1} \mathcal S_{j_2}^2 |u_2^{n}|^2)^T\big)^T\Big)$.

For $\mathbf U-$subproblem, one has to solve the following problem
\begin{equation}\label{eq:blind_U}
\begin{split}
\mathbf U^{n+1}&=\arg\min_{\mathbf U} \tfrac{\eta}{2}\|\mathbf B(\mathbf W^{n+1},\mathbf U)-\mathbf Z^{n+1}+\mathbf \Gamma^{n+1}\|^2+\tfrac{r}{2}\|\bm \pi \mathbf U-\mathbf I_0 v^{n+1}+\mathbf \Lambda^n\|^2+\tfrac{\gamma}{2}\|\mathbf U-\mathbf U^n\|^2\\
&=\big(\eta \mathbf D_{\mathbf W^{n+1}}^*\mathbf D_{\mathbf W^{n+1}}+r\bm\pi^T\bm\pi+\gamma \mathbf I\big)^{-1}
\big(\eta\mathbf D_{\mathbf W^{n+1}}^*(\mathbf Z^{n+1}-\mathbf \Gamma^{n+1})+r\bm\pi^T(\mathbf I_0v^{n+1}-\mathbf \Lambda^n) +\gamma \mathbf U^n \big),
\end{split}
\end{equation}
where 
$\mathbf D_{\mathbf W^{n+1}}^*\mathbf D_{\mathbf W^{n+1}}=\mathrm{diag}\Big(\big((\sum\nolimits_{j_1=0}^{J_1-1} (\mathcal S_{j_1}^1)^T |w_1^{n+1}|^2)^T, (\sum\nolimits_{j_2=0}^{J_2-1} (\mathcal S_{j_2}^2)^T |w_2^{n+1}|^2)^T\big)^T\Big)$
and $\mathbf I$ denotes the identity matrix, i.e. $\mathbf I \mathbf U=\mathbf U$.

Finally the {\bf o}verlapping {\bf DD} based {\bf b}lind {\bf p}tychography   (OD${}^2$BP)  is given as below:
\vskip .2in
\begin{center}
\begin{minipage}{.95\textwidth}
\rule{\textwidth}{1mm}
\vskip .1in
\centering{Algorithm 3: \ {\bf O}verlapping {\bf DD} based {\bf B}lind {\bf P}tychography   (OD${}^2$BP)}
\vskip .1in
\rule{\textwidth}{.8mm}
\begin{itemize}
\item[Step 0.] Set $n:=0.$
\item[Step 1.] Update $w^{n+1}, v^{n+1}, \mathbf Z^{n+1}$ in parallel by \eqref{eq:blind_w}, \eqref{eq:blind_v}, \eqref{eq:blind_Z}.
\item[Step 2.] Update the multipliers $\mathbf \Gamma^{n+1}$   by \eqref{eq:admm_Gamma}.
\item[Step 3.] Update $\mathbf W^{n+1}$ by \eqref{eq:blind_W}.
\item[Step 4.] Update $\mathbf U^{n+1}$ by \eqref{eq:blind_U}.
\item[Step 5.] Update  $\mathbf \Delta^{n+1}$ and $\mathbf \Lambda^{n+1}$    by  \eqref{eq:admm_Delta} and \eqref{eq:admm_Lambda}.
\item[Step 6.] If satisfying the stopping condition, then stop and output $\mathbf U^{n+1}$ as the final output; otherwise, set $n:=n+1,$ and go to Step 1.
\end{itemize}
\end{minipage}
\vskip .1in
\rule{.95\textwidth}{1mm}
\end{center}
For the convergence analysis,  the blind case becomes more sophisticated. Limited to the current technique, we have to make the  assumption of the boundedness for the iterative sequence. Then with proper $(r,\eta,\mu, \gamma)$,  similarly to the nonblind case, one can prove that any limit point of $\{\widehat X^n\}$ converges to the stationary point of the saddle point problem \eqref{eq:saddlepointBlind}.

\section{Numerical Experiments}
\label{sec4}
To get perfect reconstruction results, periodical boundary condition is usually enforced as \cite{chang2018Blind},
that produces almost same redundancy of each pixel of the sample. To make the simulation more realistic, instead
we assume that the image to be recovered has boundary with pixel values being all ones, because the sample during the experiments are usually  in the vacuum and the region without sample occupying has constant contrast. In the implementation of all algorithms, we enforce the region to be ones.

We consider the zone-plate probe with $64\times 64$ pixels (See Fig. \ref{fig2} (c)-(d)), and the complex-valued sample is generated
by taking two real-valued images as its absolute parts (See Fig. \ref{fig2} (a) ) and phase (See Fig. \ref{fig2} (b) ) parts respectively. The raster grid is adopted for ptychography scan, i.e. all scanning points lie in isometric square grid.

For noisy free measurements, the proposed algorithms stop at the $n^{th}$ iteration, if  the R-factors $\mathrm{RF}^n:=\tfrac{\sum_d\||A_d u^n_d|-\sqrt{f_d}\|_1}{\sum_d\|\sqrt{f_d}\|_1}\leq 1.0\times 10^{-5}$ or $\tfrac{\sum_d\||B_d (w^n,u^n_d)|-\sqrt{f_d}\|_1}{\sum_d\|\sqrt{f_d}\|_1}\leq 1.0\times 10^{-5}$ for  non-blind and blind recovery, respectively,  or the maximum iteration number reaches 1000.  For the noisy measurements, the proposed algorithms stop at $n^{th}$ iteration,  if  the relative error $\mathrm{RE}^n:=\max_d\tfrac{\|u_d^{n}-u_d^{n-1}\|}{\|u_d^n\|}\leq 1.0\times 10^{-3}$ or the maximum iteration number reaches 200. In order to measure the quality of recovery from noisy measurements, the signal-to-noise ratio (SNR) in dB is used, which is denoted below:
$ \mathrm{SNR}(u_{r}, u_g)=-10\log_{10}\tfrac{\| u_r-u_g\|^2}{\|u_r\|^2},$
where $u_r$ and $u_g$ corresponds to the recovery image and ground truth respectively.

We initialize all variables as shown in Algorithms 1-3. As mentioned in subsection \ref{sec-ST-AGM}, the parameter $\epsilon$ used for the  smooth truncated metric \eqref{eq:obj-1} sets to 0.5. The other parameters appearing in the proposed algorithms are chosen heuristically  in order to either get faster convergence for noiseless measurements, or get better recovery quality for noisy measurements, which will be specified in the following tests. Finally,  the final solution over the entire domain is obtained by merging the solutions on the subdomains derived by proposed algorithms, i.e.   keeping the pixel values of non-overlapping region unchanged, while averaging the overlapping regions.
\begin{figure}
\begin{center}
\subfigure[]{\includegraphics[width=.15\textwidth]{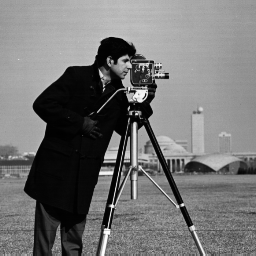}}
\subfigure[]{\includegraphics[width=.15\textwidth]{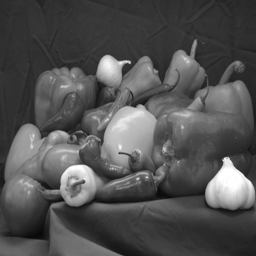}}\ \ \ \ \qquad\qquad
\subfigure[]{\includegraphics[width=.04\textwidth]{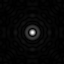}}
\subfigure[]{\includegraphics[width=.04\textwidth]{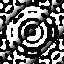}}
\end{center}
    \caption{Ground truth of the sample with absolute part (a) and phase part (b), probe with absolute part (c) and  phase part (d).  }
\label{fig2}
\end{figure}

\subsection{Performances of OD$^2$P}\label{subsec-1}

We first show the performance of the proposed  OD$^2$P algorithm for the nonblind case with two-subdomain decomposition. Set the parameters $\eta=0.1, r=4.0\times 10^3$.
The resolution of the sample  is of $256\times 256$ pixels. Set the scan stepsize to 8 pixels for both $x$ and $y$ directions. Therefore the entire measurement $f$ consists of $J=25\times 25$ frames, and its partitions $f_1$ and $f_2$ consist of $13\times 25$ and $12\times 25$ frames respectively. { The width of the overlapping region is 56 (pixels).} The recovery results are put in Fig. \ref{fig3}, where the recovery results (absolute parts in the first row of Fig. \ref{fig3} and phase parts in the second row of Fig. \ref{fig3})  at 1st, 2nd, 5th, 50th, and the final iterations demonstrate how the proposed algorithm improve the results gradually.  Especially, one may notice that the mismatch between the overlapping regions of  two decomposed parts disappears as iteration goes, by inferring from the first three columns of Fig. \ref{fig3}. That again shows the efficiency of the proposed algorithm. The absolute values of residual between the recovery images and the truth are put in Fig. \ref{fig3-1}, where one can obviously observe the  differences between two decomposed parts get weak as iteration goes. 
The errors changes including the relative errors and R-factors w.r.t. iteration numbers are  put in Fig. \ref{fig3-2} (a), that implies the convergence of the proposed algorithm. { However,  ripples appear in these curves of Fig. \ref{fig3-2} (a). With larger $\eta=0.2$, such ripples disappear as shown in Fig. \ref{fig3-2} (b), while the convergence gets slower since the proposed algorithm reaches the desired error tolerance after more iterations (181 iterations).  Please see more tests about the impact by the parameters in section \ref{sec-4.3}.}

\begin{figure}
\begin{center}
\subfigure{\includegraphics[width=.18\textwidth]{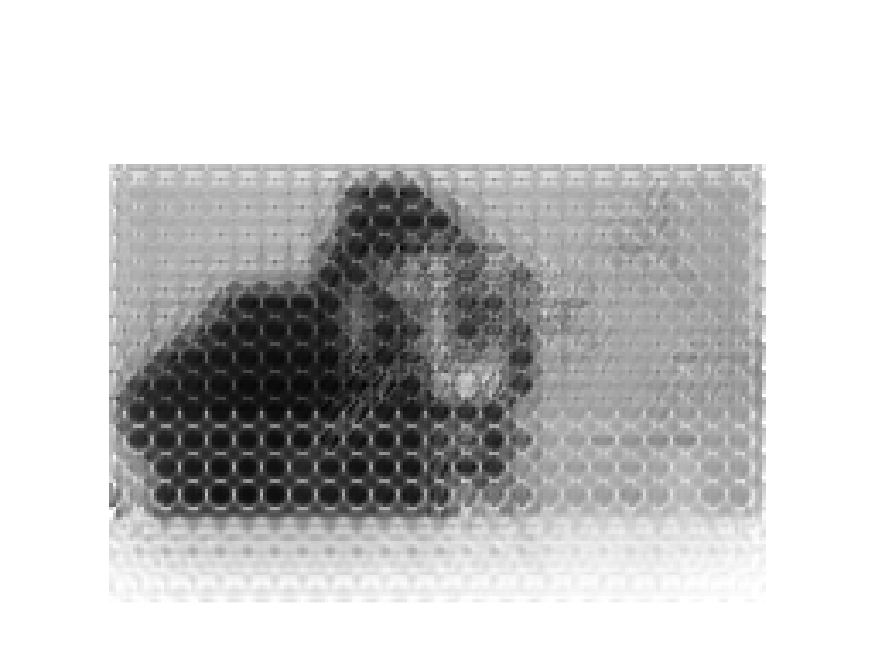}}
\subfigure{\includegraphics[width=.18\textwidth]{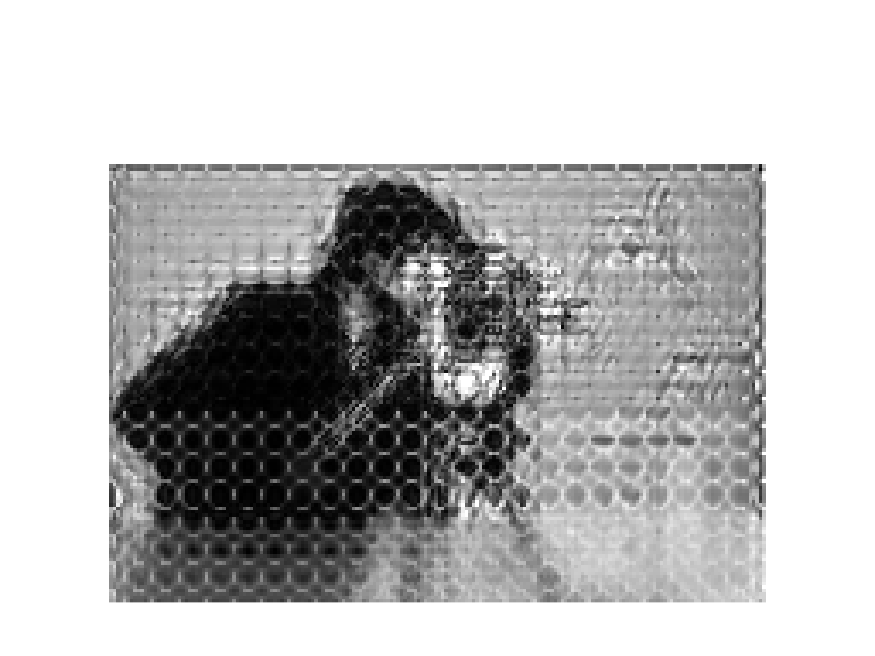}}
\subfigure{\includegraphics[width=.18\textwidth]{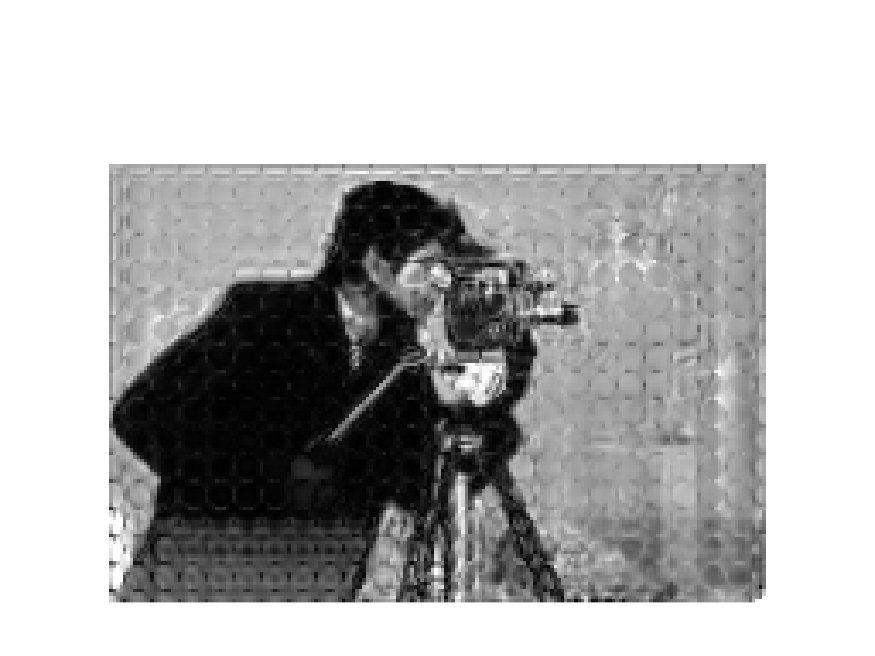}}
\subfigure{\includegraphics[width=.18\textwidth]{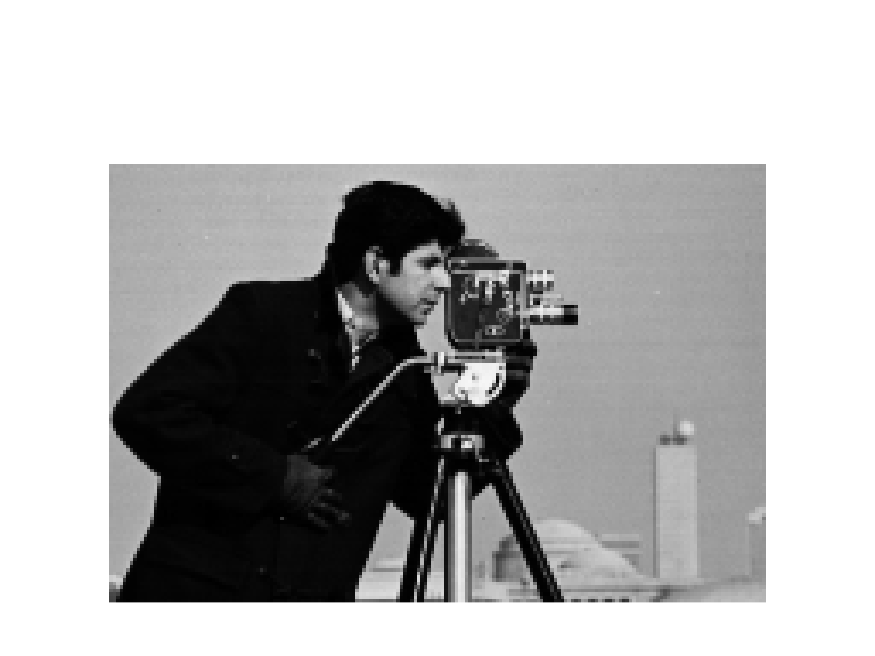}}
\subfigure{\includegraphics[width=.18\textwidth]{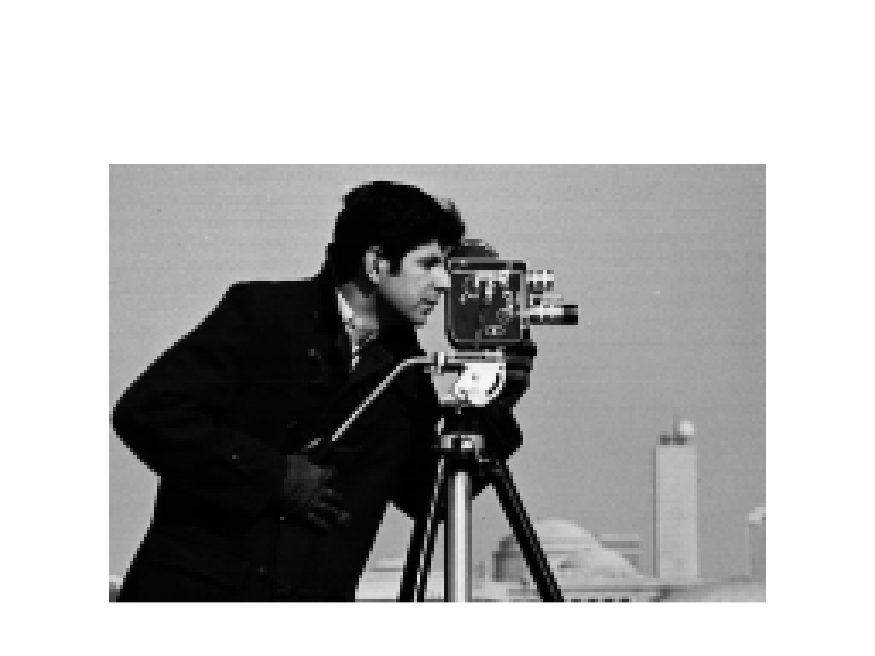}}
\vskip -.4in
\subfigure{\includegraphics[width=.18\textwidth]{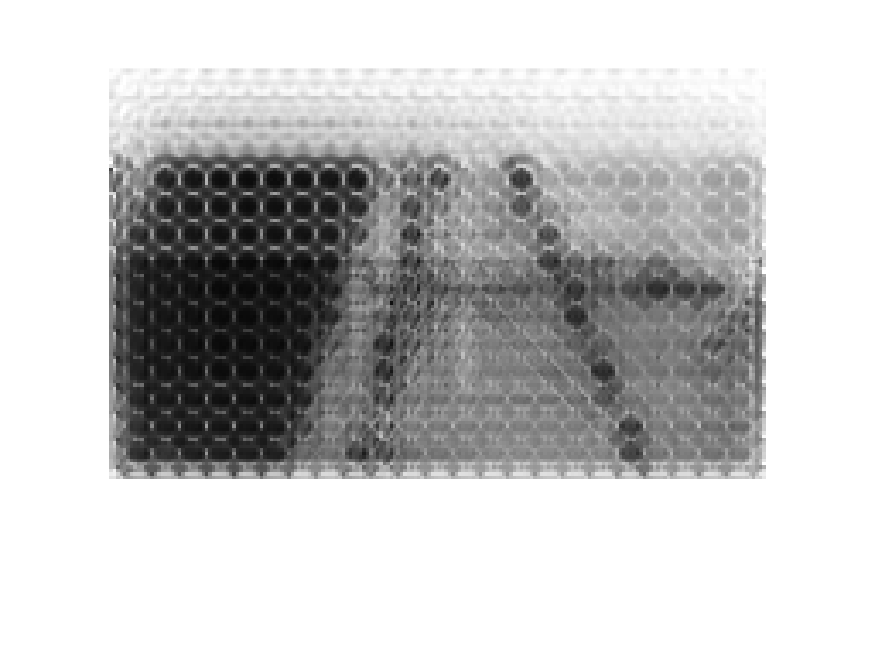}}
\subfigure{\includegraphics[width=.18\textwidth]{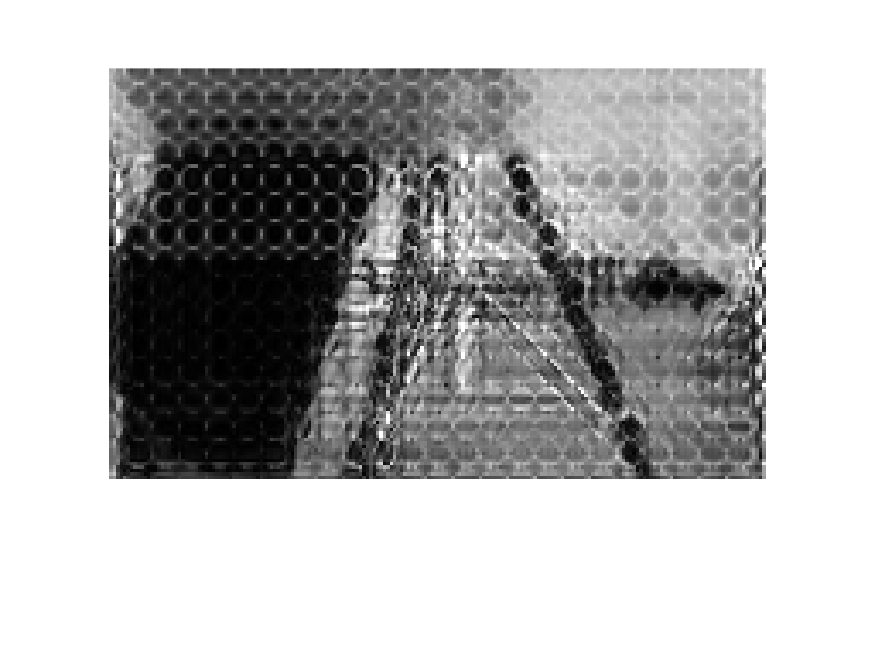}}
\subfigure{\includegraphics[width=.18\textwidth]{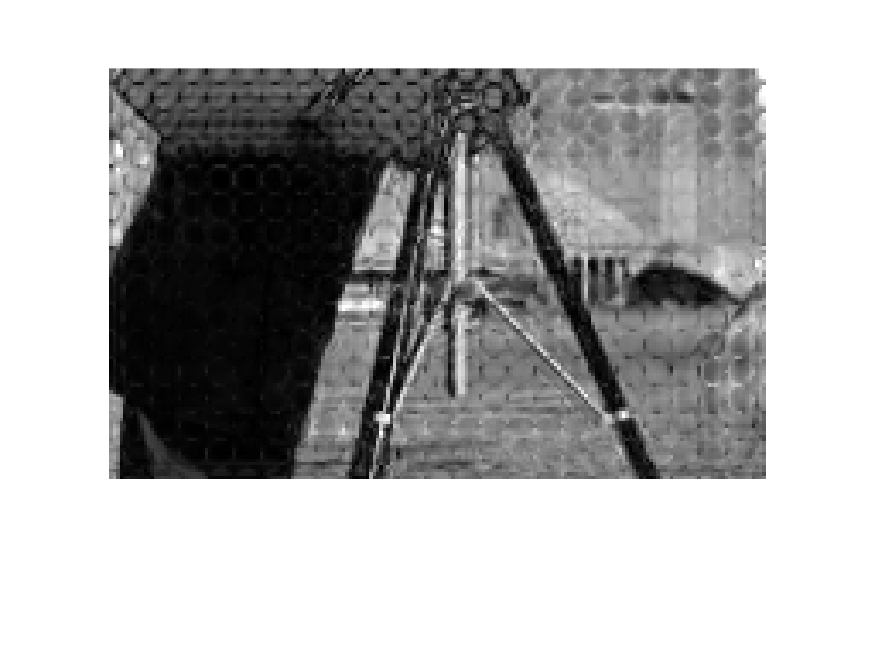}}
\subfigure{\includegraphics[width=.18\textwidth]{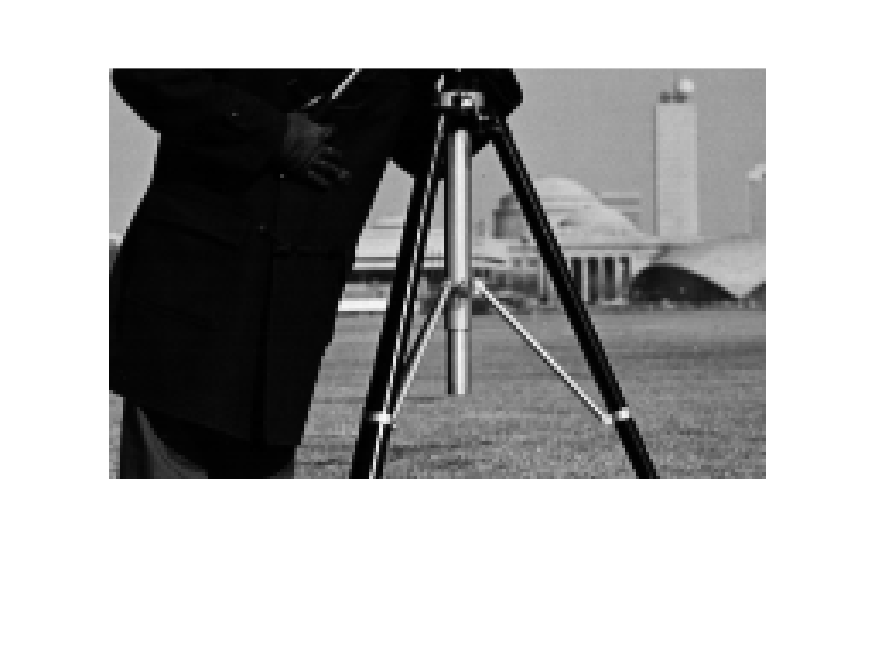}}
\subfigure{\includegraphics[width=.18\textwidth]{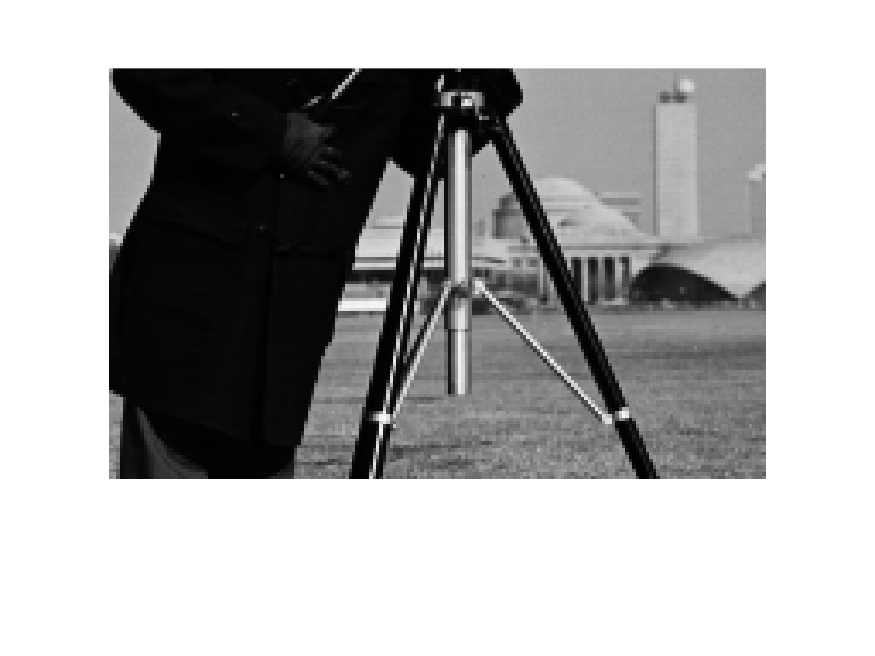}}
\\
\subfigure{\includegraphics[width=.18\textwidth]{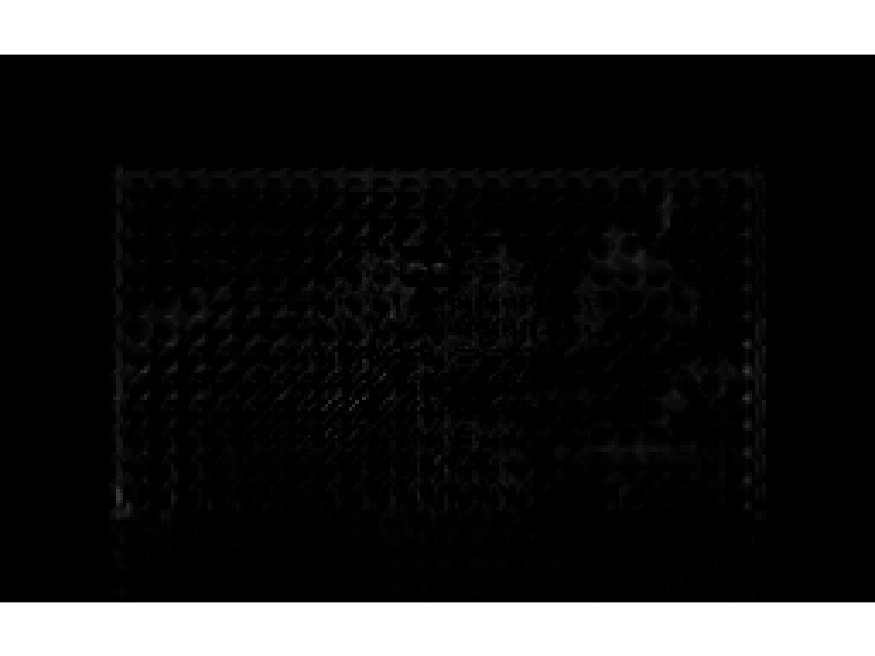}}
\subfigure{\includegraphics[width=.18\textwidth]{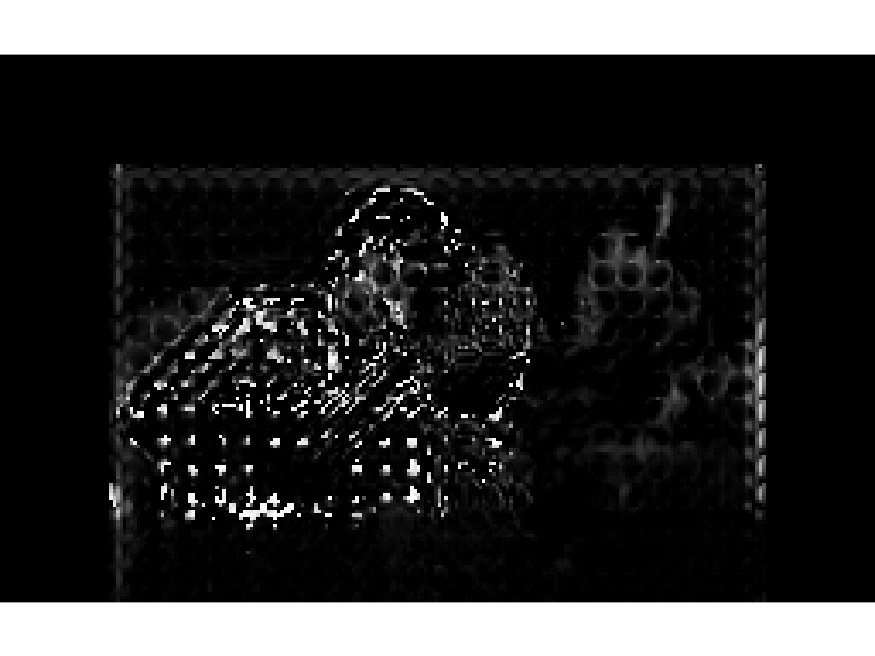}}
\subfigure{\includegraphics[width=.18\textwidth]{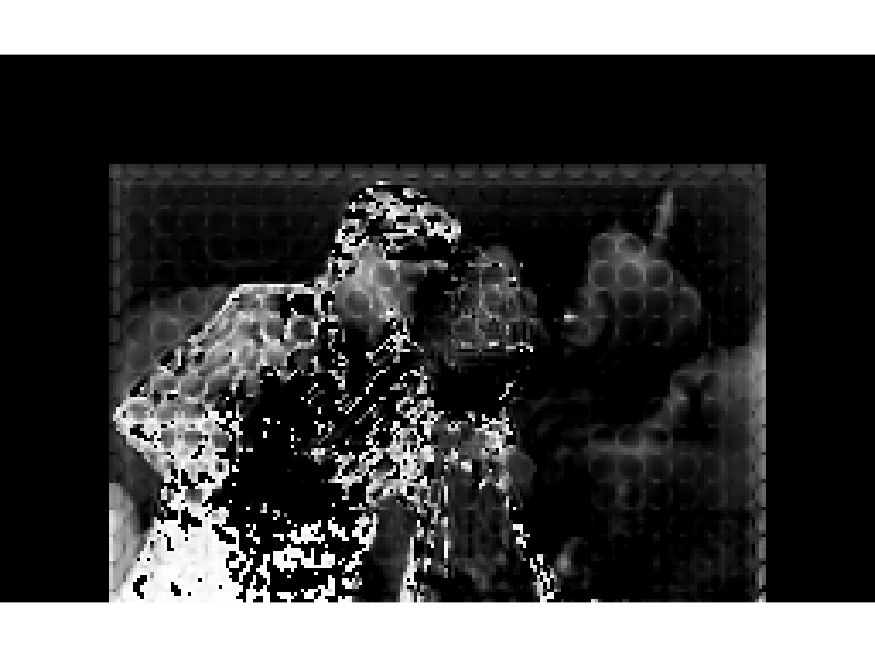}}
\subfigure{\includegraphics[width=.18\textwidth]{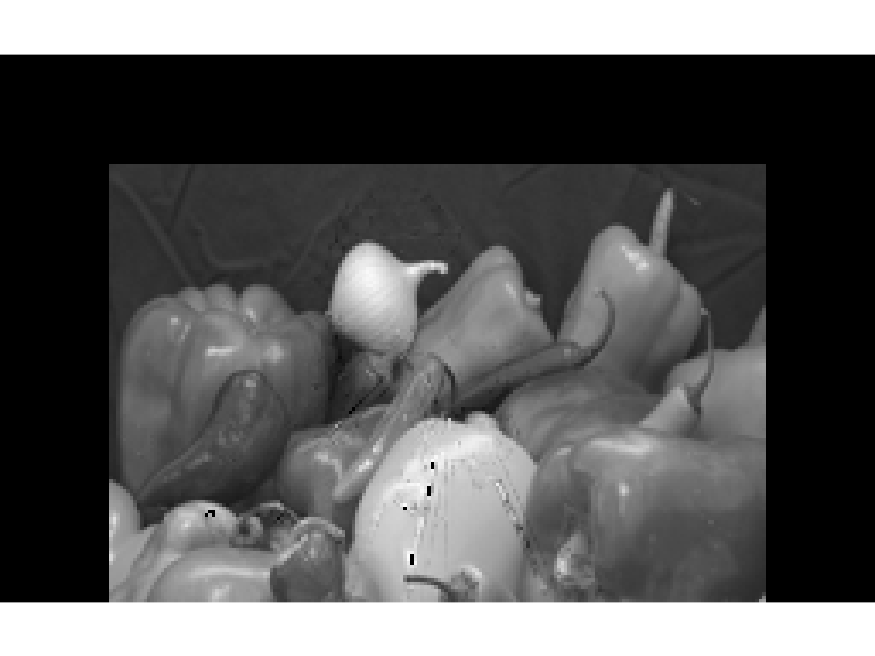}}
\subfigure{\includegraphics[width=.18\textwidth]{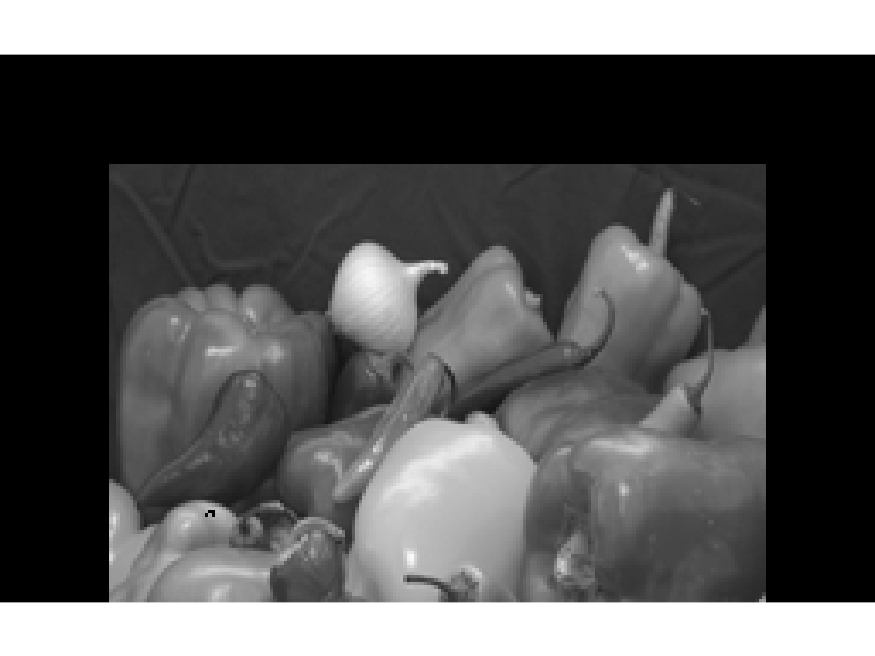}}
\vskip -.4in
\setcounter{subfigure}{0}
\subfigure[1st]{\includegraphics[width=.18\textwidth]{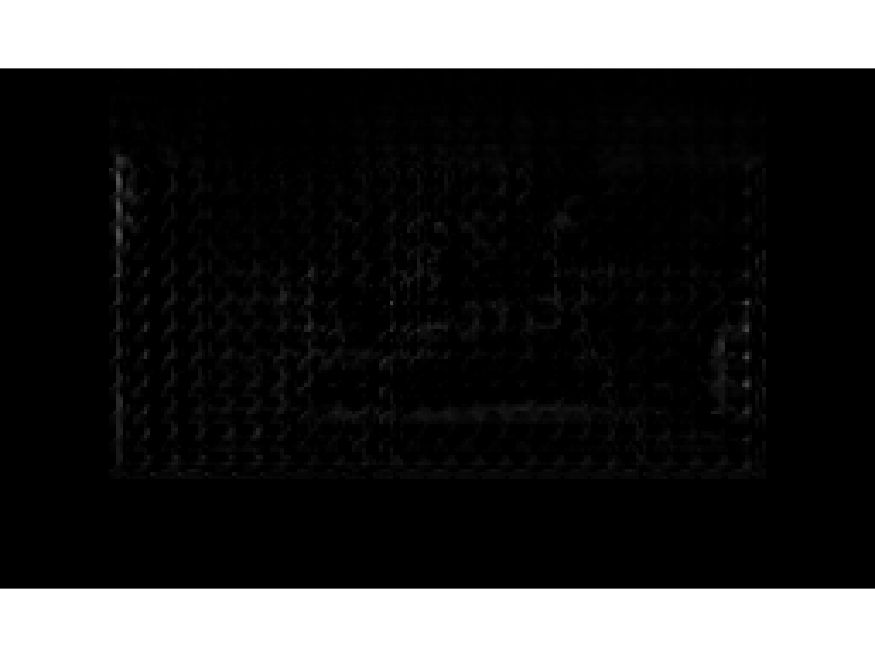}}
\subfigure[2nd]{\includegraphics[width=.18\textwidth]{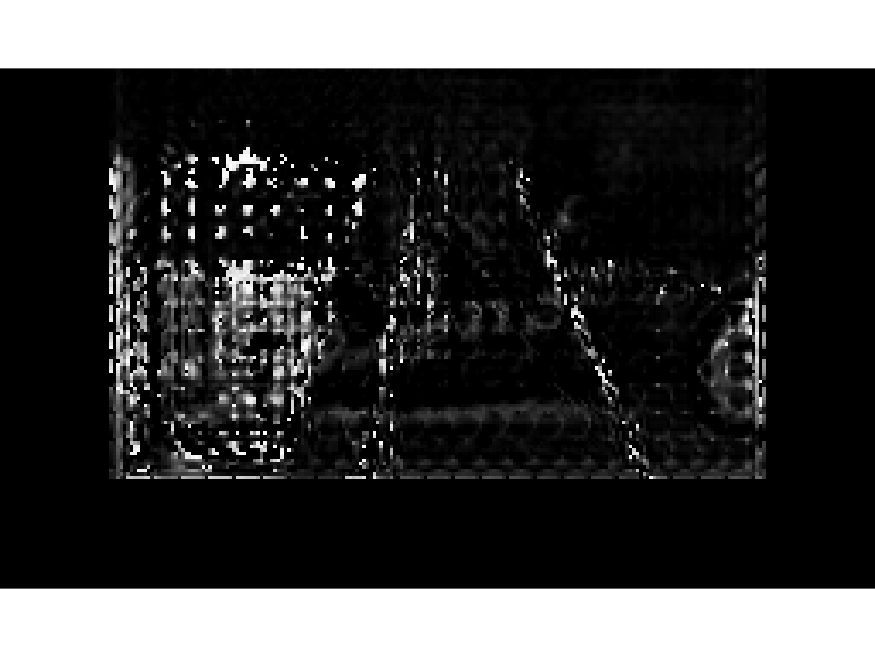}}
\subfigure[5th]{\includegraphics[width=.18\textwidth]{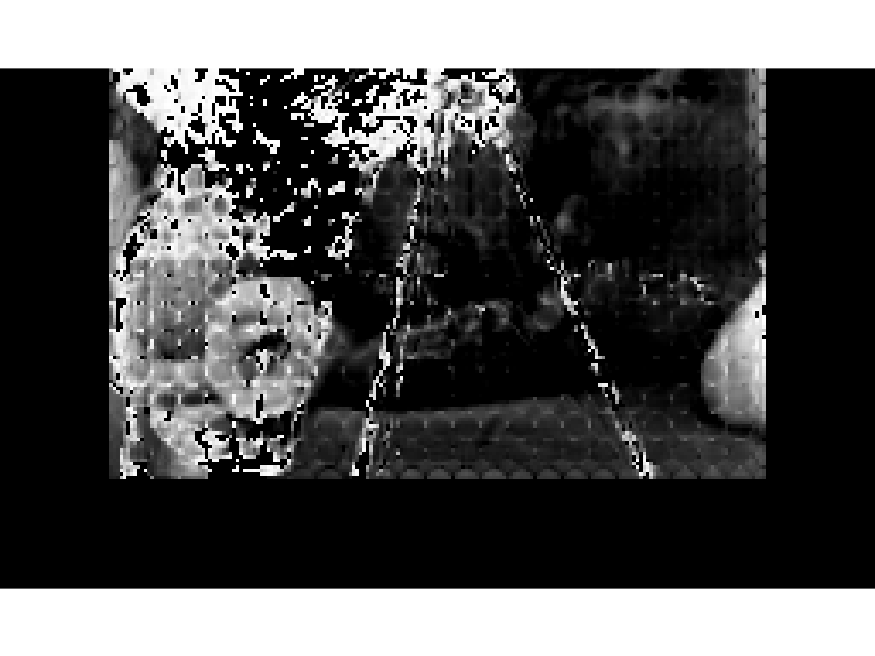}}
\subfigure[50th]{\includegraphics[width=.18\textwidth]{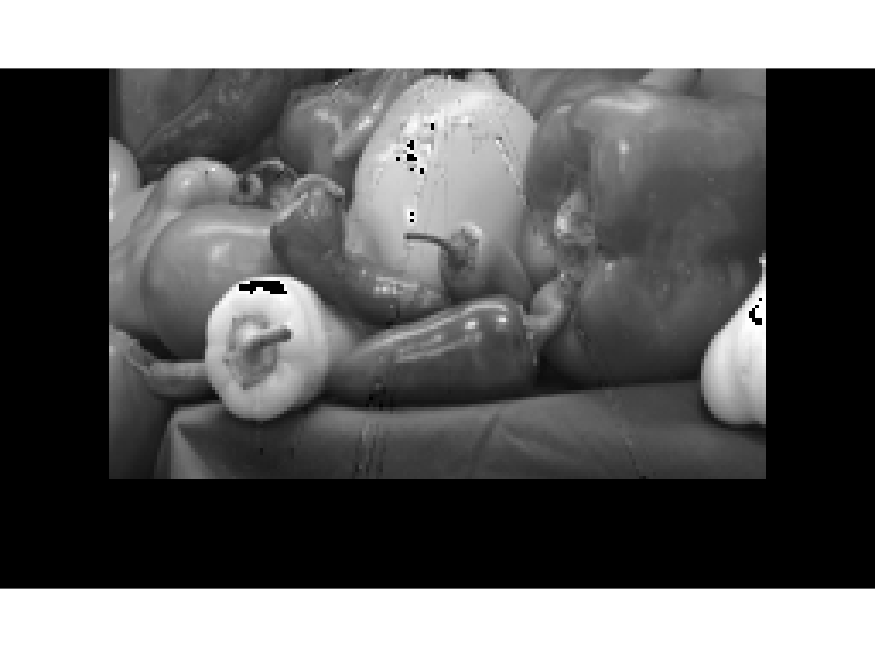}}
\subfigure[168th]{\includegraphics[width=.18\textwidth]{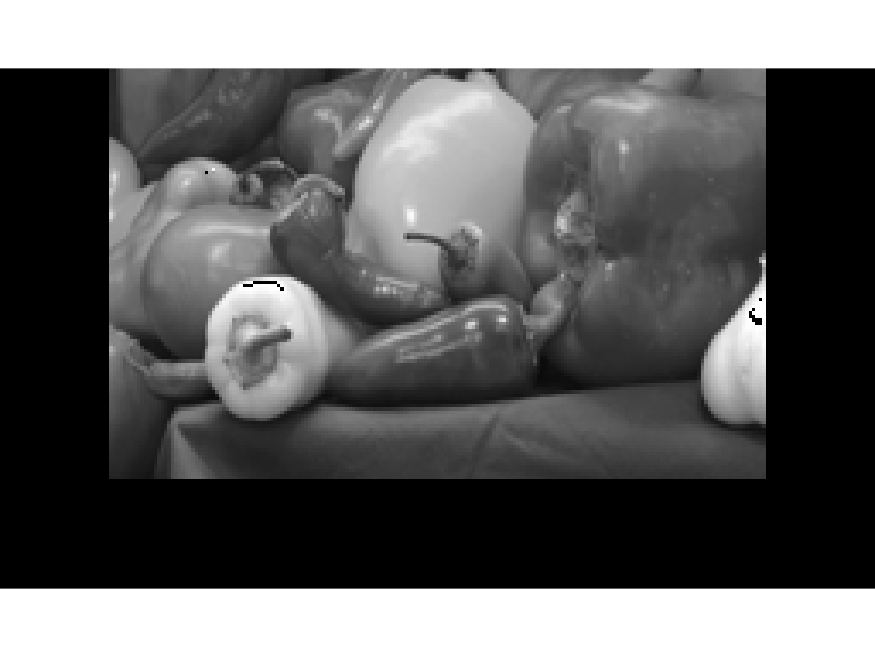}}
\end{center}
    \caption{Recovery results at 1st, 2nd, 5th, 50th, 168th (final iteration) iterations.  Recovered absolute parts (top and down parts for decomposed samples) in 1-2 rows and phase parts in 3-4 rows respectively (top and down parts for decomposed samples), shown in the range of $[0,1]$
    and $[0,\pi]$ for the absolute and phase parts respectively.}
\label{fig3}
\end{figure}

\begin{figure}
\begin{center}
\subfigure[]{\includegraphics[width=.19\textwidth]{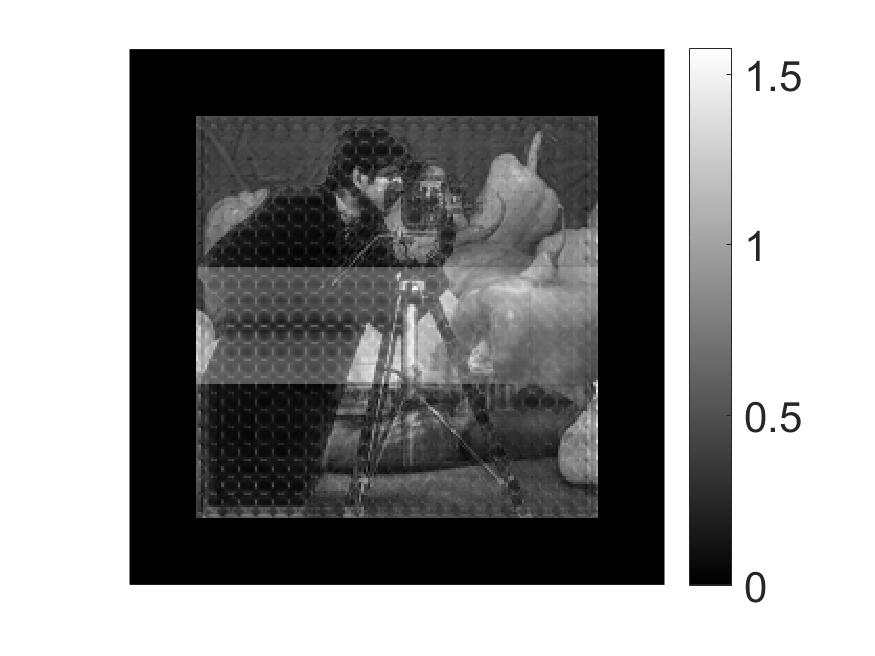}}
\subfigure[]{\includegraphics[width=.19\textwidth]{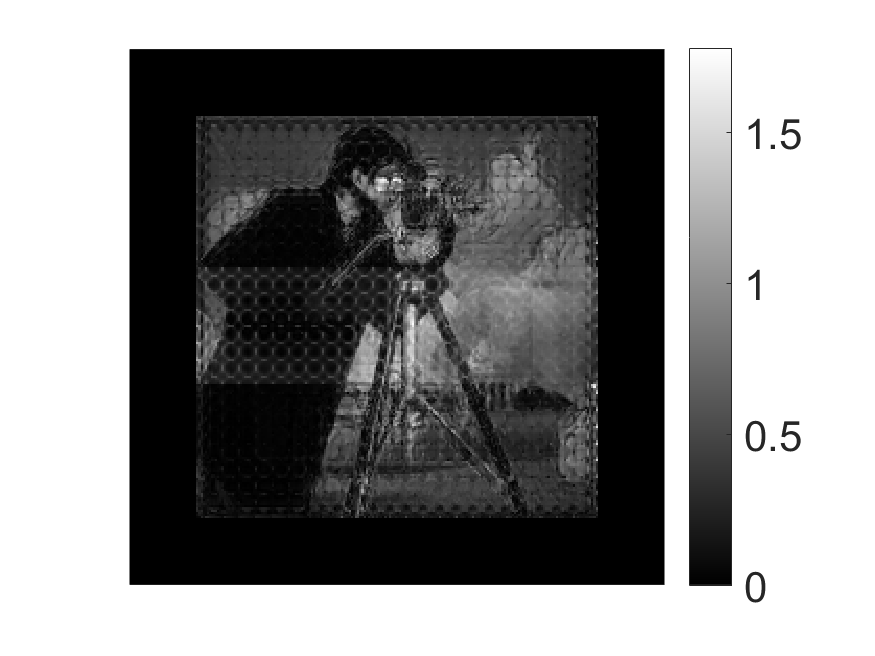}}
\subfigure[]{\includegraphics[width=.19\textwidth]{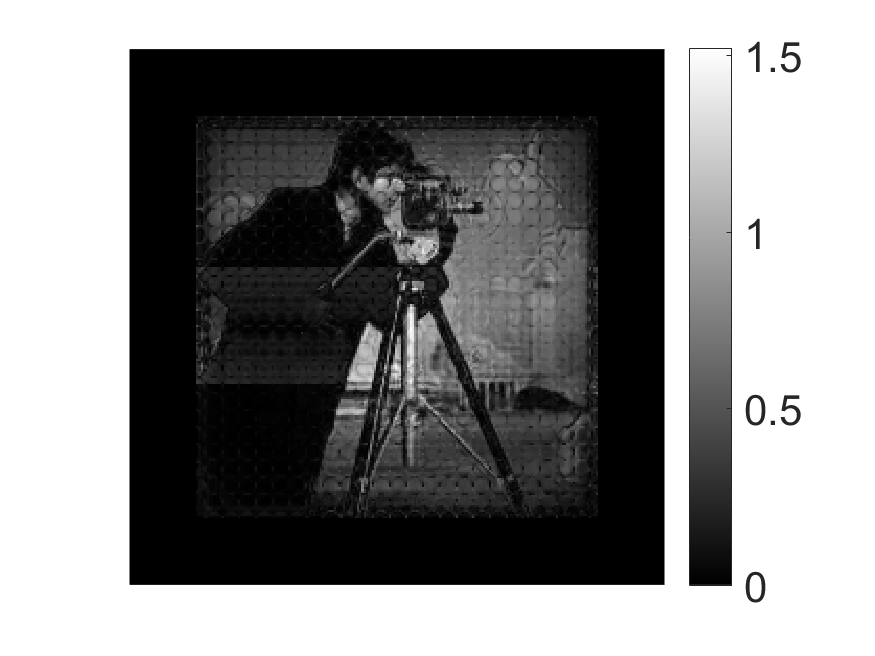}}
\subfigure[]{\includegraphics[width=.19\textwidth]{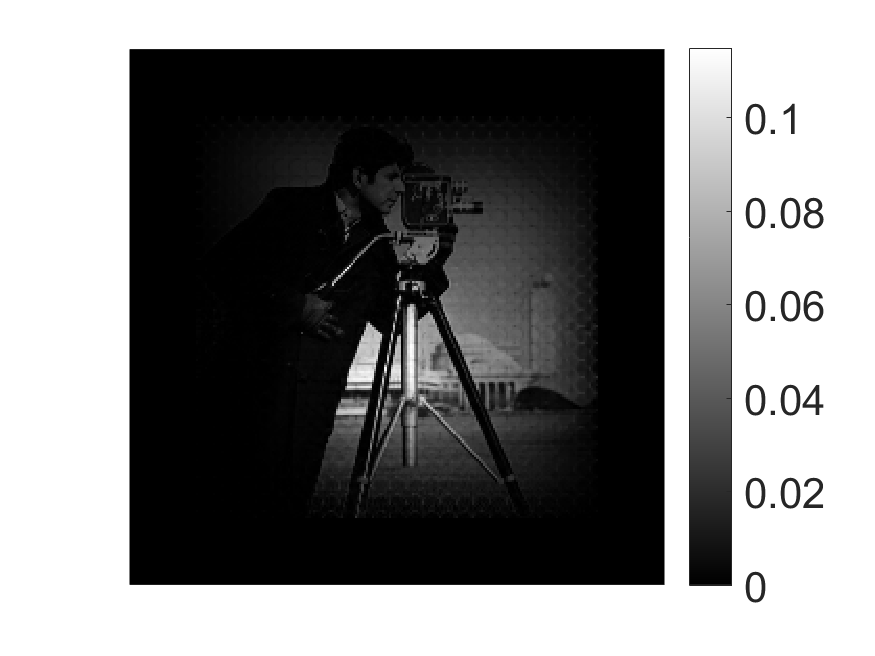}}
\subfigure[]{\includegraphics[width=.19\textwidth]{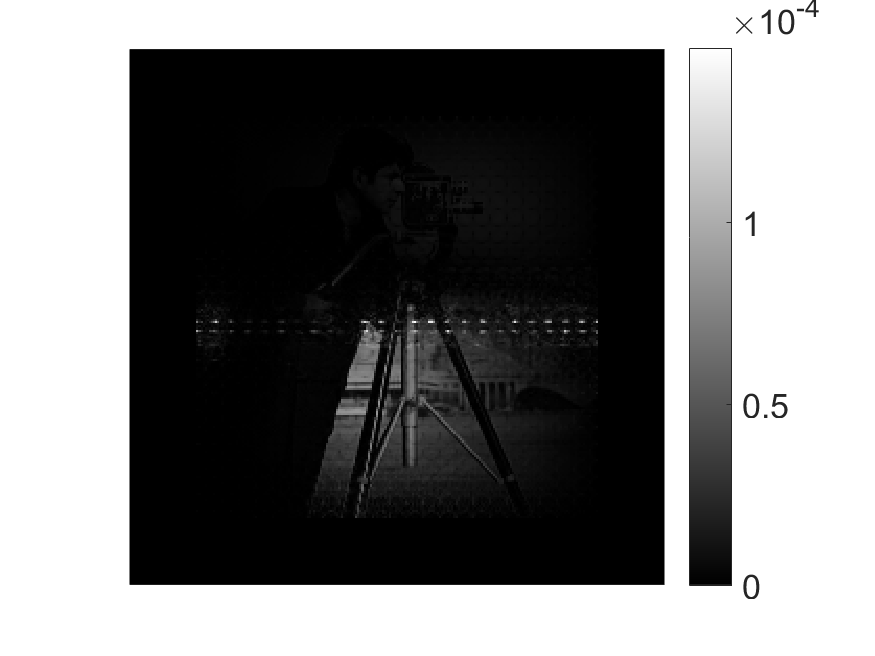}}
\vspace{-.1in}
\\
\subfigure[]{\includegraphics[width=.19\textwidth]{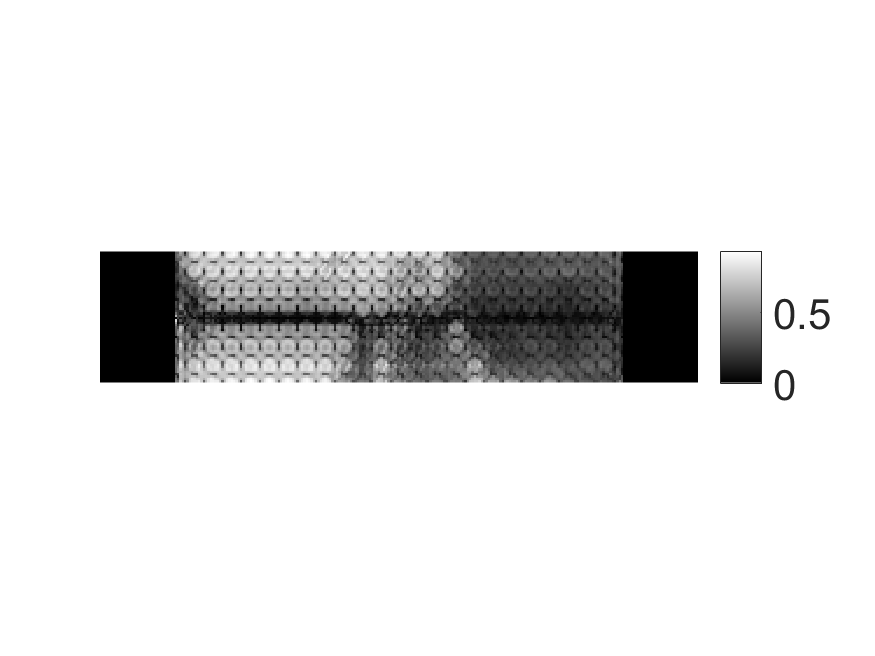}}
\subfigure[]{\includegraphics[width=.19\textwidth]{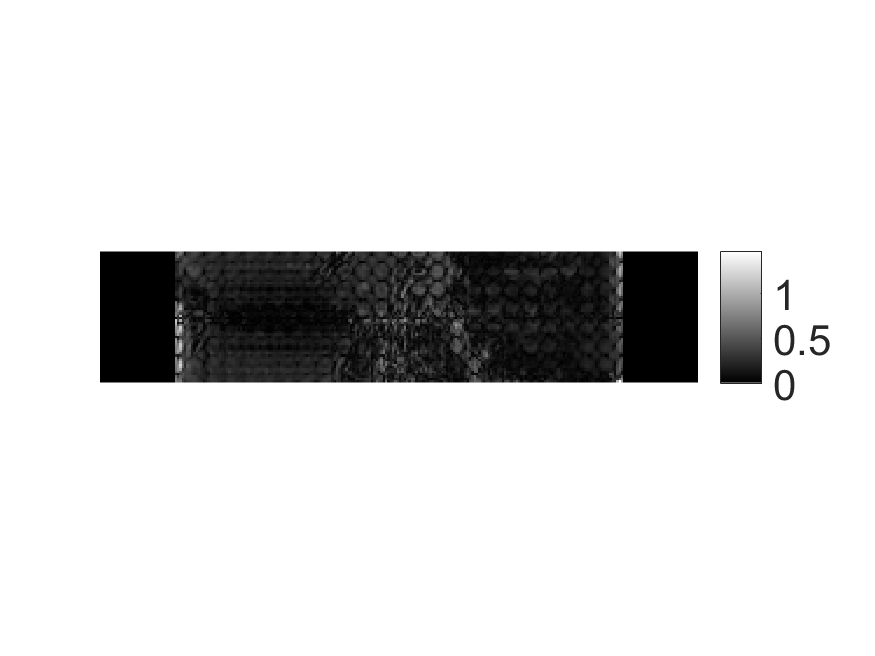}}
\subfigure[]{\includegraphics[width=.19\textwidth]{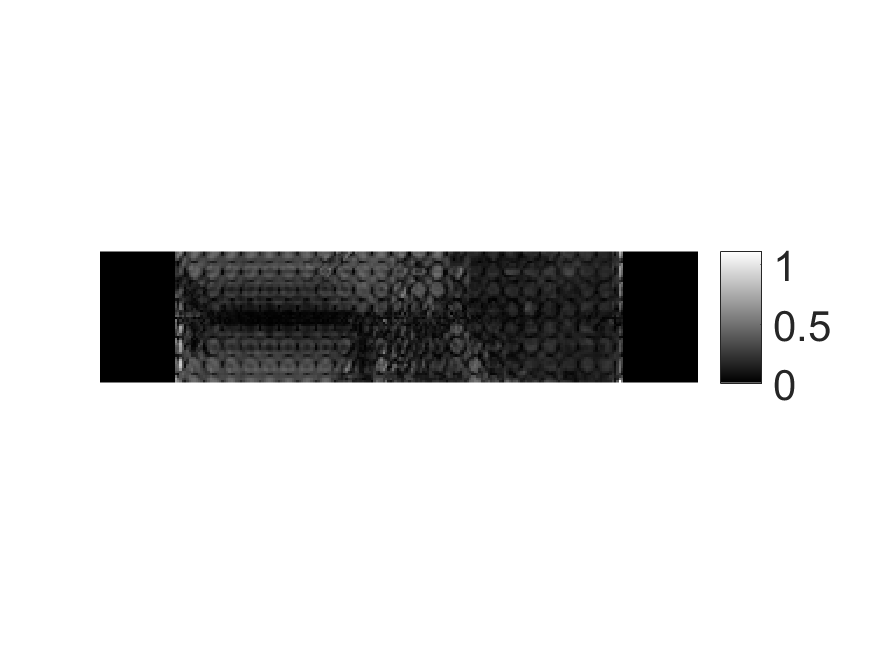}}
\subfigure[]{\includegraphics[width=.19\textwidth]{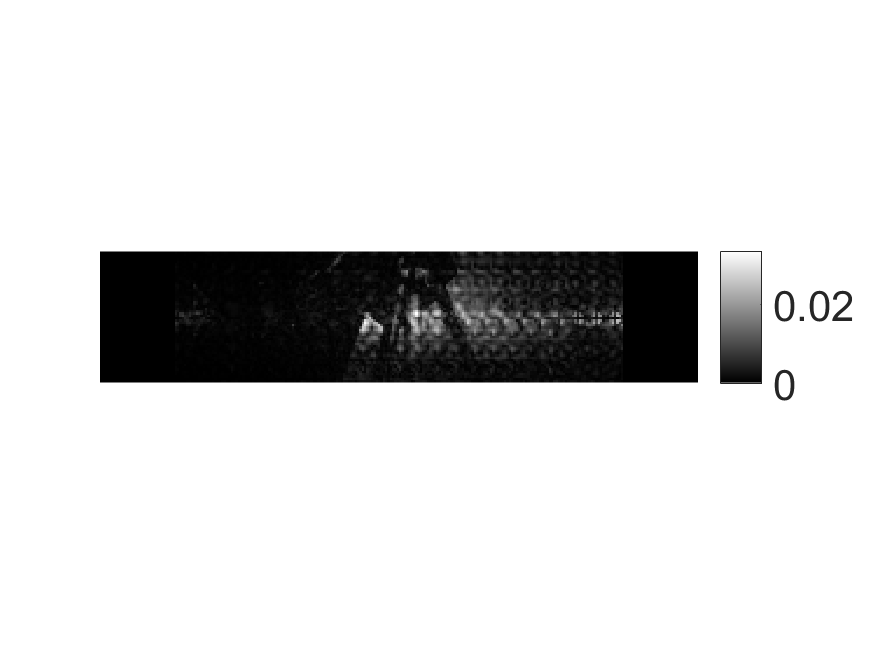}}
\subfigure[]{\includegraphics[width=.19\textwidth]{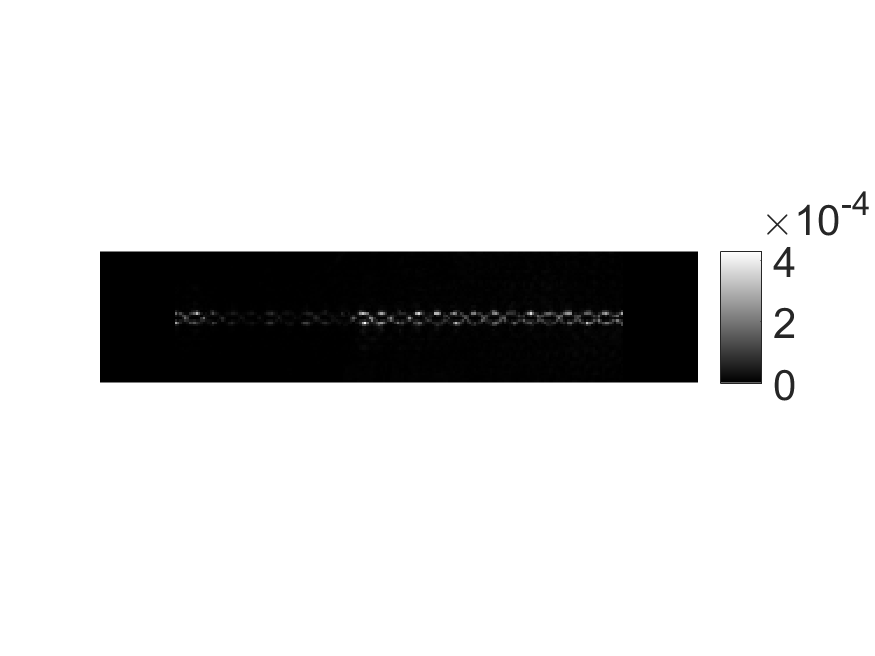}}
\vspace{-.15in}
\end{center}
    \caption{Residuals. (a)-(e): absolute values of the differences between the recovery results  at 1st, 2nd, 5th, 50th, 168th (final iteration) iterations and the truth;  (f)-(j): errors between the overlapping parts of two subdomains at different iterations.}
\label{fig3-1}
\end{figure}

\begin{figure}
\begin{center}
\subfigure[]{\includegraphics[width=.35\textwidth]{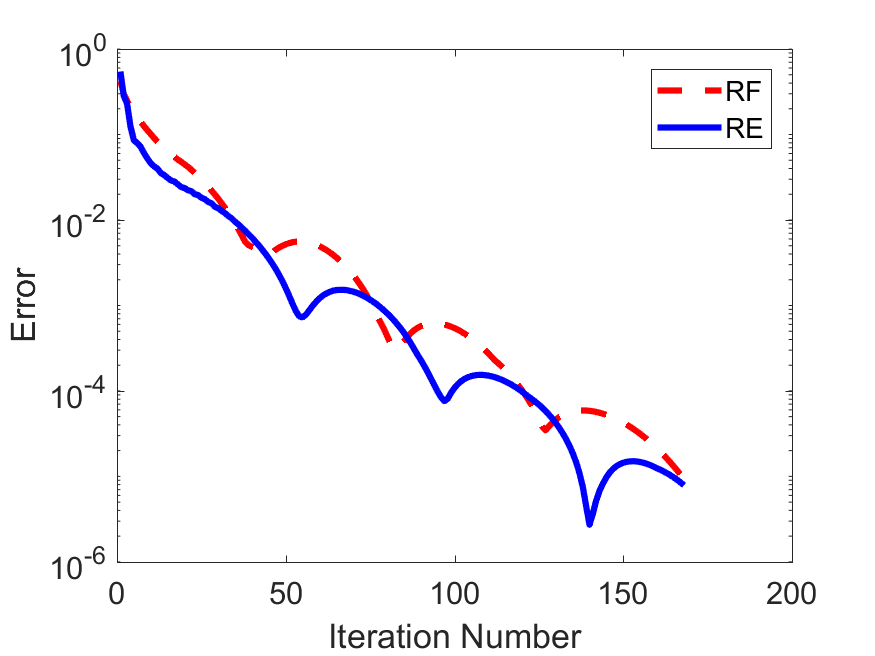}}
\subfigure[]{\includegraphics[width=.35\textwidth]{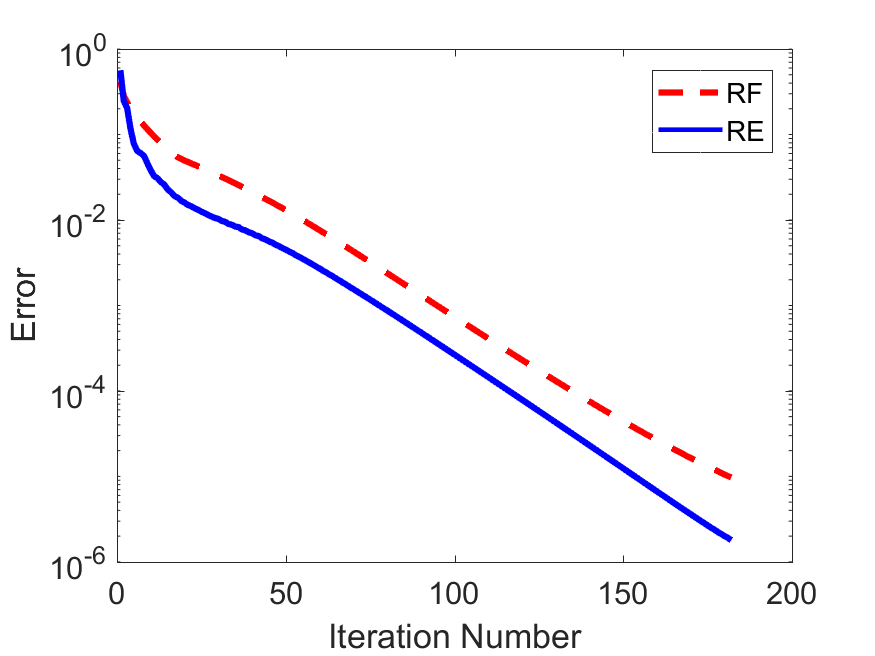}}\\
\end{center}
    \caption{RF (R-factor), and RE (Relative error)  changes v.s. iterations with different $\eta$ in the case of nonblind and 2-subdomain DD. { Left: $\eta=0.1, r=4.0\times 10^3$; Right: $\eta=0.2,  r=4.0\times 10^3$.}  }
\label{fig3-2}
\end{figure}

To further show the performance the proposed algorithm, the test on the noisy data is given as well. Noisy data is generated by contamination of Poisson noises  with different levels (using a scaling factor to control the noise level), such that two different noisy cases with SNR${}_{intensity}=39.8, 29.9$dBs (SNR${}_{intensity}$  denotes the SNR between the noisy and  clean measurements) are considered.  The parameter $\eta$ sets to the same as the noiseless case, while $r$ sets to $90$, and $150$ for two different noise levels respectively. { Consider the same DD as above noiseless case.}
We put the recovery results in Fig. \ref{fig3-3} (the SNRs of final recovered images are 18.1 dB and 12.2 dB). That demonstrates the proposed algorithm can work well for the Poisson noisy data,  showing great potential for real experimental data analysis (usually contaminated by Poisson noise).

\begin{figure}[]
	\begin{center}
		\begin{tabular}{cccc}
			\subfigure{\includegraphics[width=.18\textwidth]{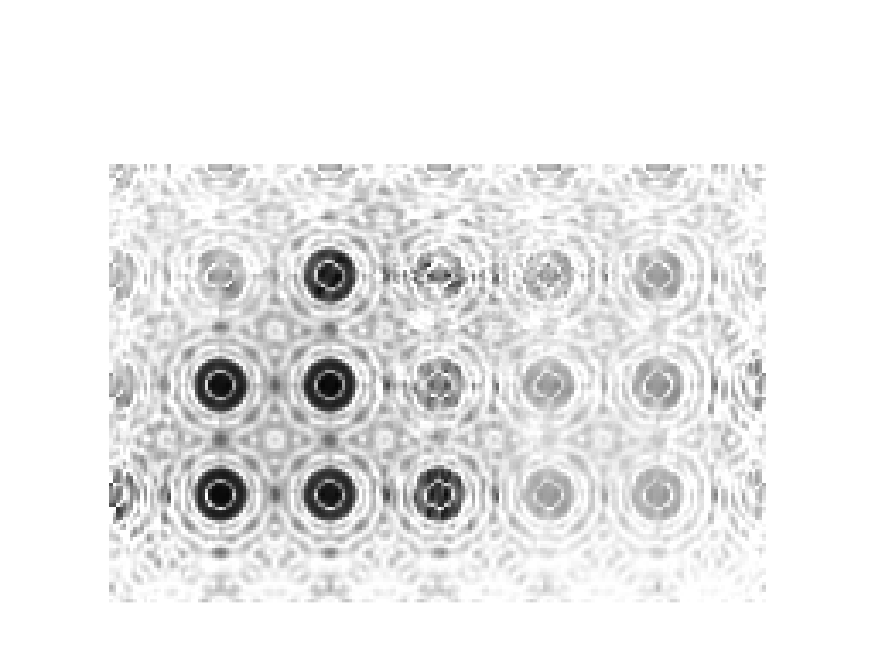}}
			&\subfigure{\includegraphics[width=.18\textwidth]{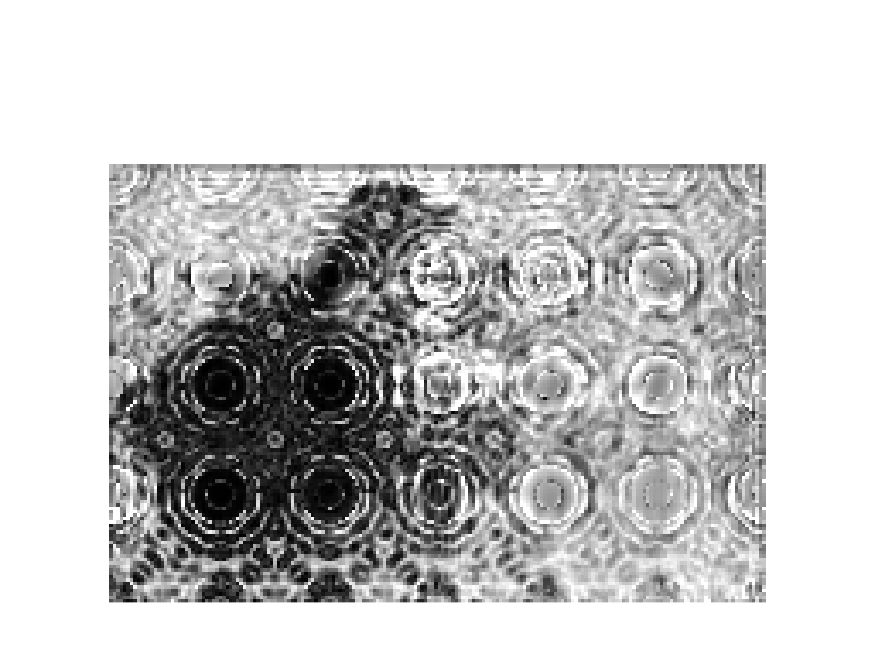}}
			&\subfigure{\includegraphics[width=.18\textwidth]{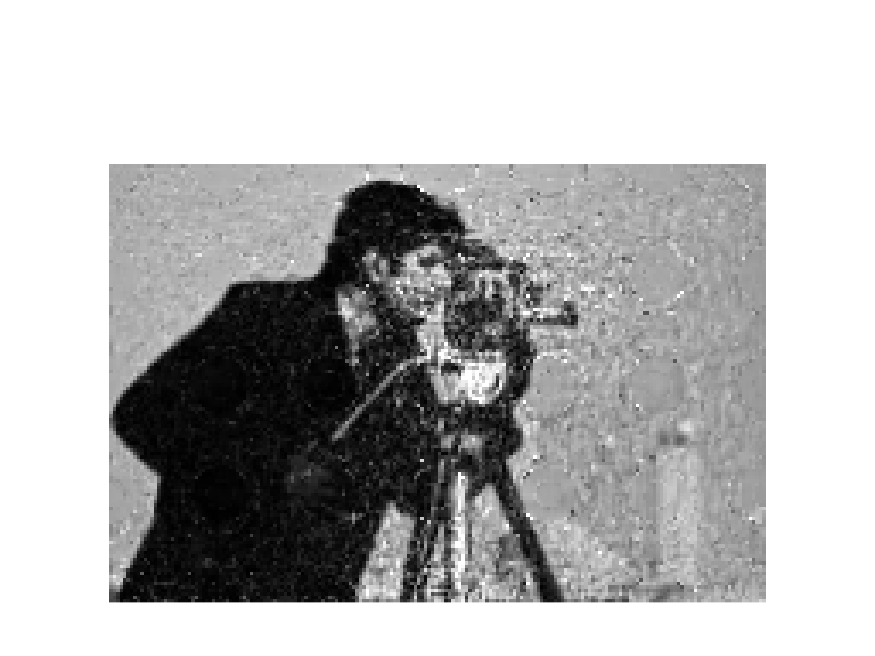}}
			&\subfigure{\includegraphics[width=.18\textwidth]{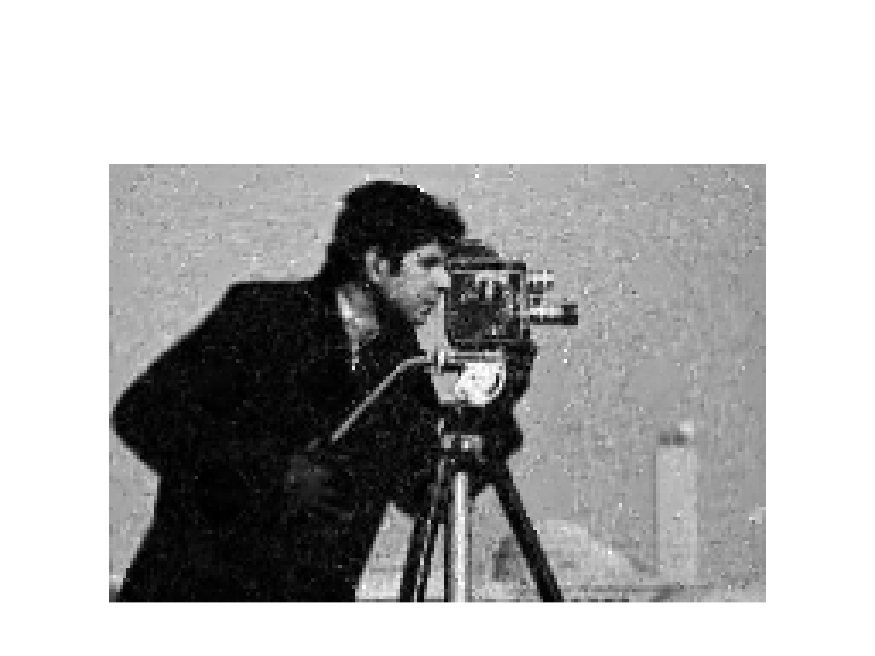}}			\vspace{-.5in}
			\\
			\subfigure{\includegraphics[width=.18\textwidth]{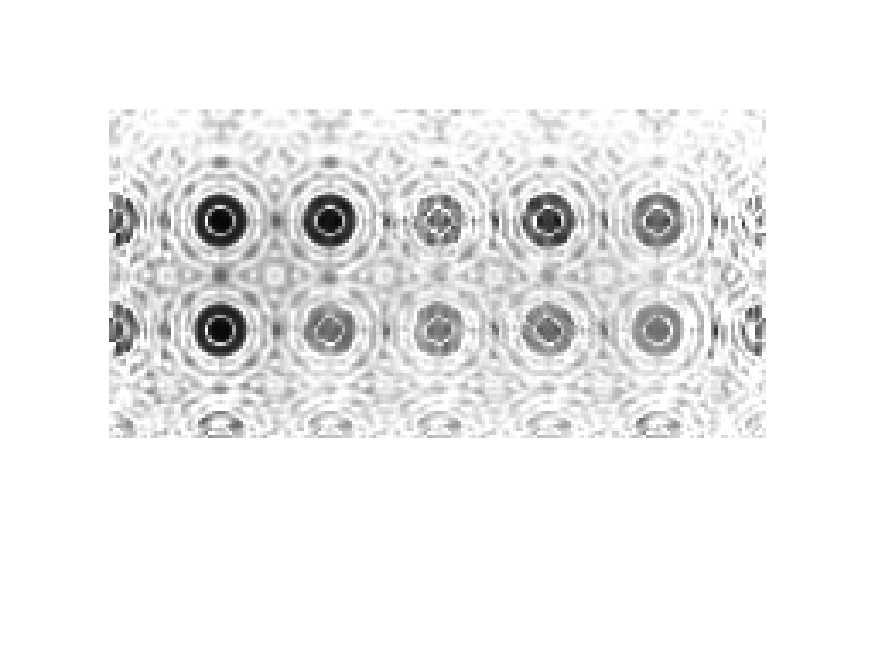}}
			&\subfigure{\includegraphics[width=.18\textwidth]{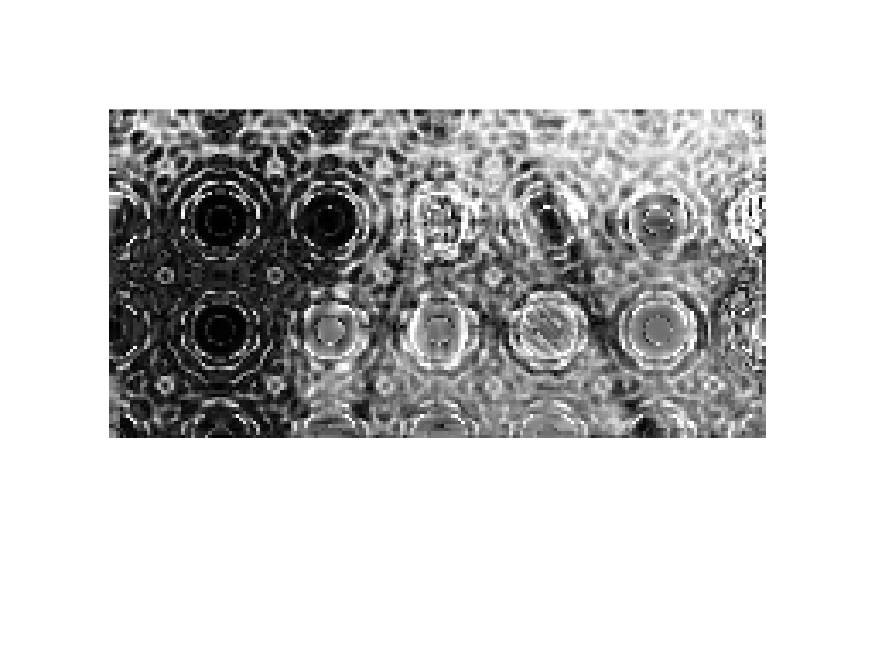}}
			&\subfigure{\includegraphics[width=.18\textwidth]{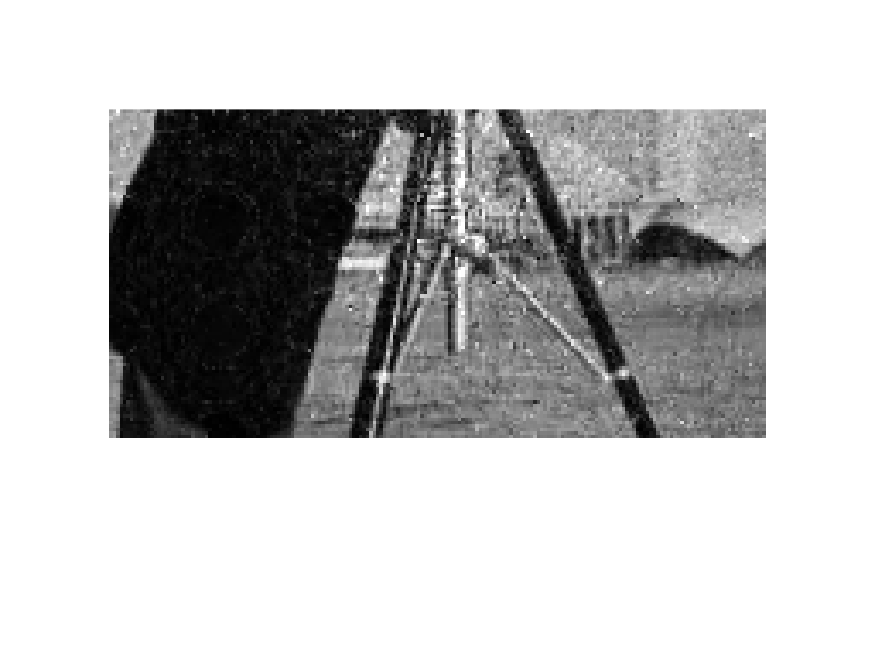}}
			&\subfigure{\includegraphics[width=.18\textwidth]{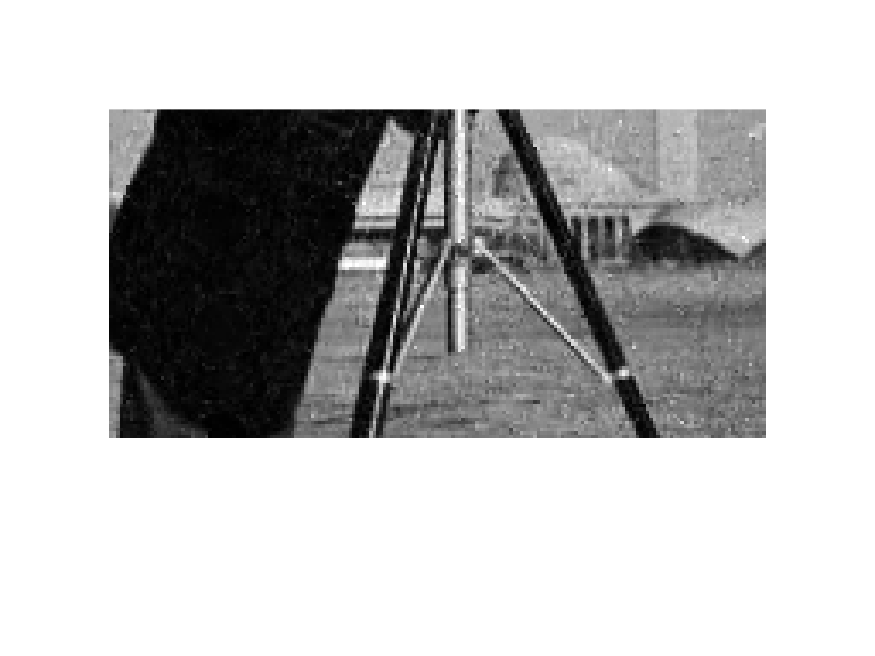}}
			\vspace{-.3in}
			\\
			\subfigure{\includegraphics[width=.18\textwidth]{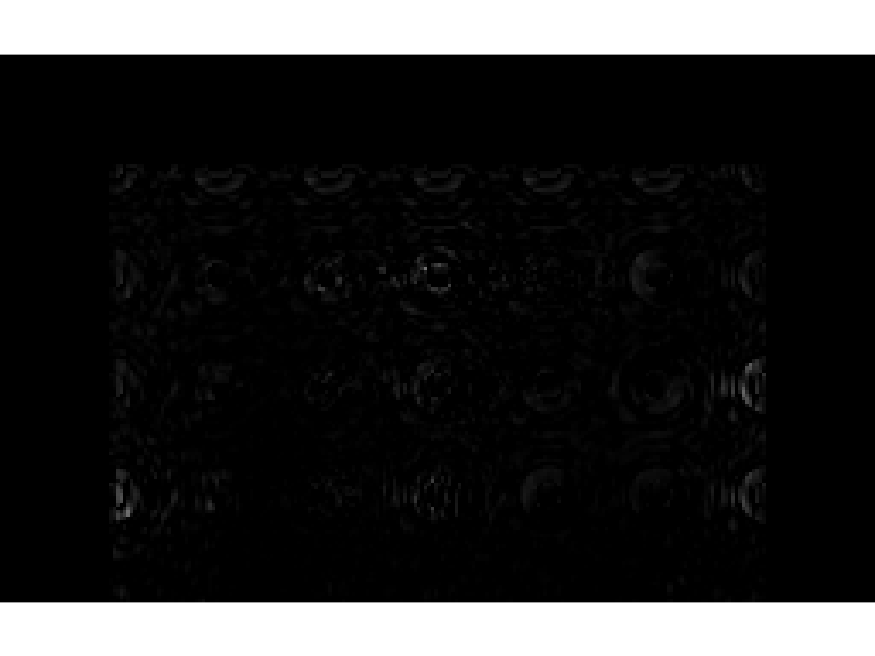}}
			&\subfigure{\includegraphics[width=.18\textwidth]{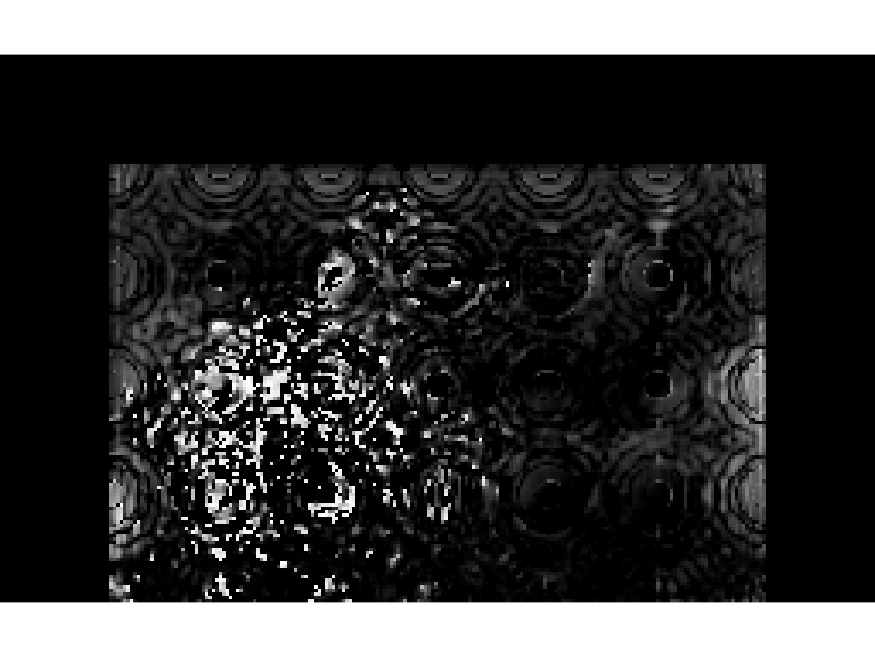}}
			&\subfigure{\includegraphics[width=.18\textwidth]{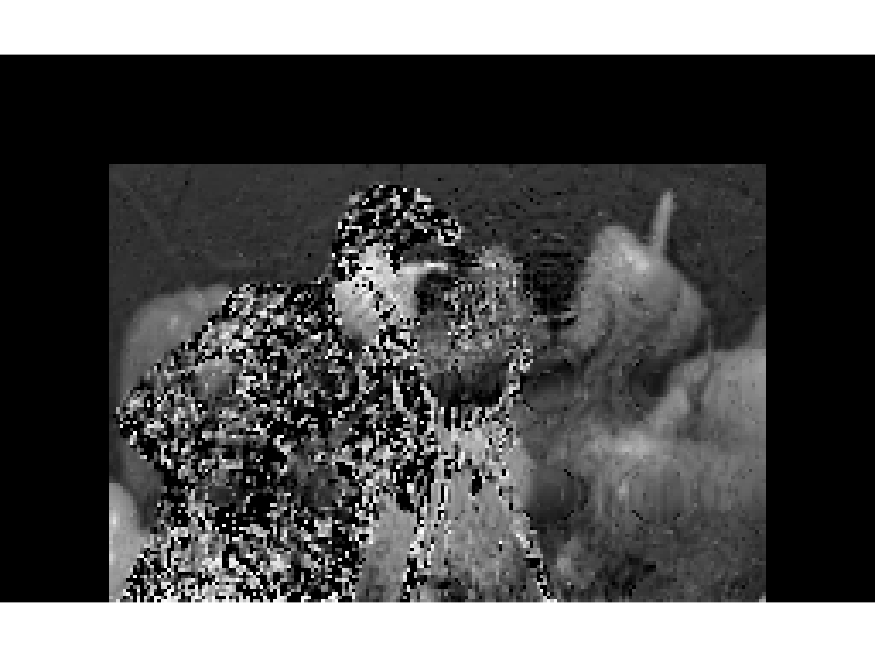}}
			&\subfigure{\includegraphics[width=.18\textwidth]{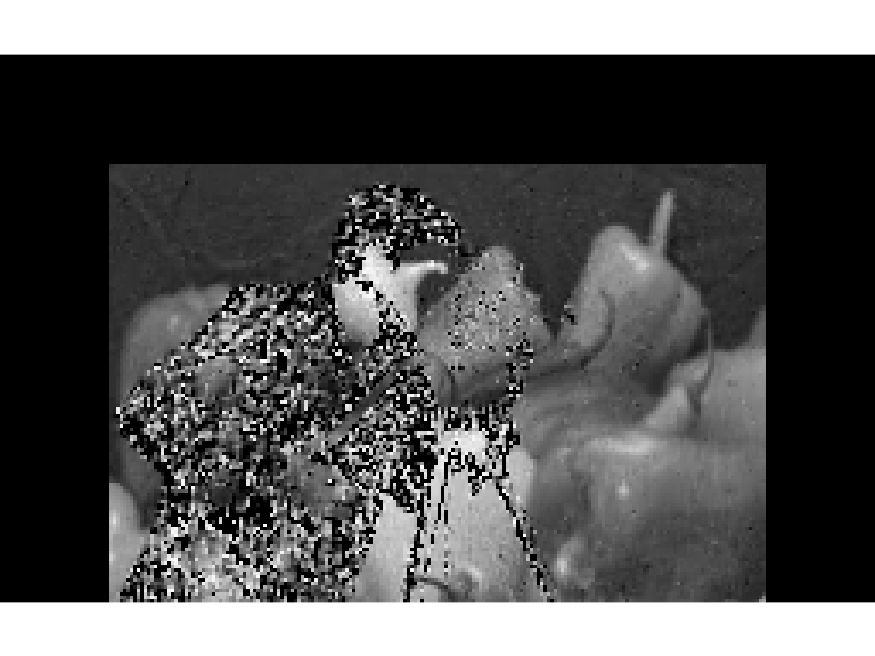}}
			\vspace{-.5in}
			\\
			\subfigure{\includegraphics[width=.18\textwidth]{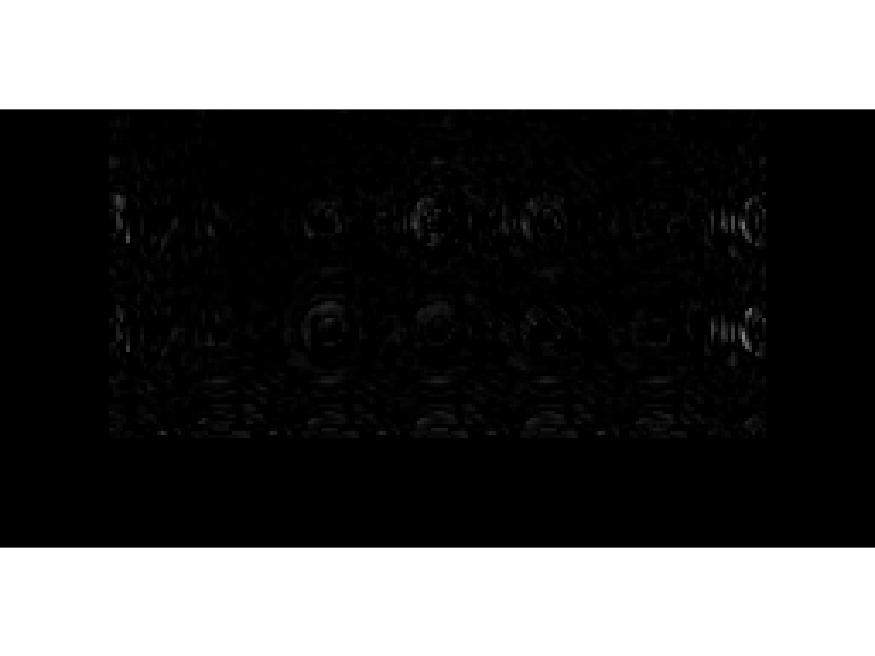}}
			&\subfigure{\includegraphics[width=.18\textwidth]{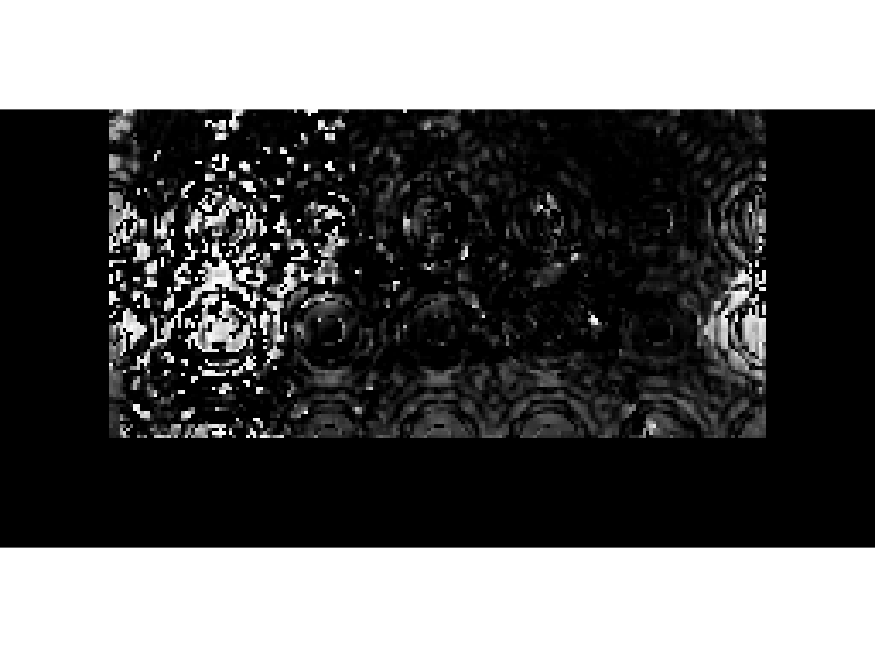}}
			&\subfigure{\includegraphics[width=.18\textwidth]{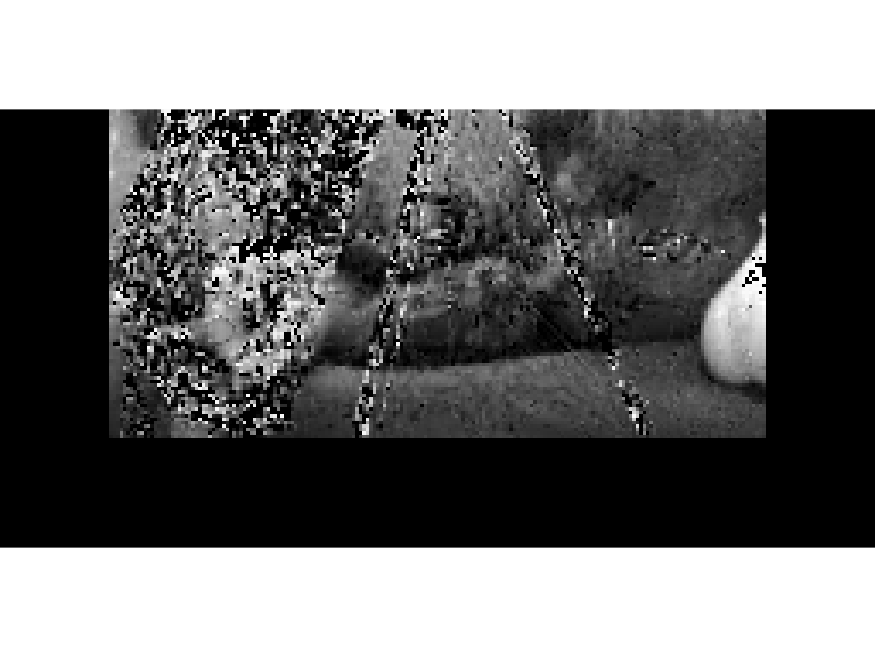}}
			&\subfigure{\includegraphics[width=.18\textwidth]{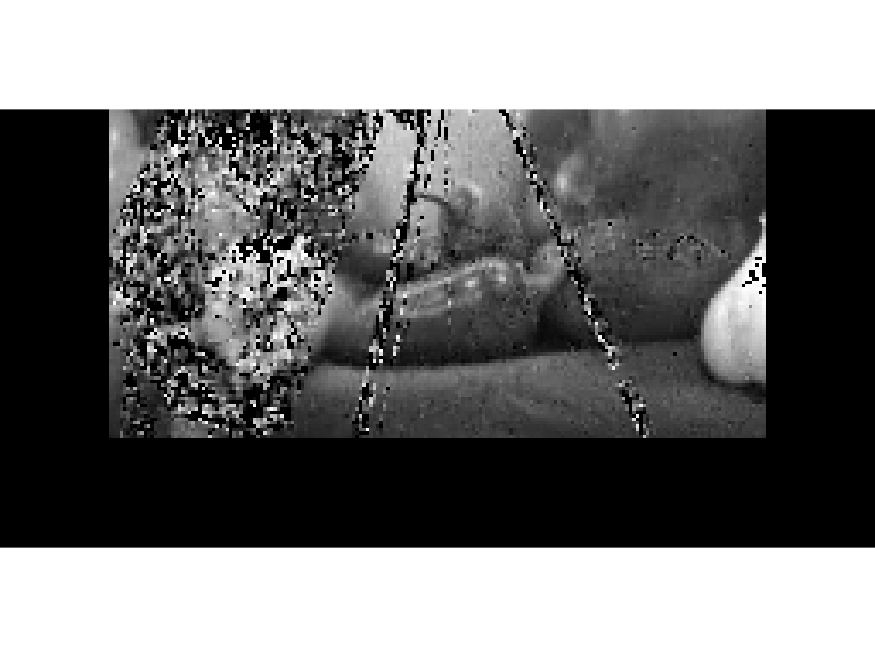}}\\
			\vspace{-.1in}
		    (a) 1st & (b) 5th &(c) 50th&(d) 200th\\
			\subfigure{\includegraphics[width=.18\textwidth]{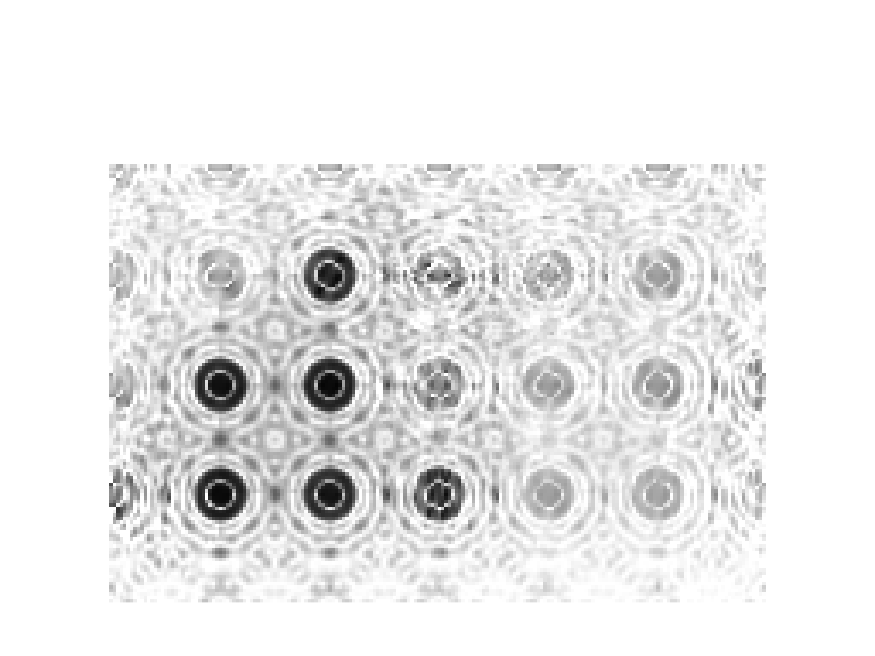}}
			&\subfigure{\includegraphics[width=.18\textwidth]{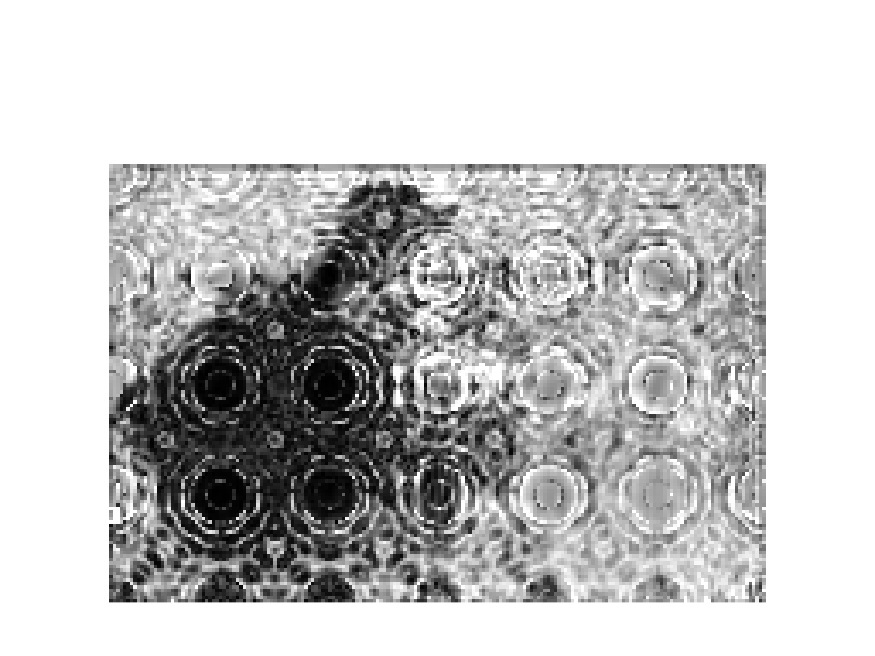}}
			&\subfigure{\includegraphics[width=.18\textwidth]{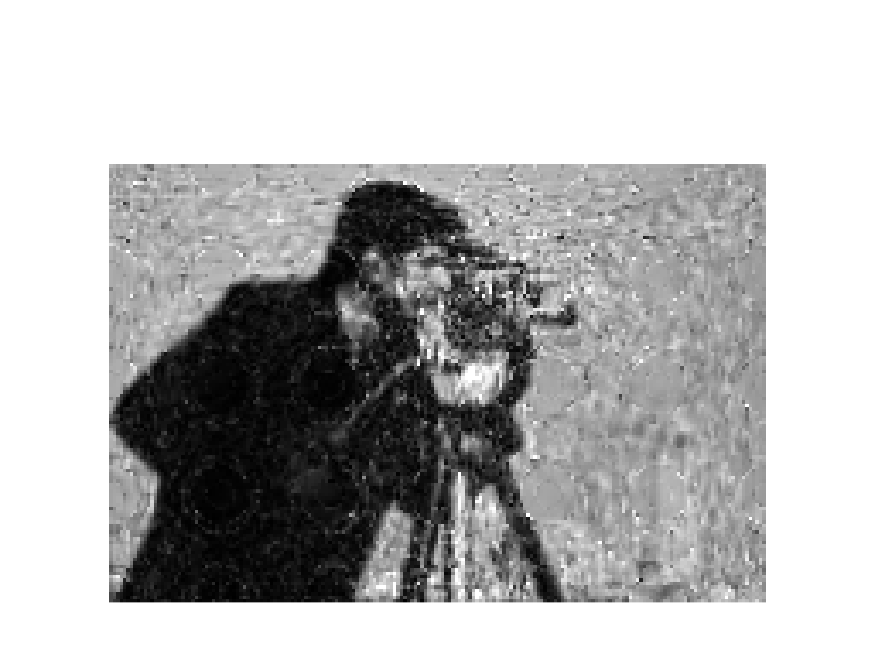}}
			&\subfigure{\includegraphics[width=.18\textwidth]{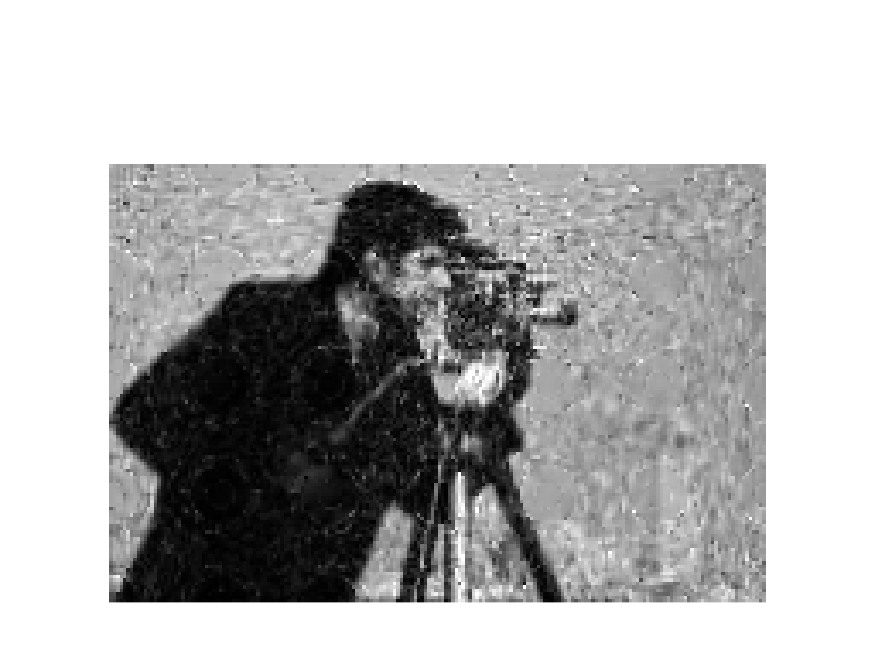}}
			\vspace{-.5in}
			\\
			\subfigure{\includegraphics[width=.18\textwidth]{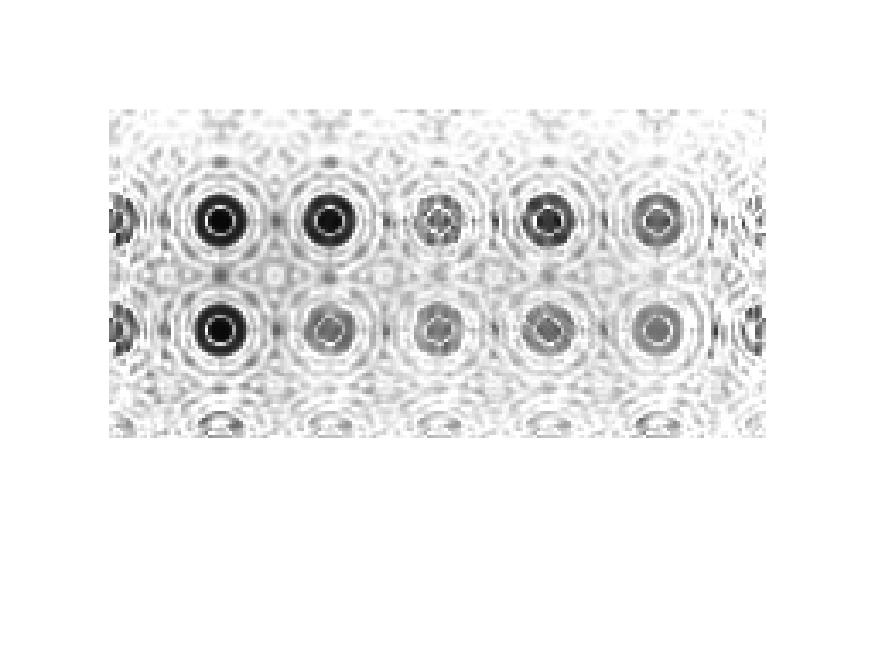}}
			&\subfigure{\includegraphics[width=.18\textwidth]{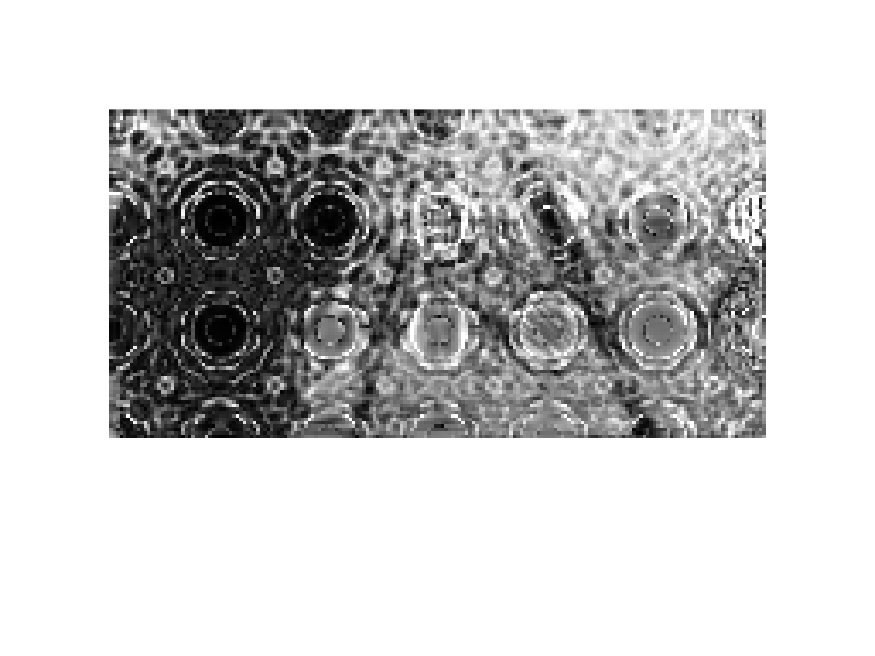}}
			&\subfigure{\includegraphics[width=.18\textwidth]{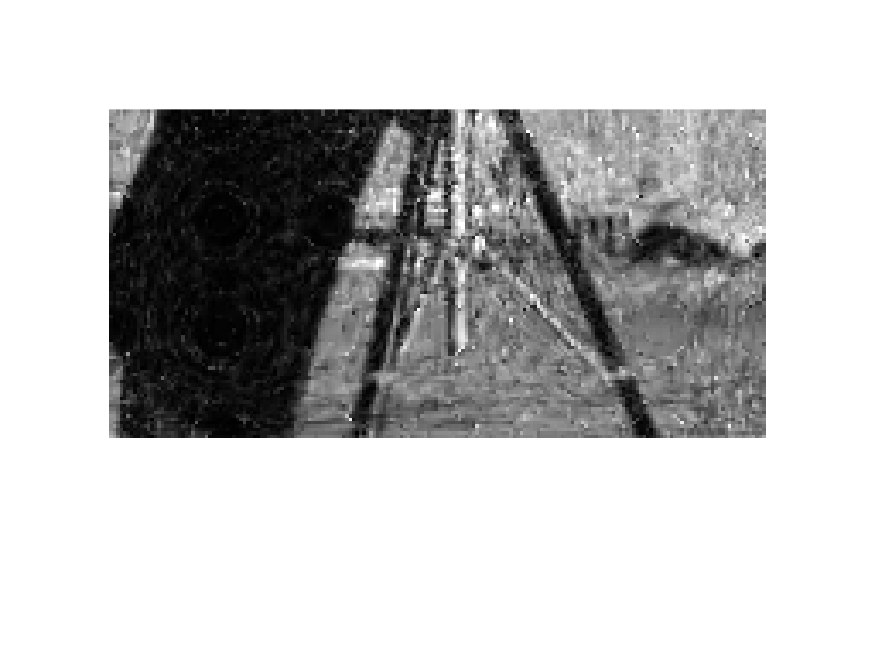}}
			&\subfigure{\includegraphics[width=.18\textwidth]{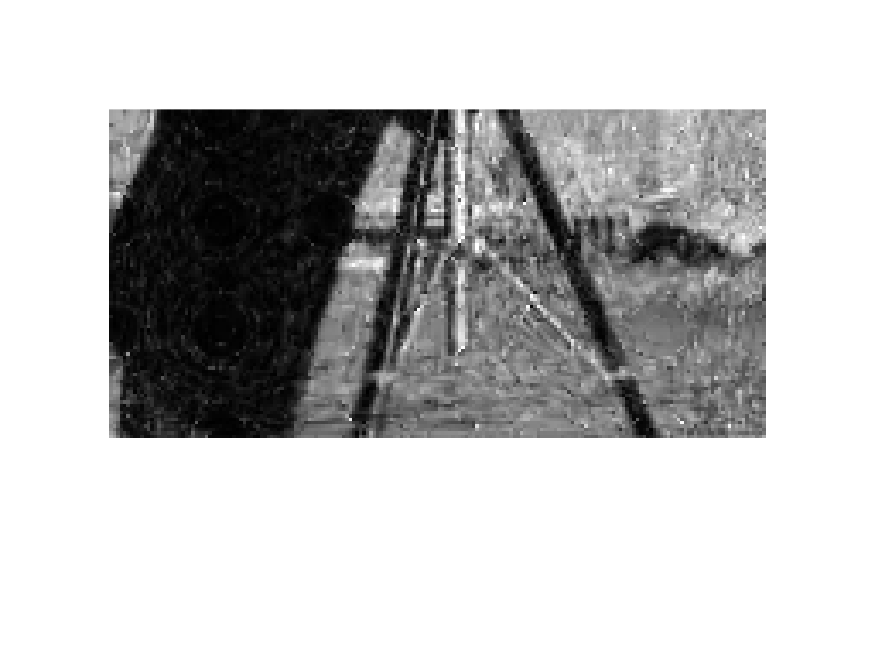}}
			\vspace{-.3in}
			\\
			\subfigure{\includegraphics[width=.18\textwidth]{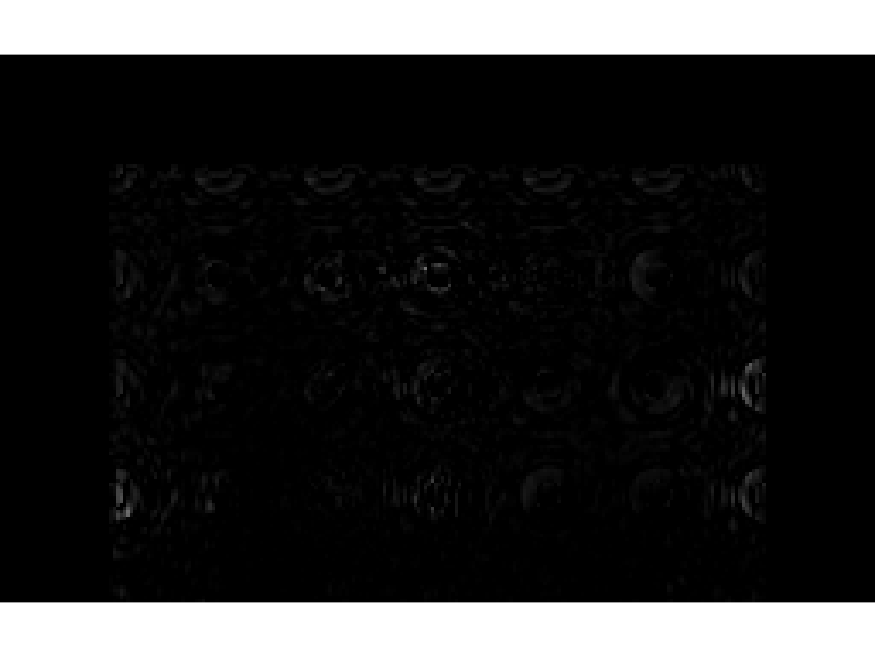}}
			&\subfigure{\includegraphics[width=.18\textwidth]{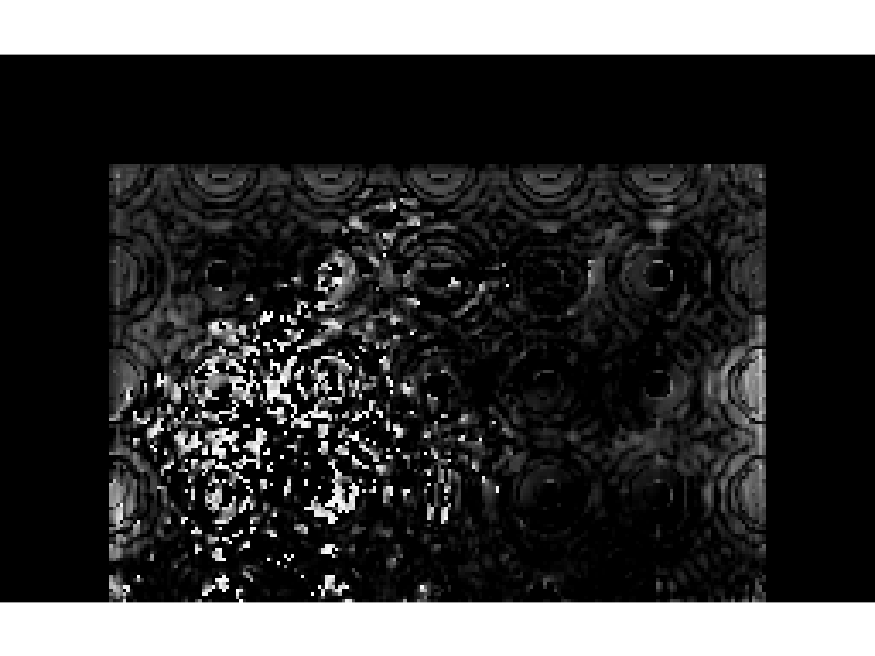}}
			&\subfigure{\includegraphics[width=.18\textwidth]{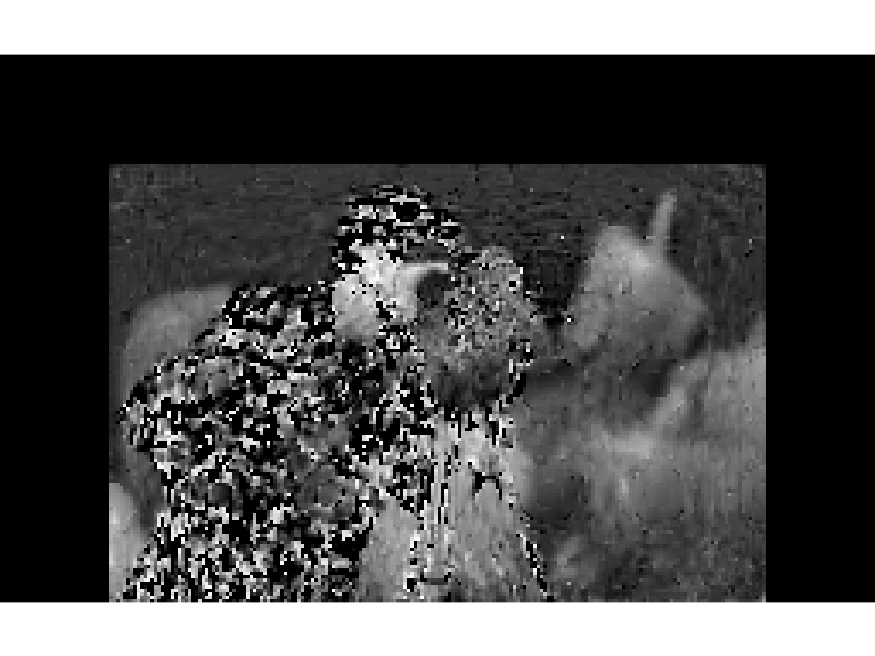}}
			&\subfigure{\includegraphics[width=.18\textwidth]{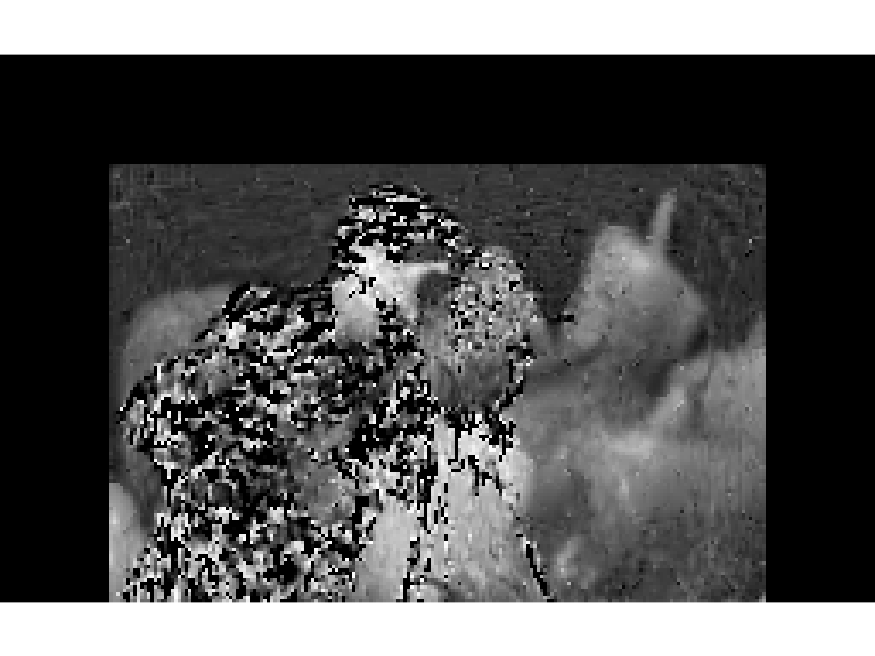}}
			\vspace{-.5in}
			\\
			\subfigure{\includegraphics[width=.18\textwidth]{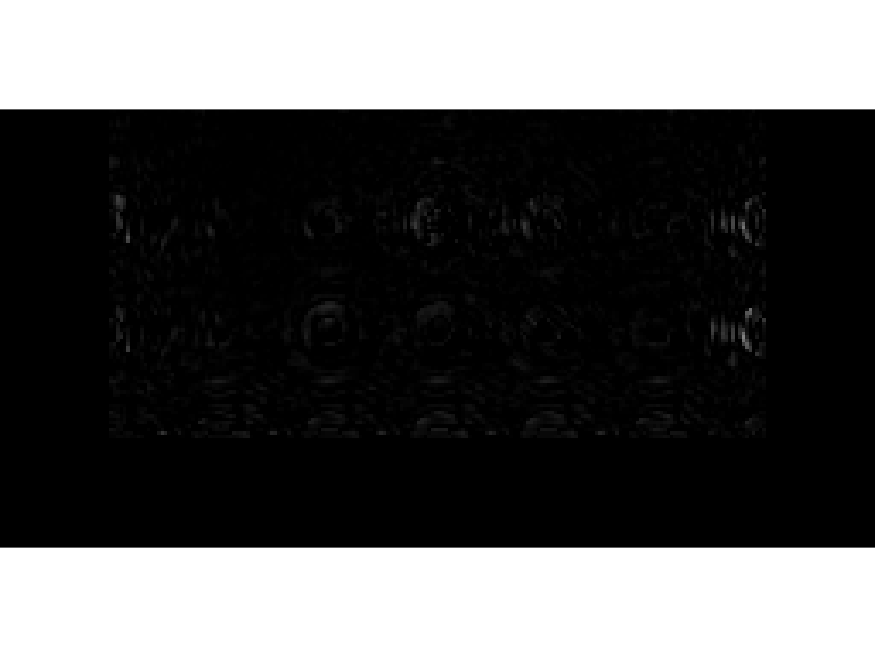}}
			&\subfigure{\includegraphics[width=.18\textwidth]{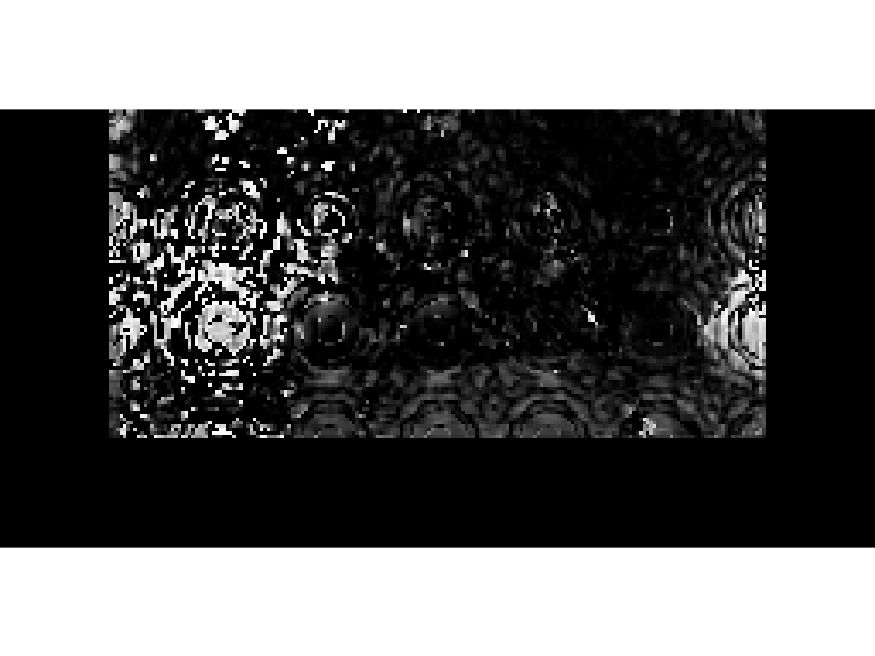}}
			&\subfigure{\includegraphics[width=.18\textwidth]{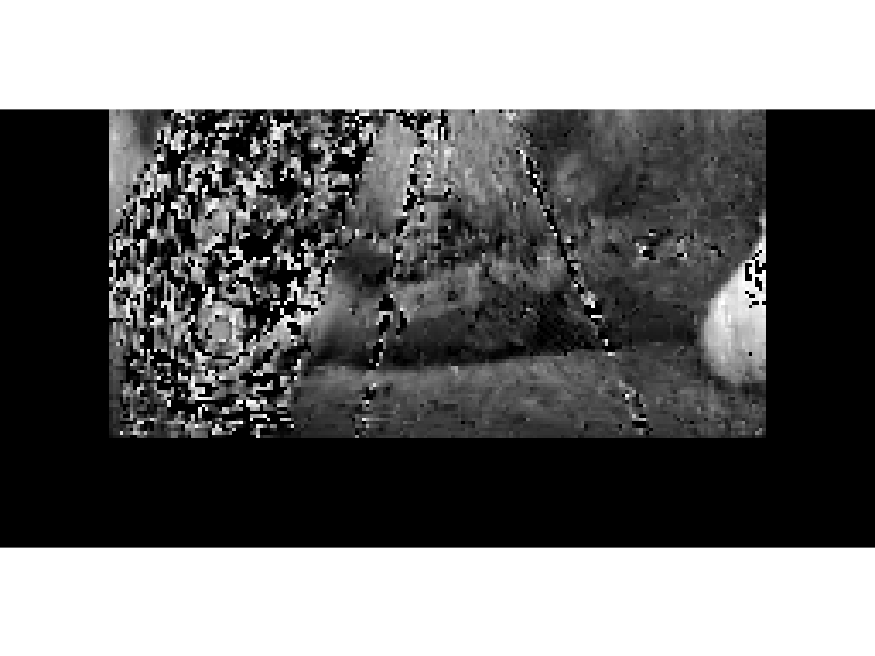}}
			&\subfigure{\includegraphics[width=.18\textwidth]{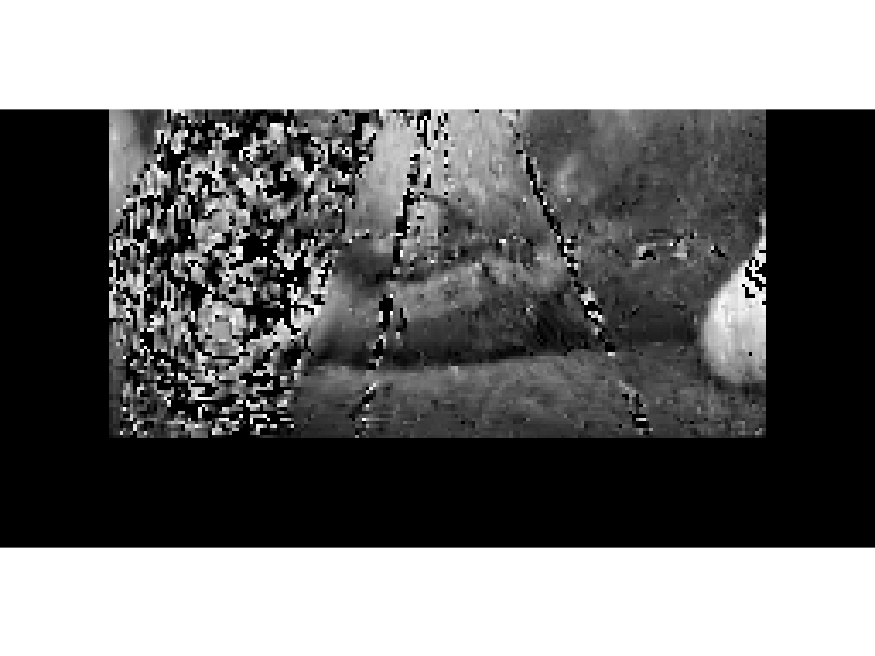}}\\
			\vspace{-.1in}
			(e) 1st & (f) 5th &(g) 50th&(h) 200th
		\end{tabular}
	\end{center}
	\caption{Recovery results from noisy measurements at 1st, 5th, 50th, 200th (final iteration) iterations. Top four rows for noisy measurements with SNR${}_{intensity}=39.8$: absolute parts (1st-2nd rows) and phase parts (3rd-4th rows) of recovered results on subdomains; Four rows below for noisy measurements with SNR${}_{intensity}=29.9$: absolute parts (5th-6th rows) and phase parts (7th-8th rows) of recovered results on subdomains.
		They are shown in the range of $[0,1]$
		and $[0,\pi]$ for the absolute and phase parts respectively.}
	\label{fig3-3}
\end{figure}

\subsection{Extended Tests}
To further test the performances of proposed DDMs, we will test more cases, including  blind recovery with two subdomains and the nonblind recovery with multiple subdomains. 
\vskip .02in
{\noindent\bf {(1) Blind recovery}}\\
Initial guess of the probe is generated by $w^0:=\tfrac{1}{J}\sum_j\mathcal F^*{\sqrt{f_j}}$  as \cite{chang2018Blind}.  Here we will show the performance of OD$^2$BP for the blind case with two-subdomain decomposition, { where same DD is considered as  section \ref{subsec-1}).}
We put the recovery results including the absolute parts of recovered probe, the absolute and phase parts of recovered images to Fig. \ref{fig4-1} at the 1st, 2nd, 5th, 50th, and the final iterations. 
Similarly to the nonblind case, the mismatch between the iterative solutions on the two subdomains gradually disappears and furthermore, the  probe is well recovered meanwhile, demonstrating that the proposed algorithm also works well for the blind case.

\begin{figure}
\begin{center}
\subfigure{\includegraphics[width=.18\textwidth]{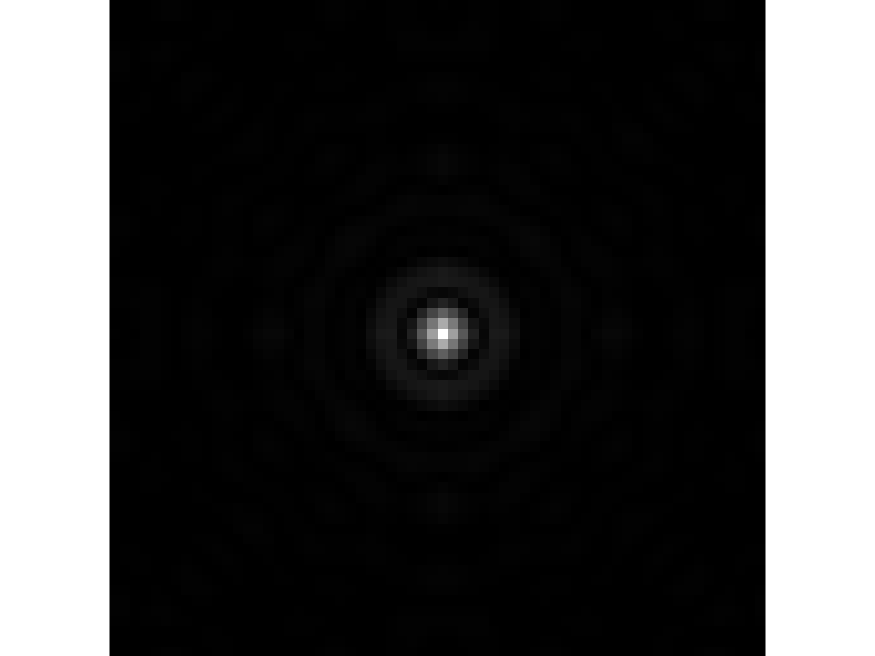}}
\subfigure{\includegraphics[width=.18\textwidth]{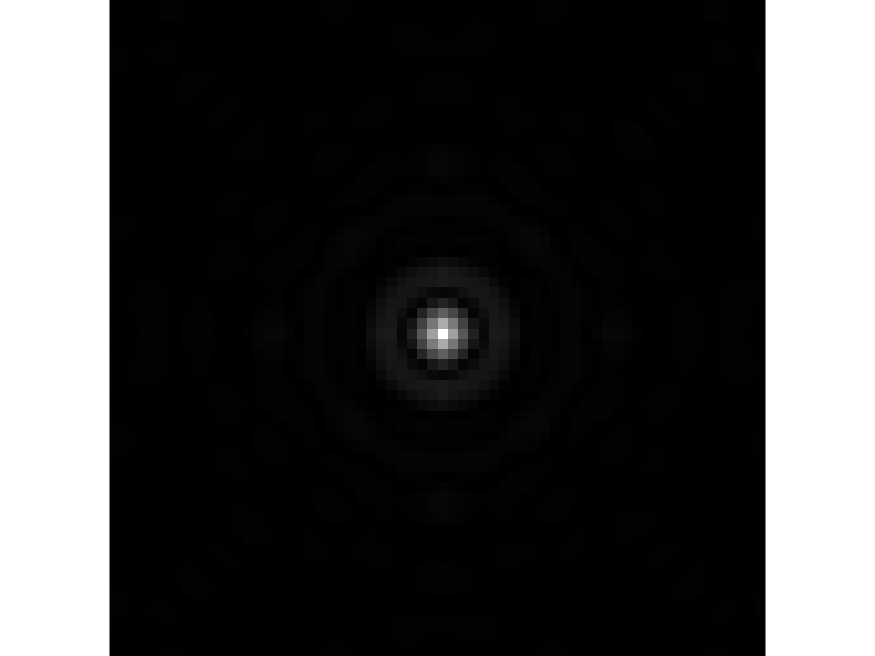}}
\subfigure{\includegraphics[width=.18\textwidth]{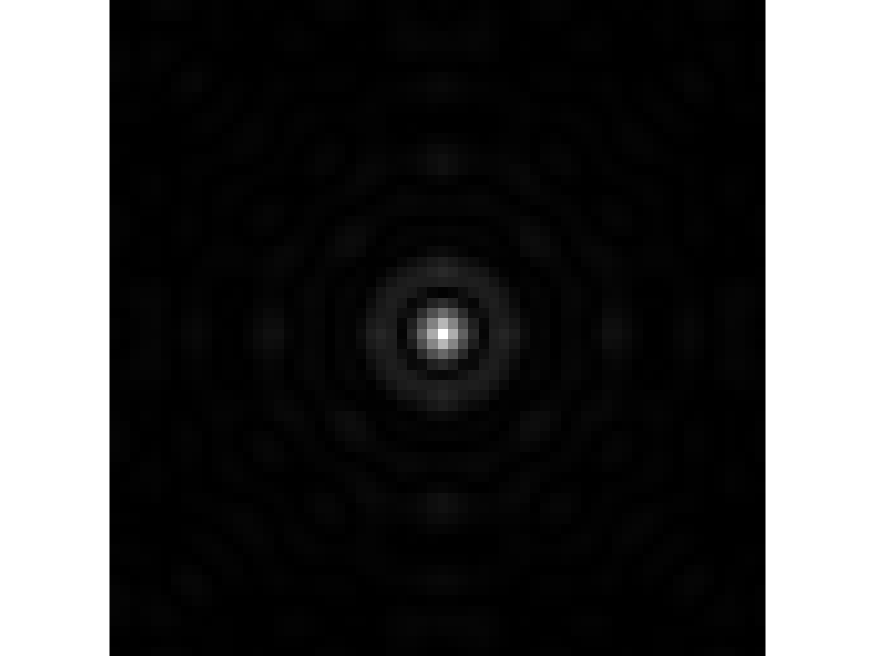}}
\subfigure{\includegraphics[width=.18\textwidth]{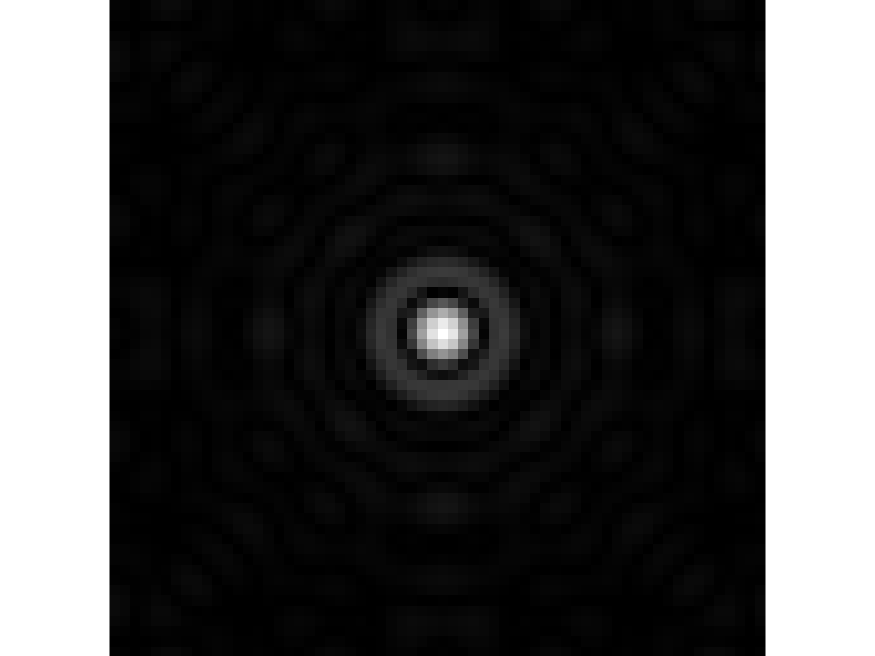}}
\subfigure{\includegraphics[width=.18\textwidth]{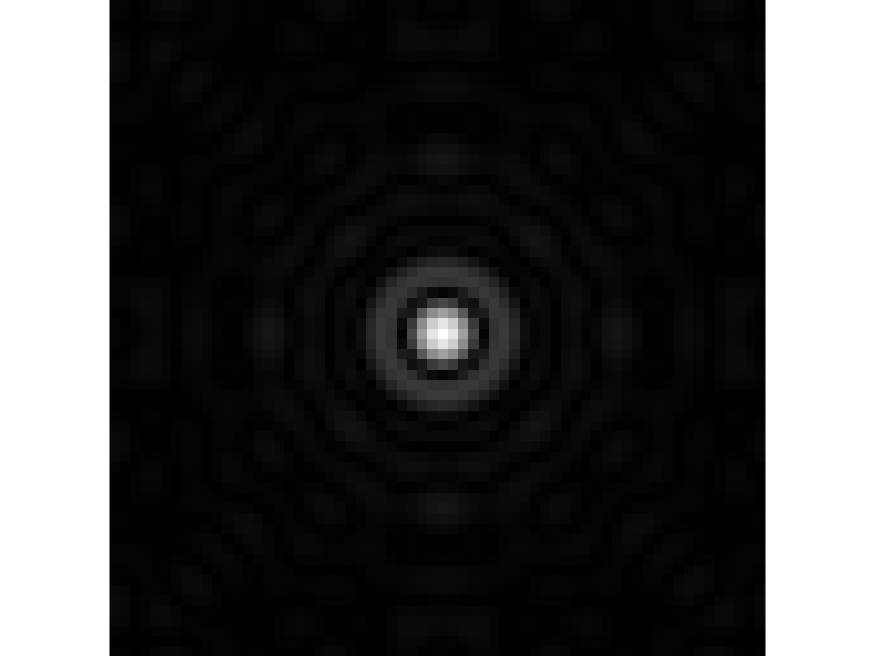}}
\\
\subfigure{\includegraphics[width=.18\textwidth]{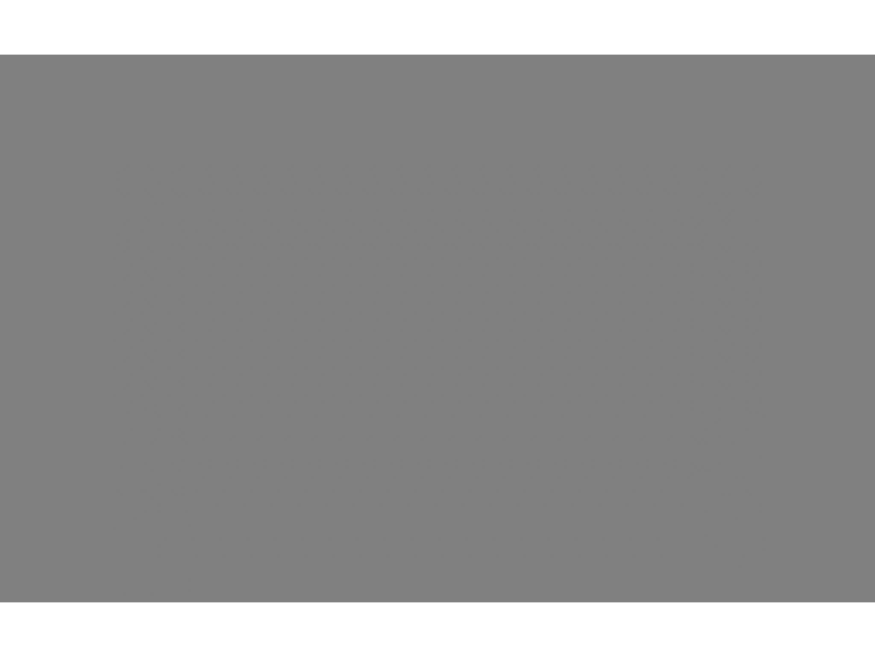}}
\subfigure{\includegraphics[width=.18\textwidth]{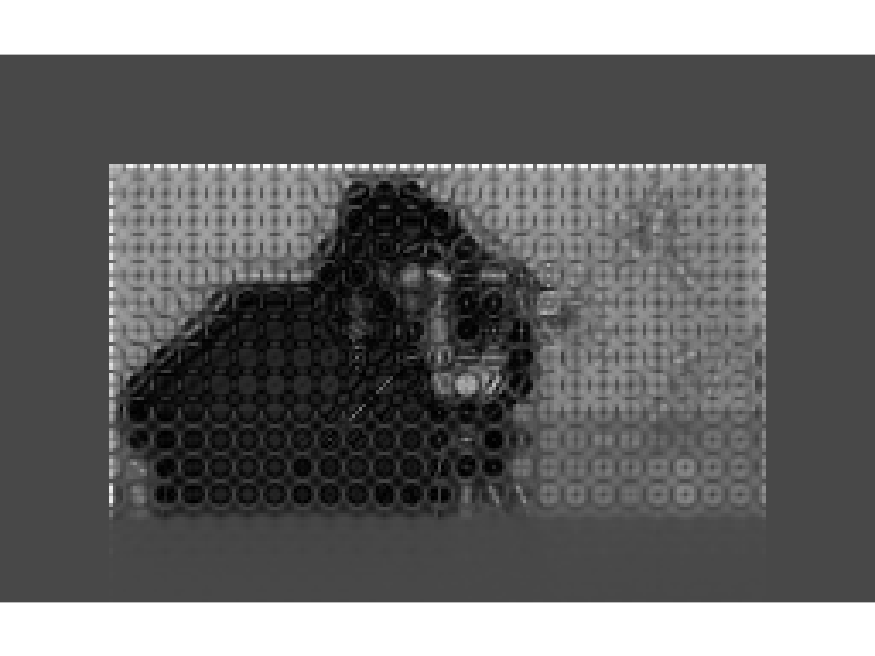}}
\subfigure{\includegraphics[width=.18\textwidth]{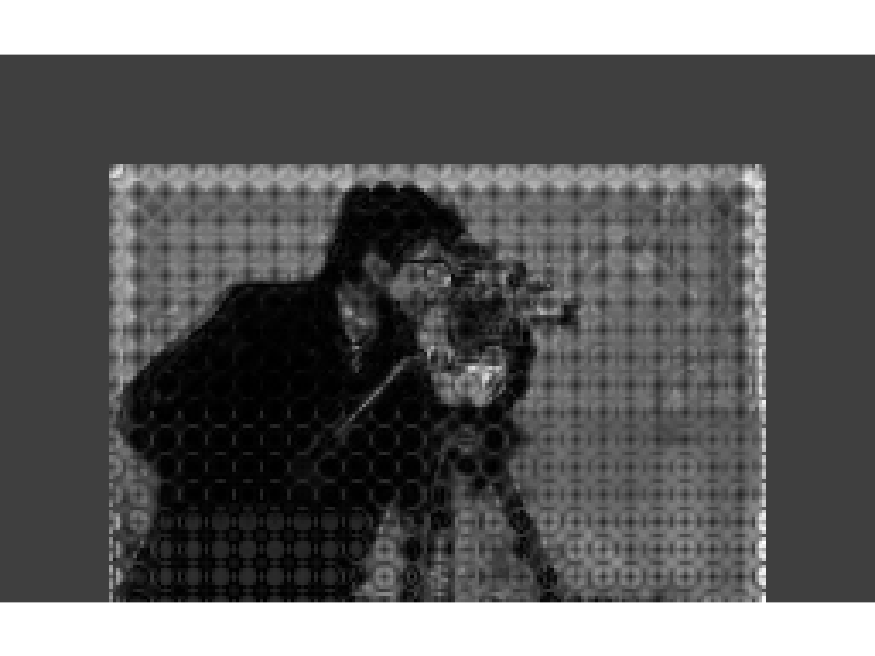}}
\subfigure{\includegraphics[width=.18\textwidth]{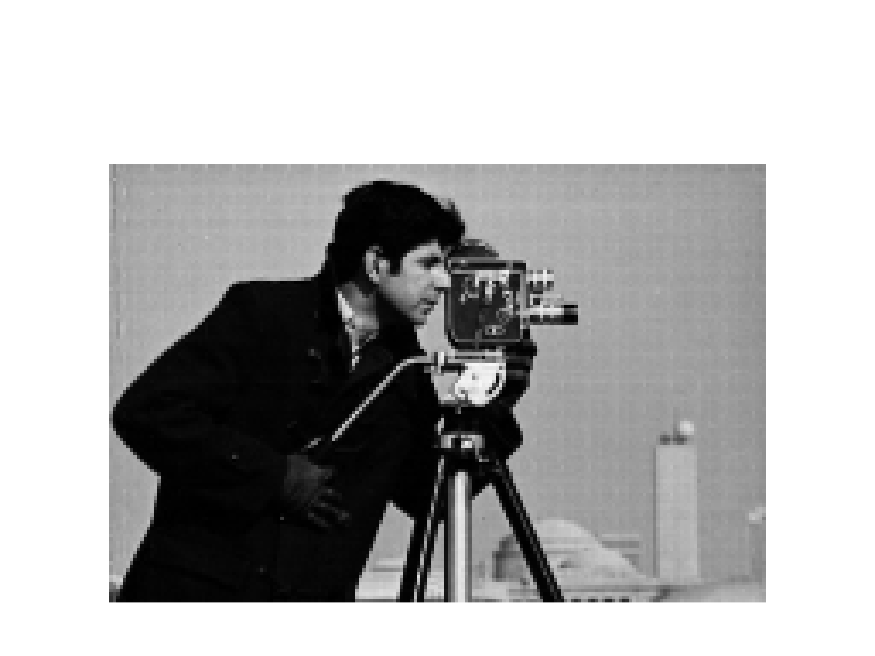}}
\subfigure{\includegraphics[width=.18\textwidth]{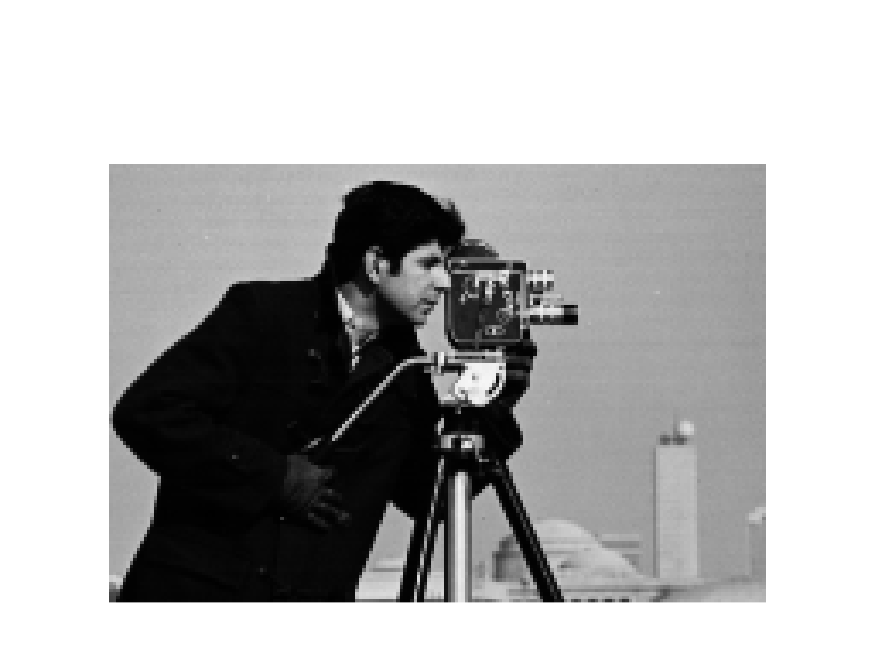}}
\vskip -.2in
\subfigure{\includegraphics[width=.18\textwidth]{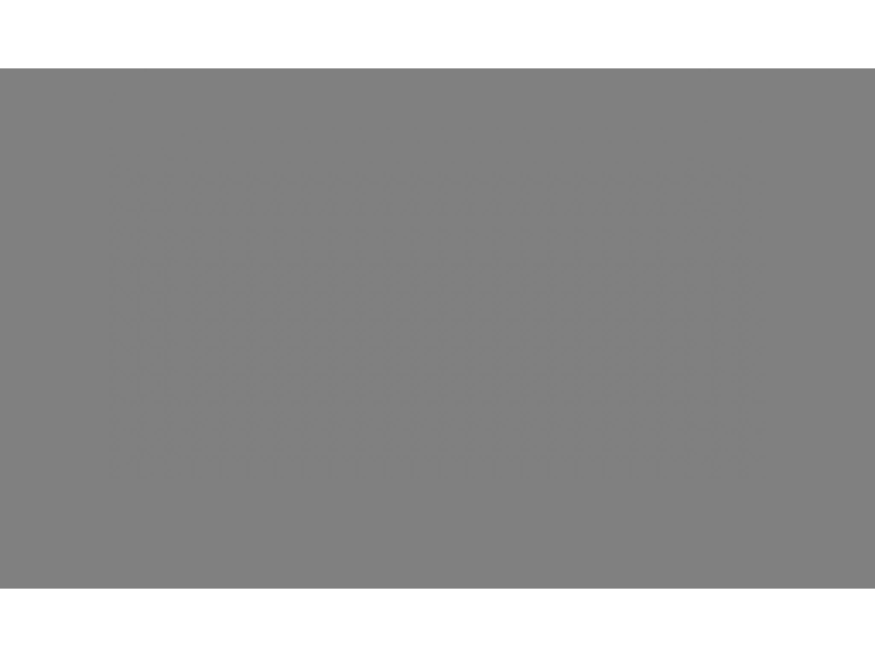}}
\subfigure{\includegraphics[width=.18\textwidth]{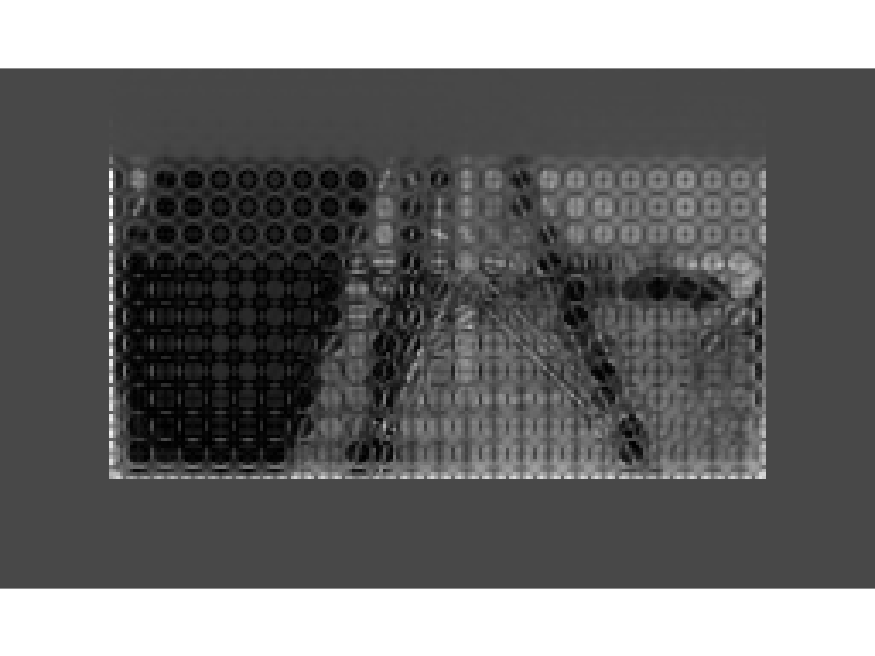}}
\subfigure{\includegraphics[width=.18\textwidth]{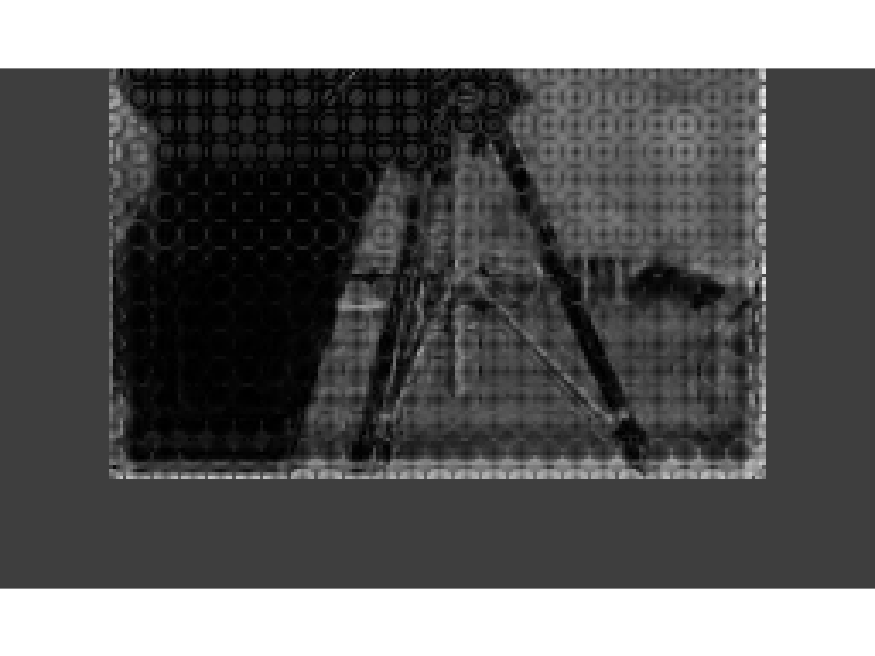}}
\subfigure{\includegraphics[width=.18\textwidth]{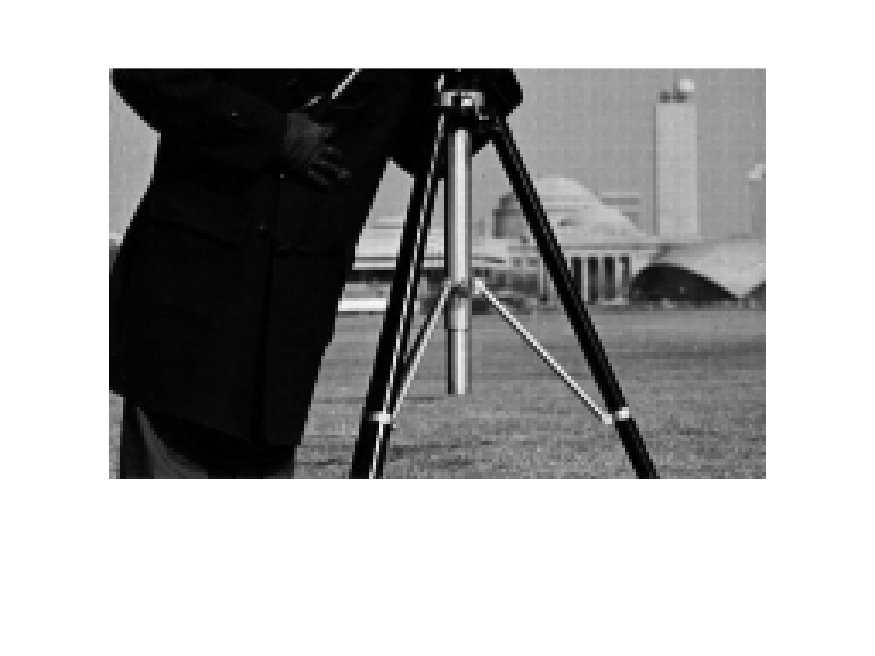}}
\subfigure{\includegraphics[width=.18\textwidth]{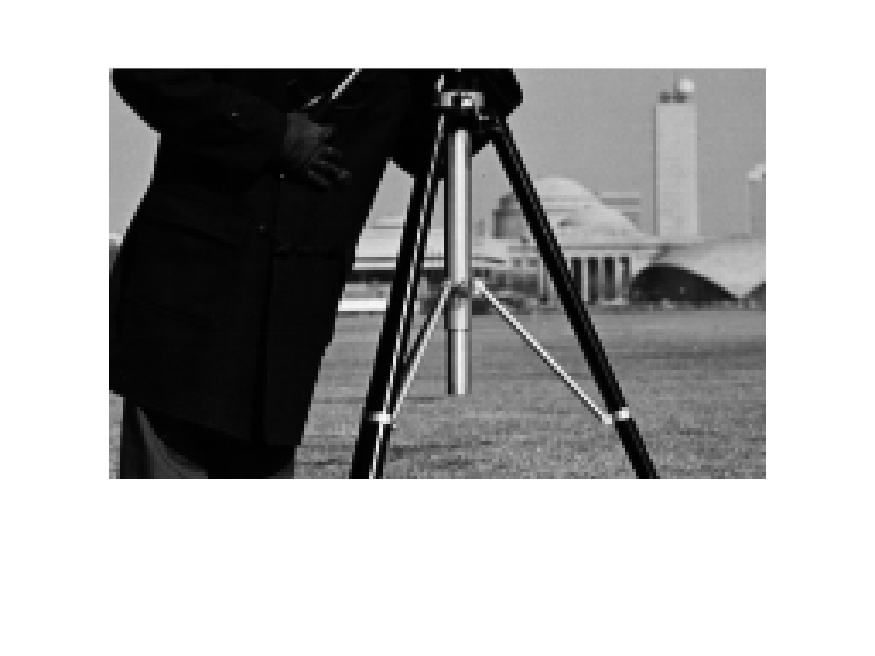}}
\\
\subfigure{\includegraphics[width=.18\textwidth]{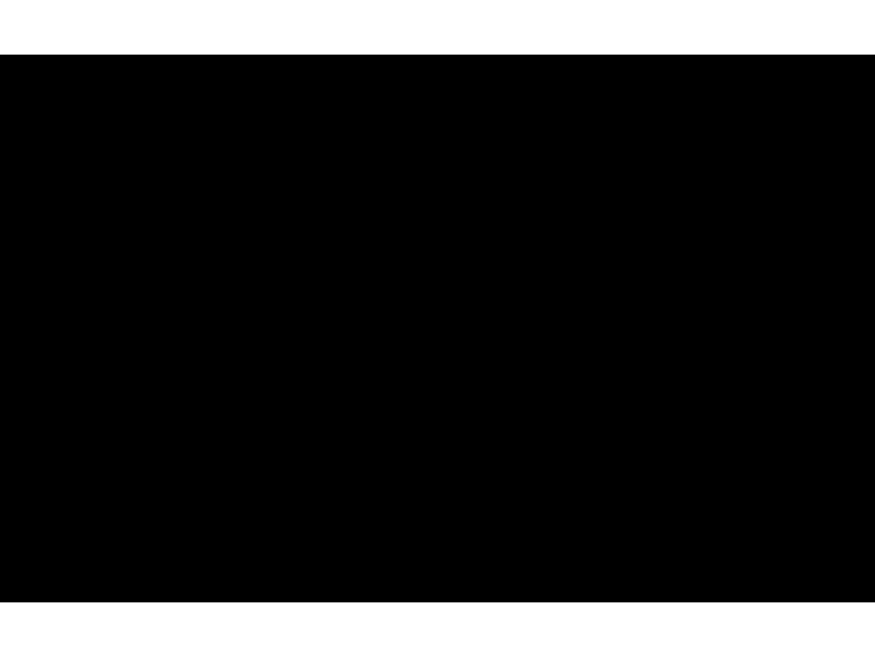}}
\subfigure{\includegraphics[width=.18\textwidth]{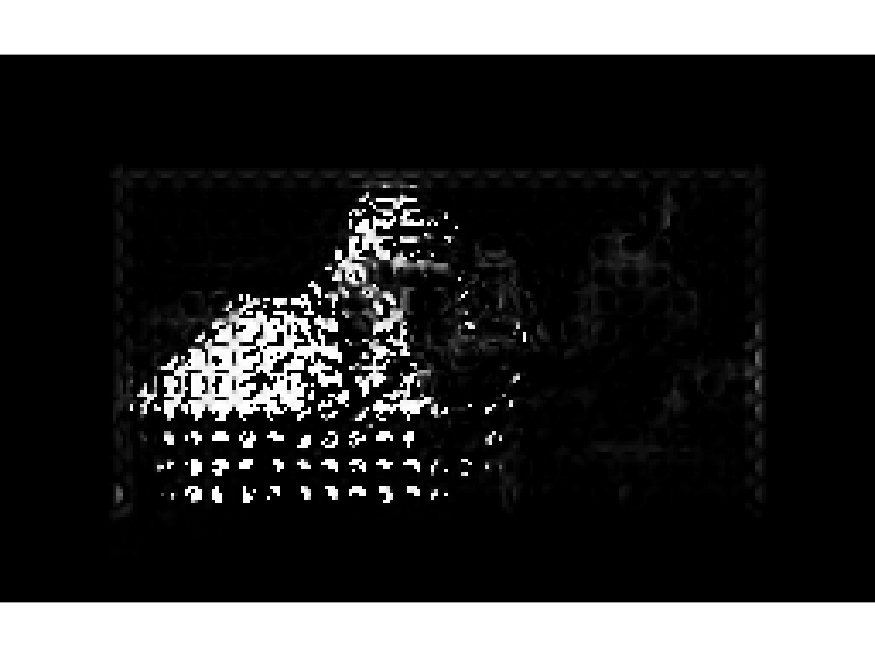}}
\subfigure{\includegraphics[width=.18\textwidth]{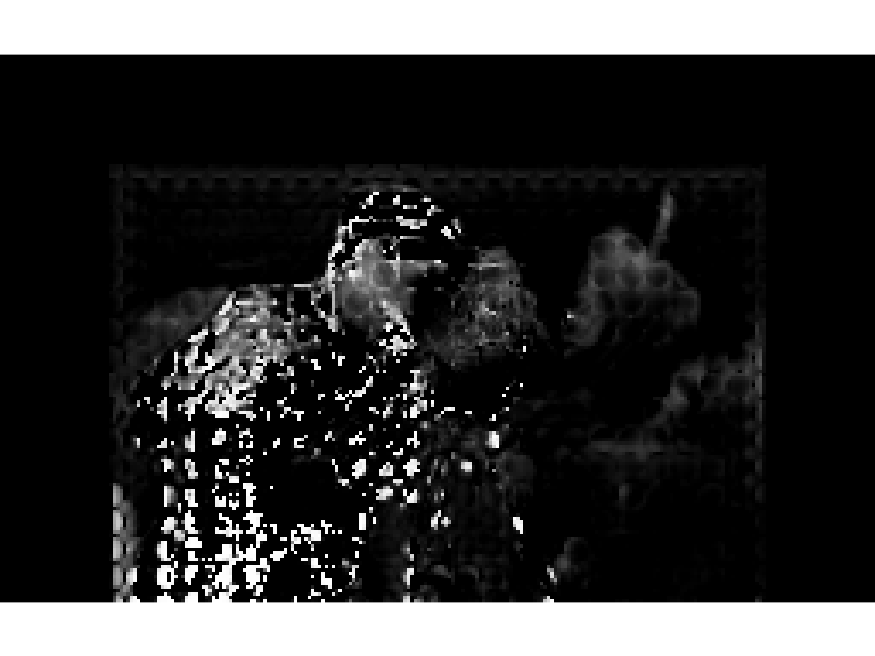}}
\subfigure{\includegraphics[width=.18\textwidth]{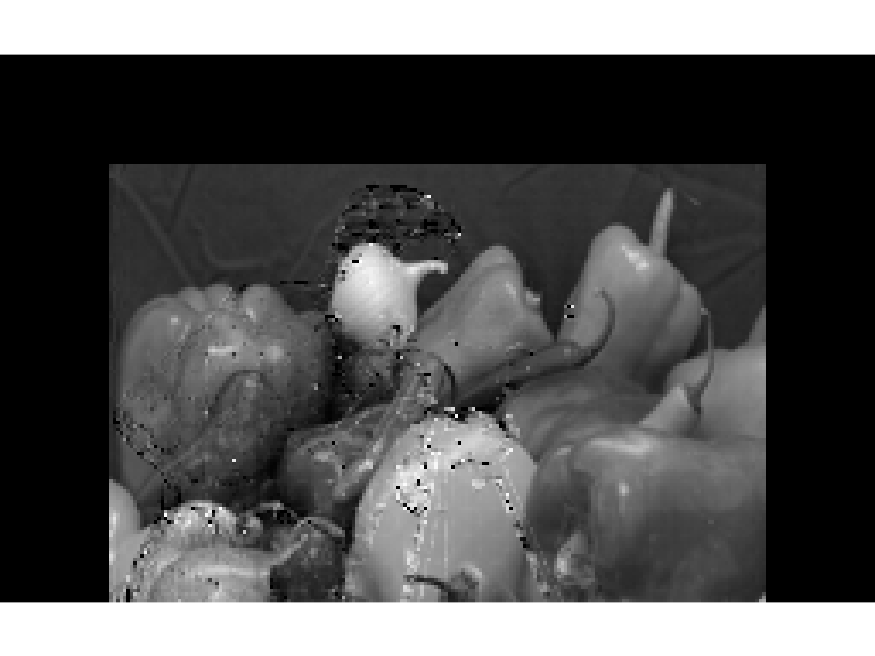}}
\subfigure{\includegraphics[width=.18\textwidth]{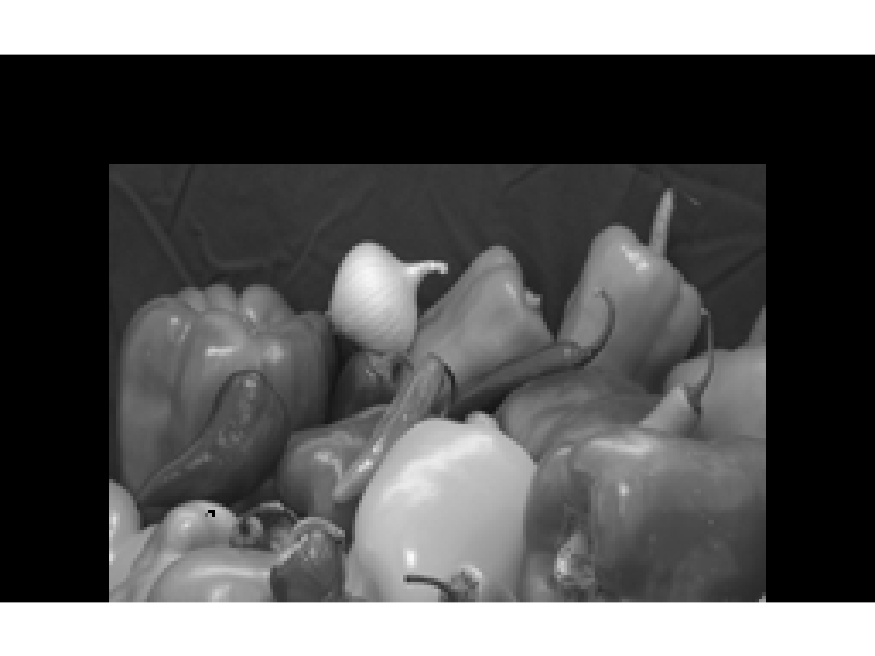}}
\vskip -.2in
\setcounter{subfigure}{0}
\subfigure[1st]{\includegraphics[width=.18\textwidth]{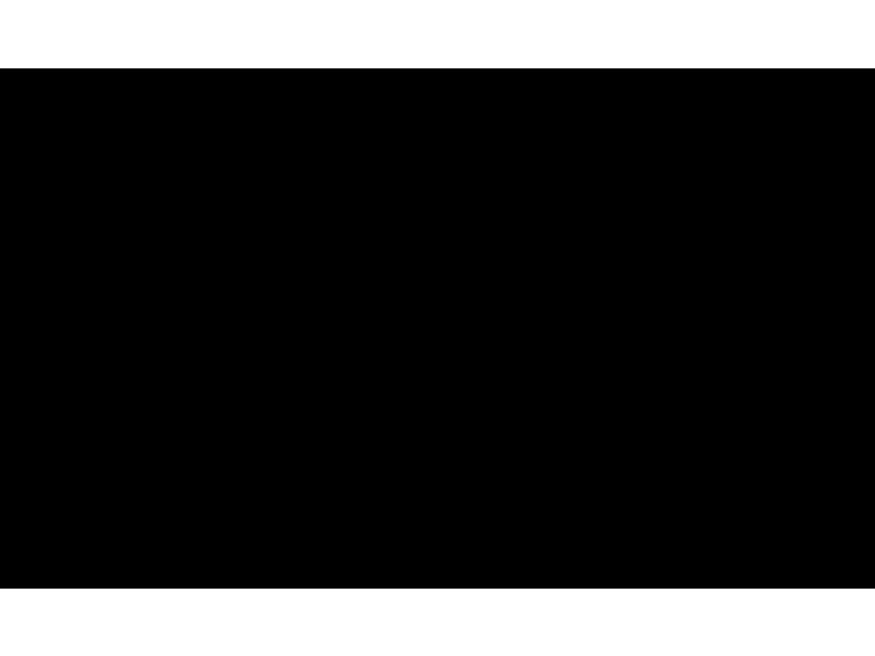}}
\subfigure[2nd]{\includegraphics[width=.18\textwidth]{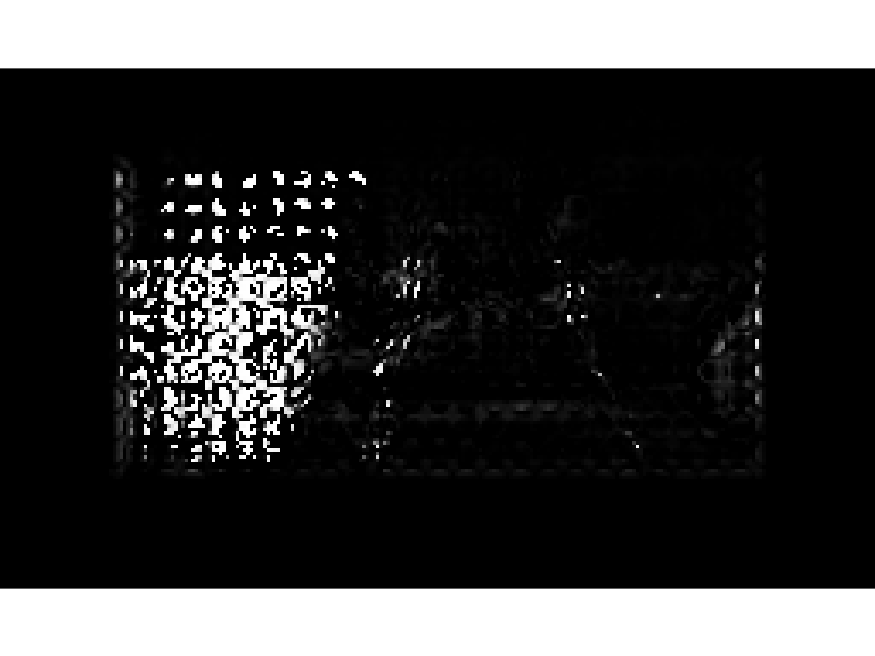}}
\subfigure[5th]{\includegraphics[width=.18\textwidth]{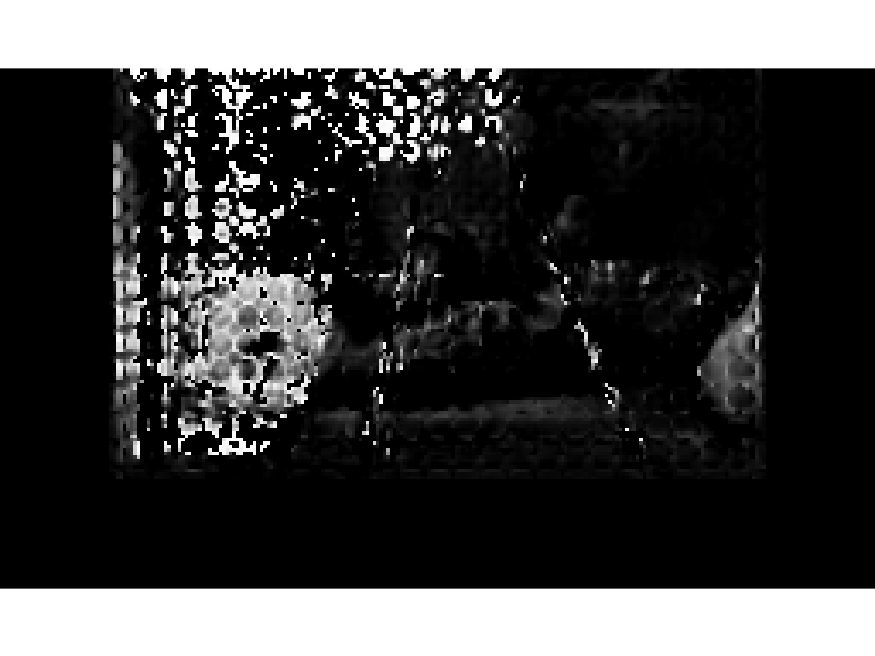}}
\subfigure[50th]{\includegraphics[width=.18\textwidth]{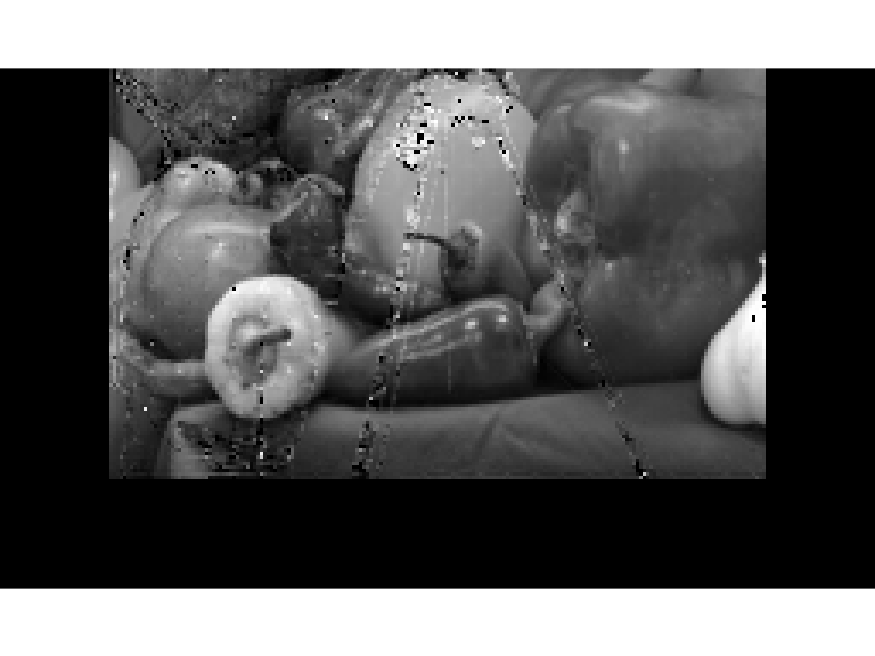}}
\subfigure[211th]{\includegraphics[width=.18\textwidth]{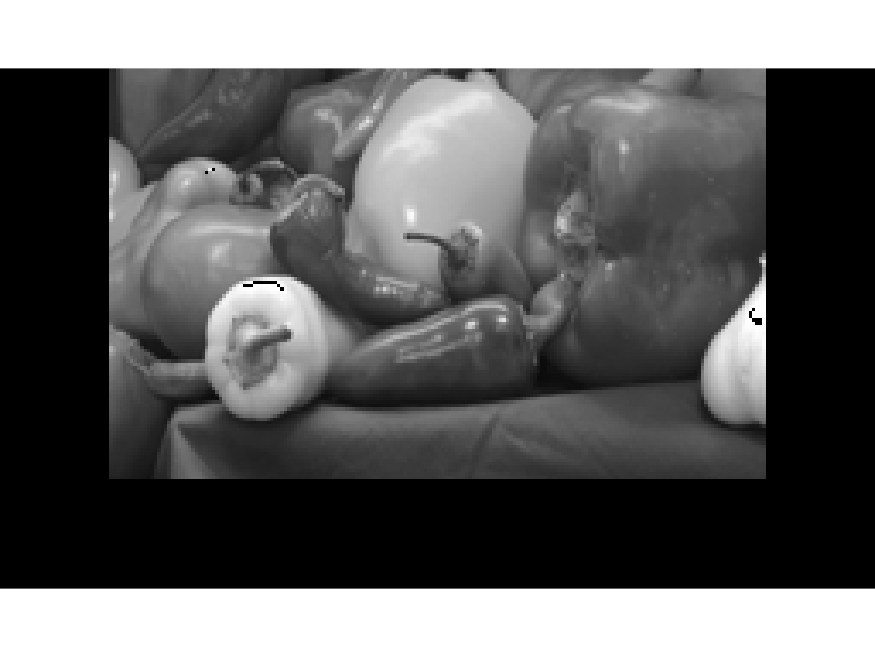}}
\end{center}
    \caption{Blind recovery results  at 1st, 2nd, 5th, 50th, 211th (final iteration) iterations. The first and second rows: Recovered absolute parts (top and down parts for decomposed samples) and phase parts (top and down parts for decomposed samples) respectively, shown in the range of $[0,1000]$, $[0,1]$
    and $[0,\pi]$ for the first, 2nd-3rd, and last two rows, respectively. The parameters  set to
$\eta=0.1, r=5.0\times 10^3, \mu=2.0\times 10^2.$ }
\label{fig4-1}
\end{figure}

\vskip .02in

{\noindent\bf {{(2) Multiple subdomains  } }}\\
Here we will show the performance of OD$^2$P$_m$. For simplicity,  the stripe-type multiple DD is implemented.   To get a relative large problem, we consider the    image with $512\times 512$ pixels (Interpolated of the image shown in Fig. \ref{fig2}), and the scan step size sets to 16 pixels. { The width of the overlapping regions is 48 (pixels).}  The parameters set to the same as the case of two-subdomain.
We first put the recovery results in the case of 4 subdomains to Fig. \ref{fig4-2}, immediately showing  the effectiveness of proposed algorithm in the case of multiple subdomains.  More tests are conducted with $D=6, 8, 10$ subdomains DD, and the related convergence curves are put to  Fig. \ref{fig4-3}. On one hand, one can readily infer that the proposed algorithm converges pretty well as numbers of subdomains vary. On the other hand, its convergence becomes slower as the number of subdomains increase. However, it seems not very sensitive to the number of subdomains { at least up to 10 subdomains}, since the maximum iteration number to reach the given accuracy changes a little bit (around 10\% more iterations for 10-subdomains than that for 2-subdomain) as reported in the following Table \ref{tab:my_label}.

We also measure the virtual wall-clock times ({The proposed   OD$^2$P$_m$ algorithm is assumed to be fully parallelized at the subdomain regardless of the communication time as \cite{lee2017primal}), and a single processor handles a single subdomain} in order to show the potential of parallel computing, and report them in Table \ref{tab:my_label}, that show high speedup ratios and the speedup efficiency (speedup ratios divided by the number of subdomains) is above $80\%$.  Although we only report the virtual wall-clock time regardless of the communication cost, the proposed algorithm should be suitable for parallel computing, due to the low cost of information exchange happening on the narrow overlapping regions of adjacent subdomains.   Hence, we will further evaluate the performance of proposed algorithms on  large-scale computer clusters, and leave it as future work.

\begin{figure}[]
\begin{center}
\subfigure{\includegraphics[width=.18\textwidth]{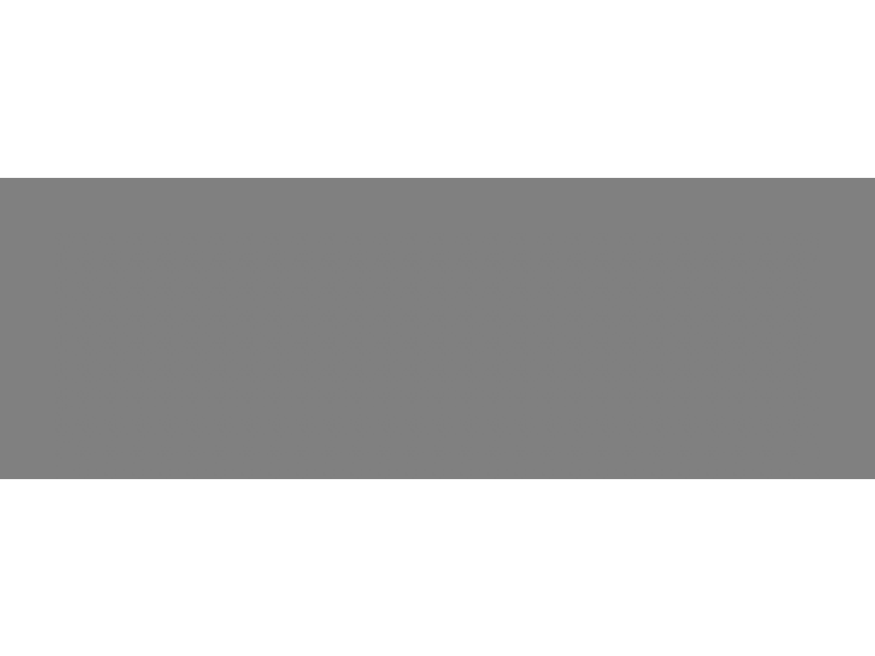}}
\subfigure{\includegraphics[width=.18\textwidth]{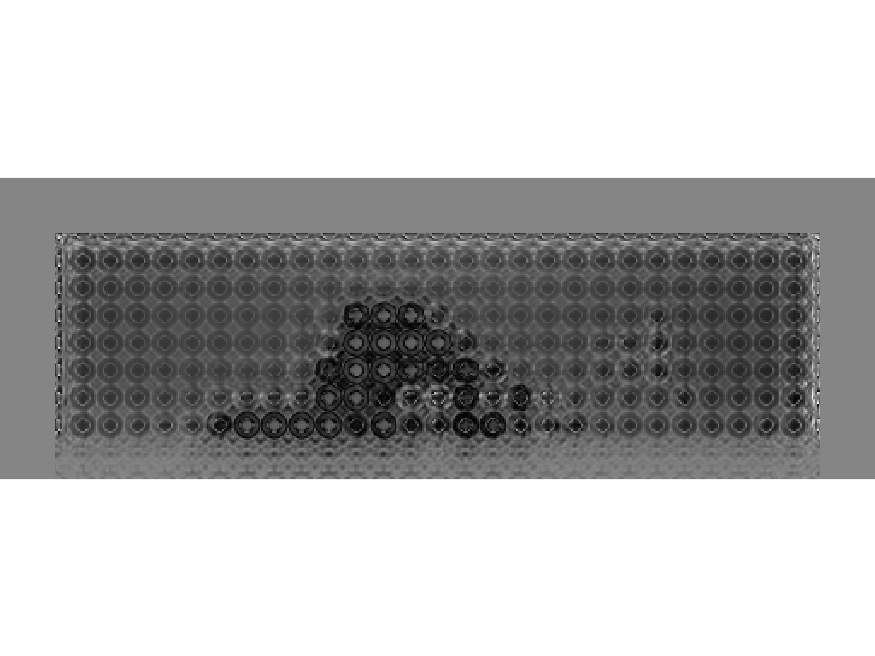}}
\subfigure{\includegraphics[width=.18\textwidth]{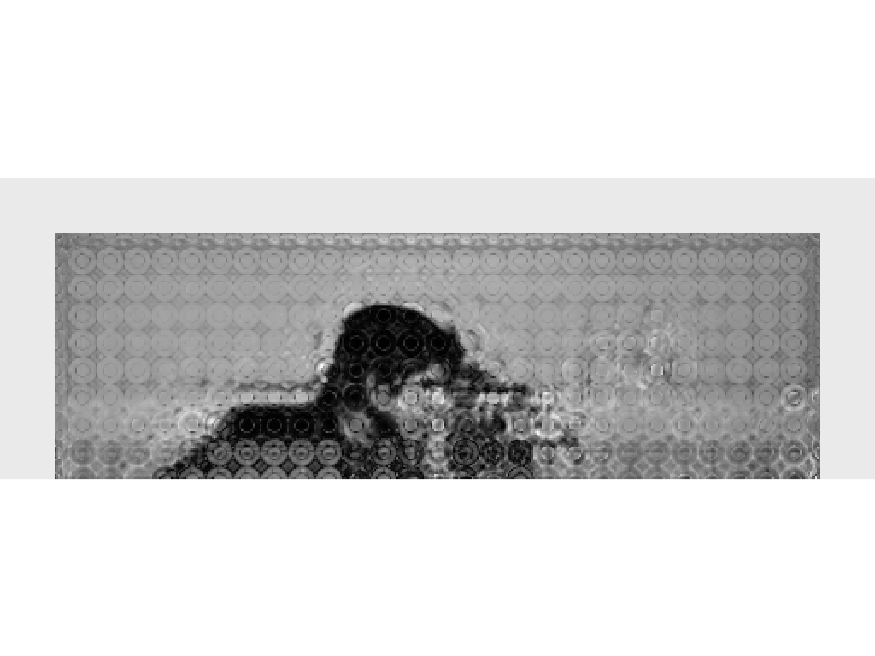}}
\subfigure{\includegraphics[width=.18\textwidth]{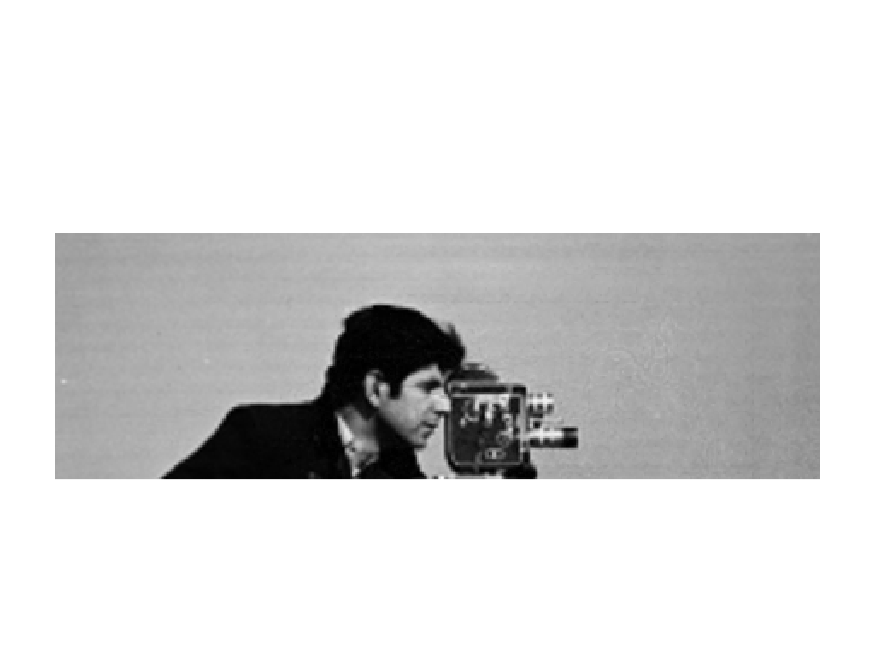}}
\subfigure{\includegraphics[width=.18\textwidth]{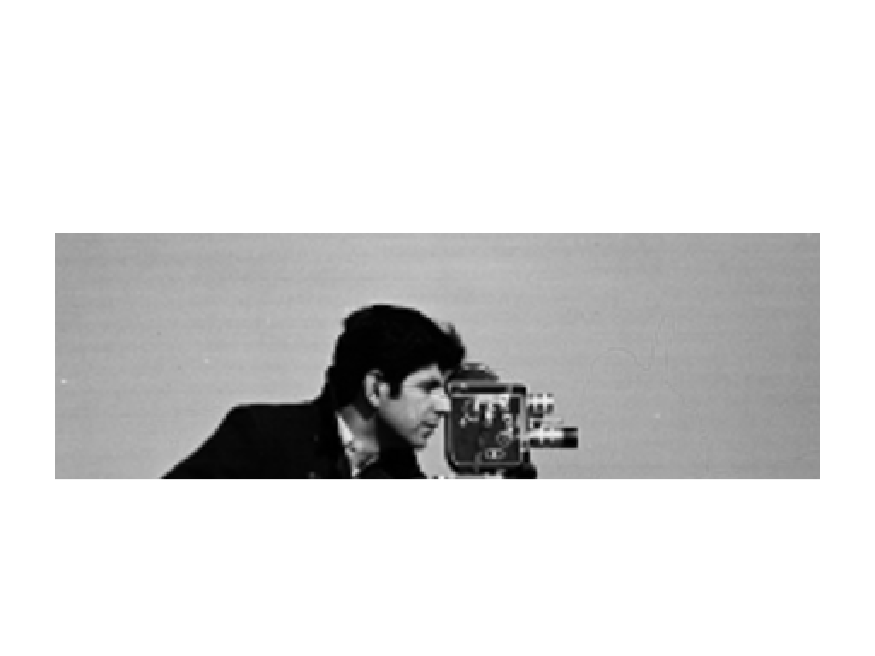}}
\vskip -.5in
\subfigure{\includegraphics[width=.18\textwidth]{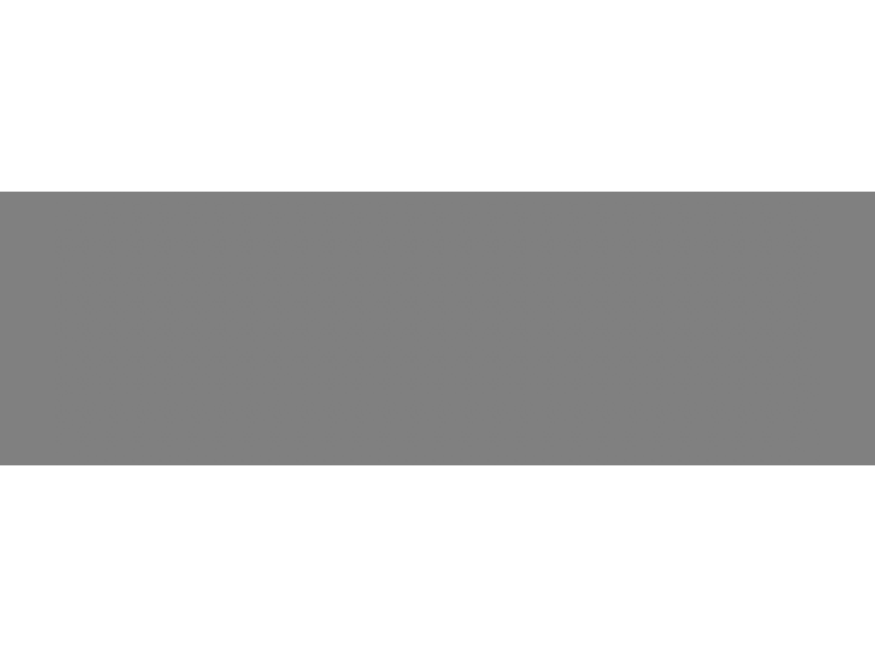}}
\subfigure{\includegraphics[width=.18\textwidth]{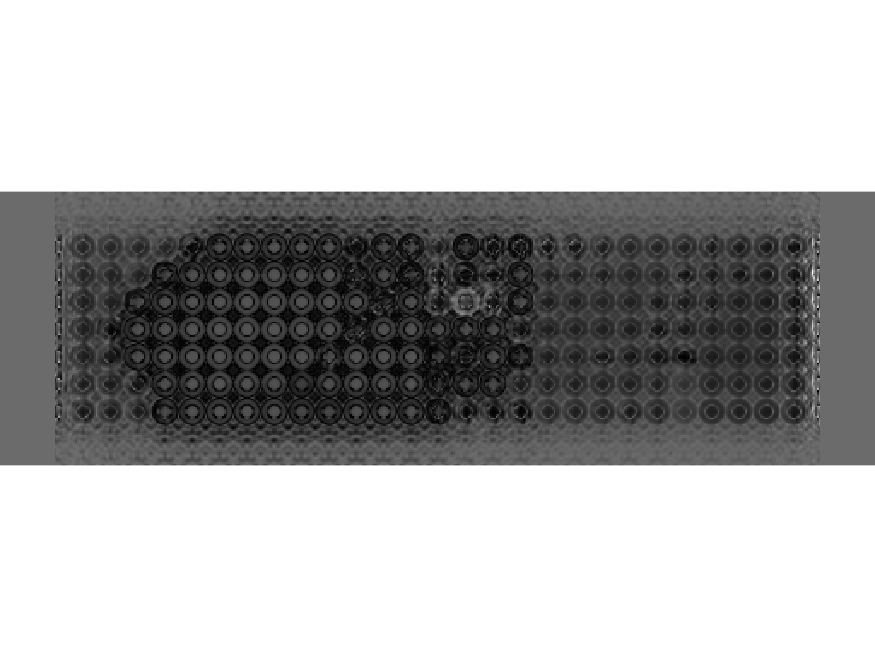}}
\subfigure{\includegraphics[width=.18\textwidth]{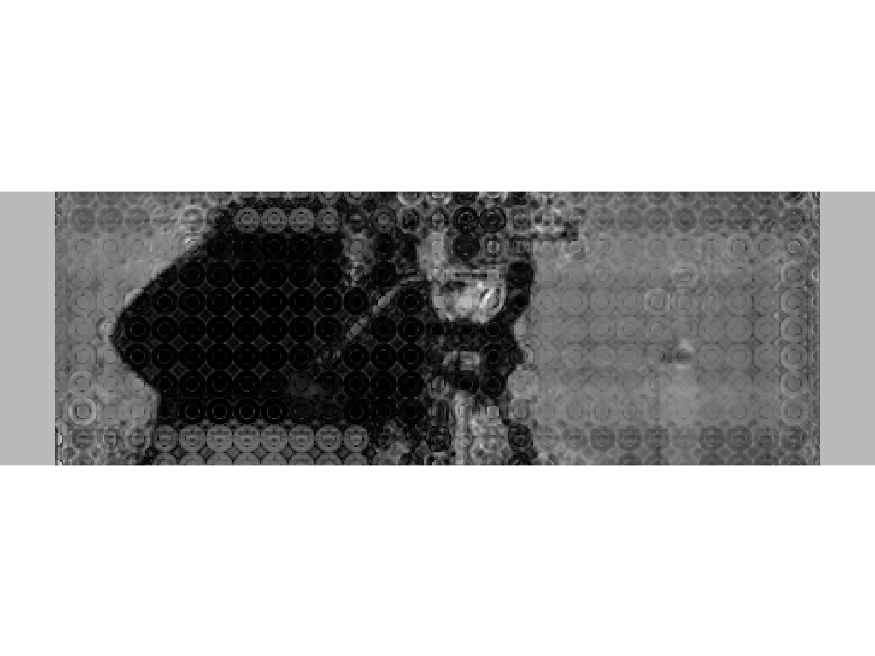}}
\subfigure{\includegraphics[width=.18\textwidth]{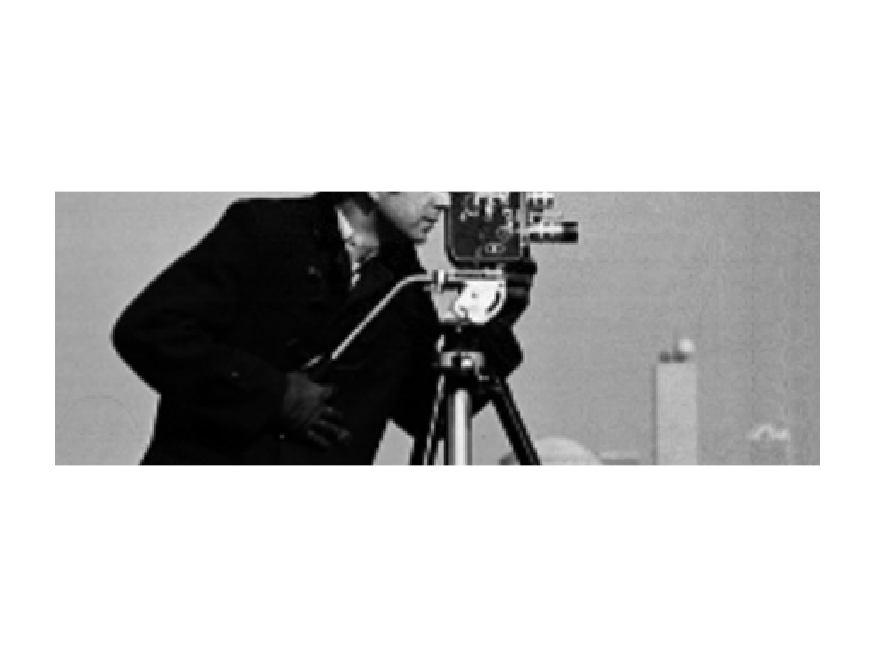}}
\subfigure{\includegraphics[width=.18\textwidth]{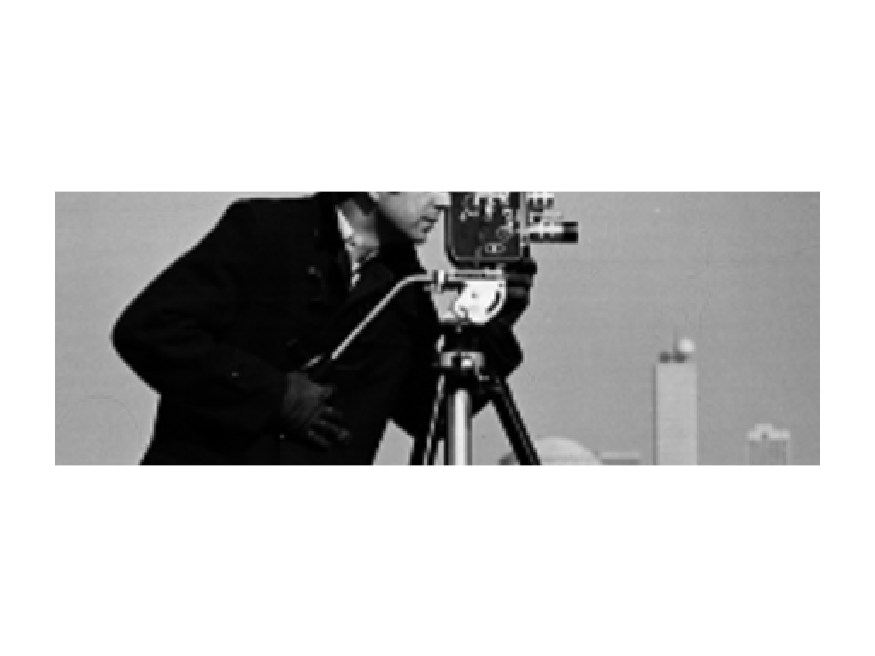}}
\vskip -.5in
\subfigure{\includegraphics[width=.18\textwidth]{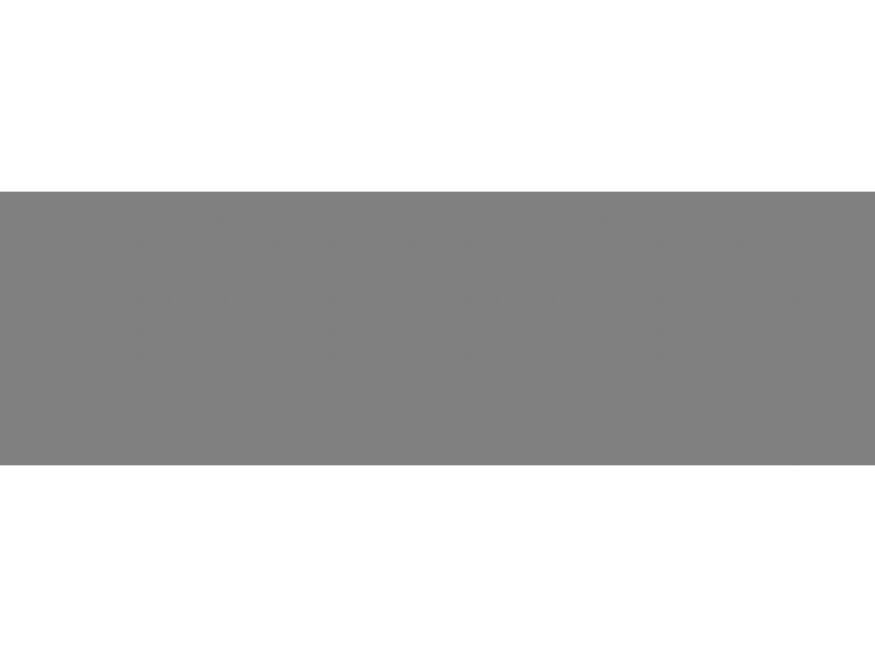}}
\subfigure{\includegraphics[width=.18\textwidth]{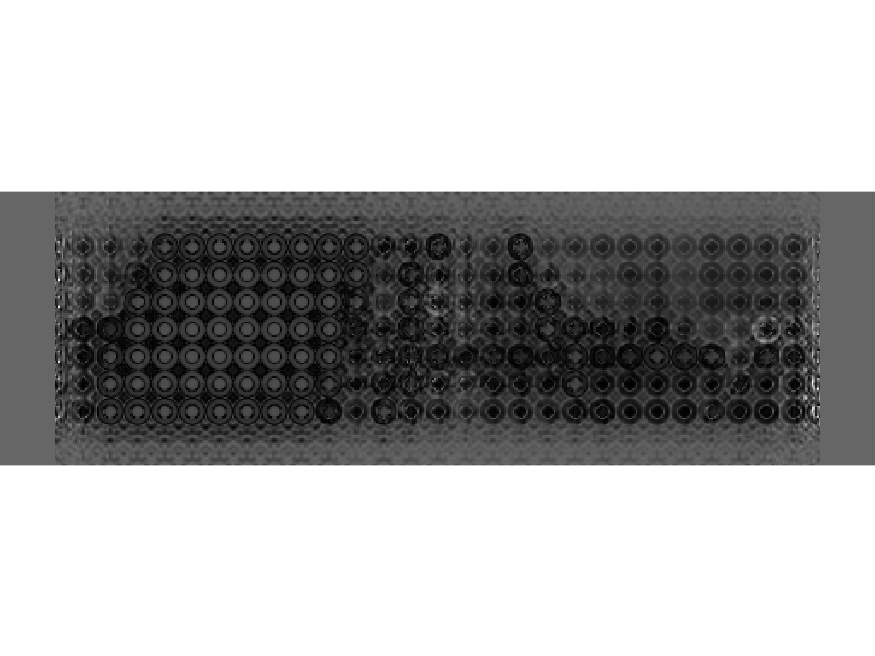}}
\subfigure{\includegraphics[width=.18\textwidth]{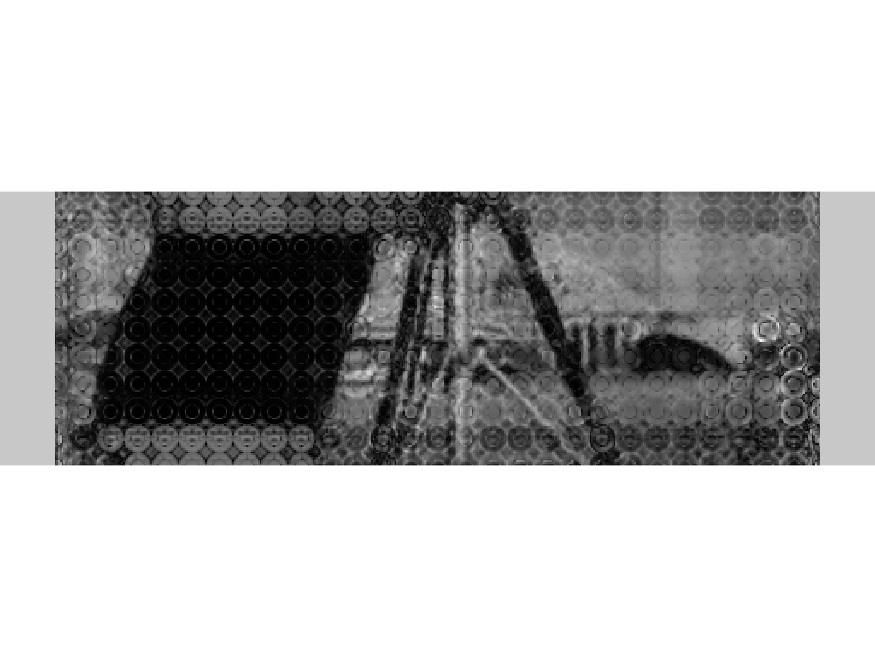}}
\subfigure{\includegraphics[width=.18\textwidth]{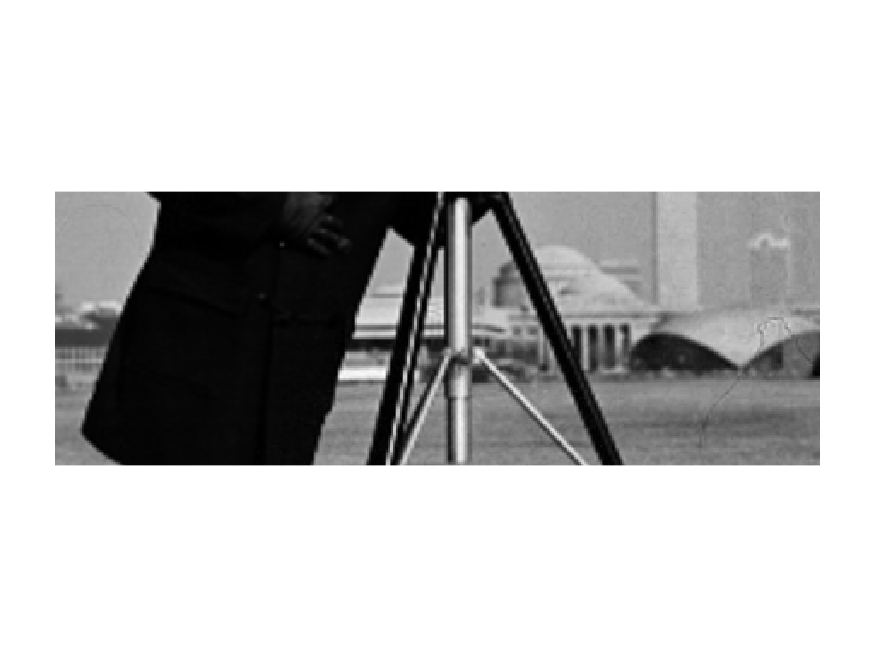}}
\subfigure{\includegraphics[width=.18\textwidth]{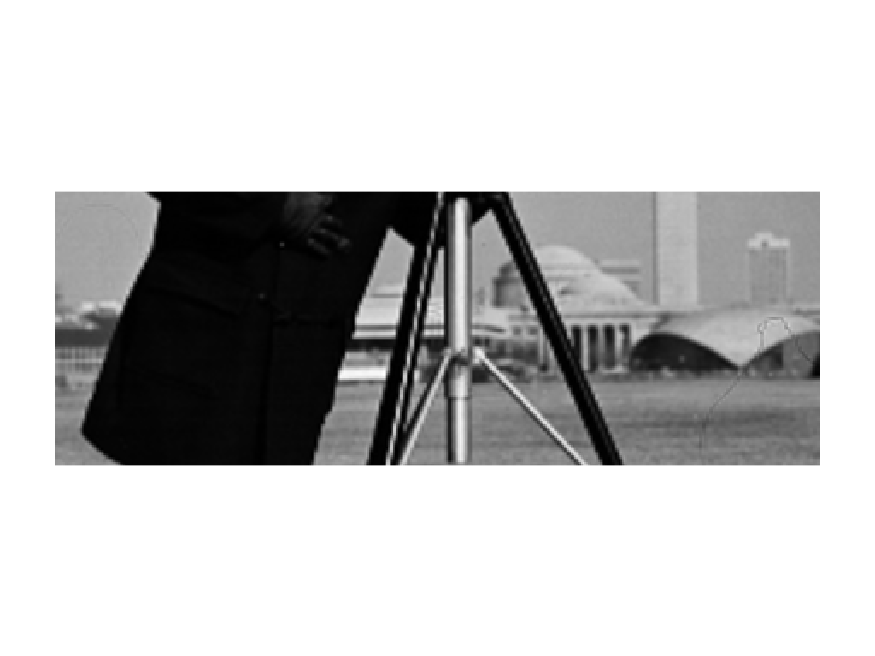}}
\vskip -.5in
\subfigure{\includegraphics[width=.18\textwidth]{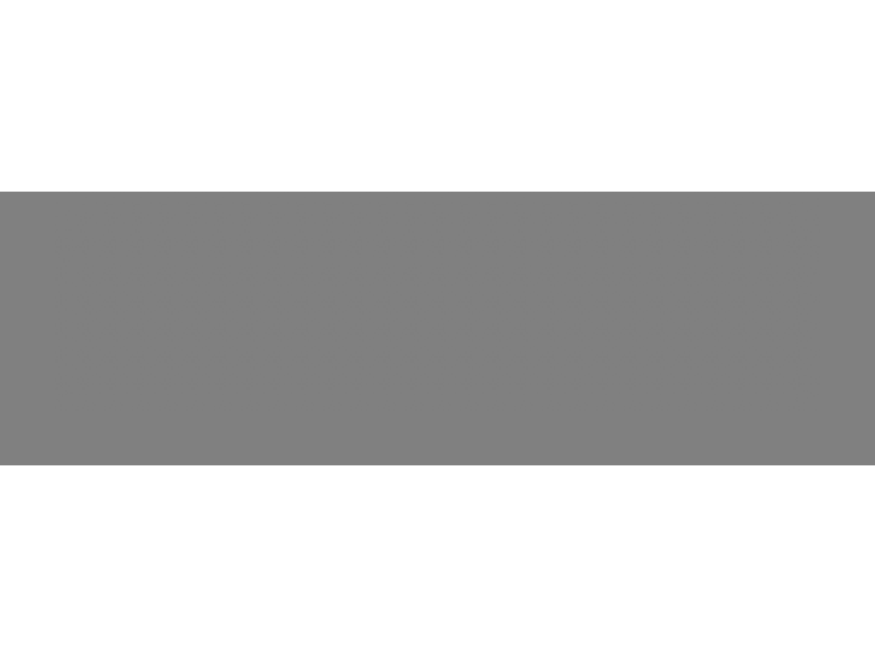}}
\subfigure{\includegraphics[width=.18\textwidth]{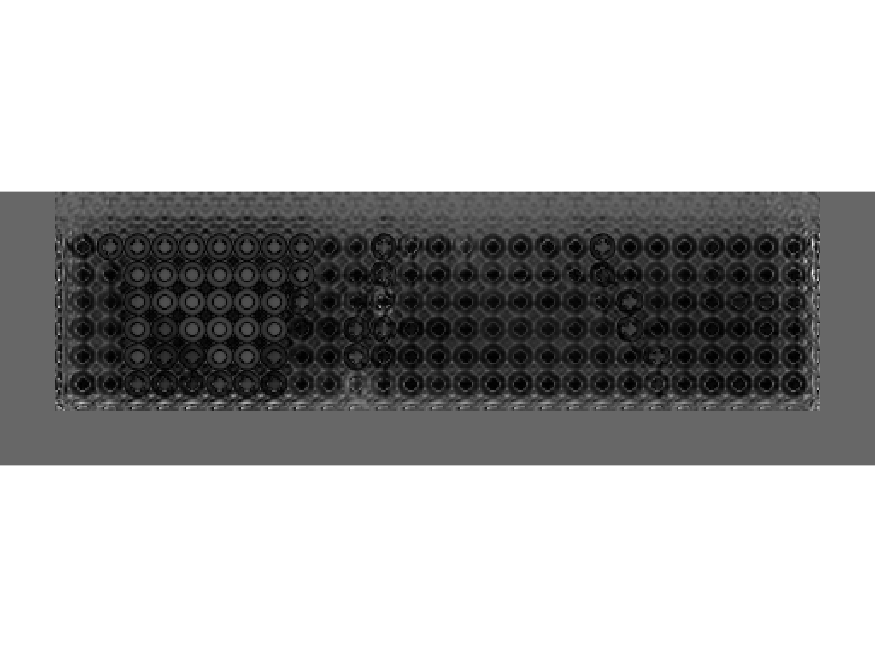}}
\subfigure{\includegraphics[width=.18\textwidth]{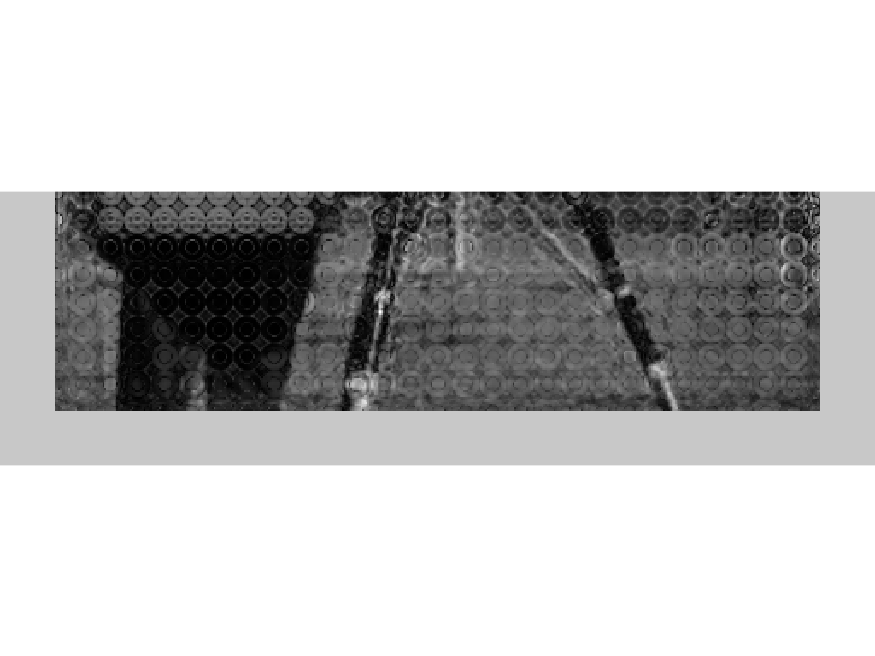}}
\subfigure{\includegraphics[width=.18\textwidth]{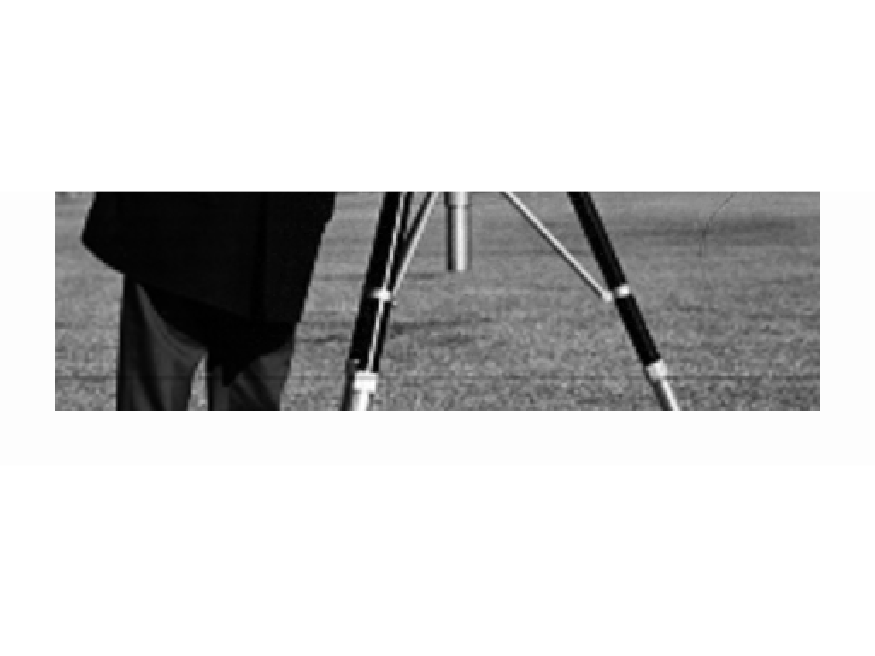}}
\subfigure{\includegraphics[width=.18\textwidth]{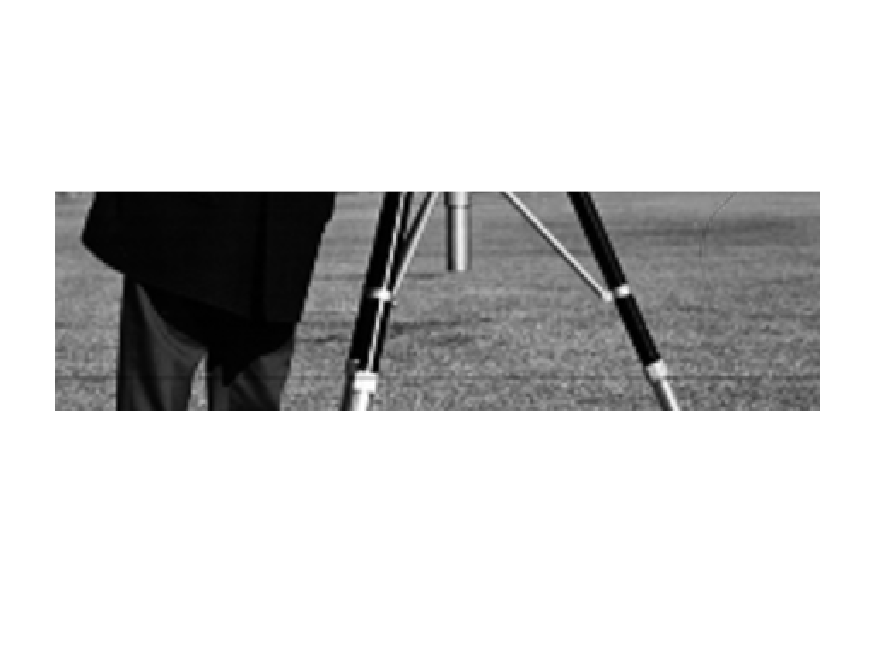}}
\\
\vspace{-.3in}
\subfigure{\includegraphics[width=.18\textwidth]{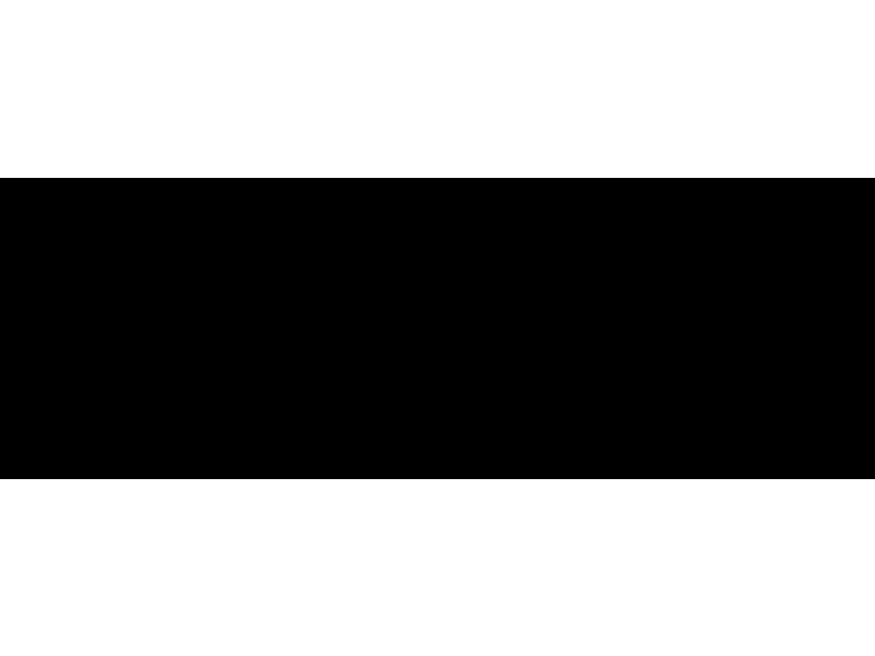}}
\subfigure{\includegraphics[width=.18\textwidth]{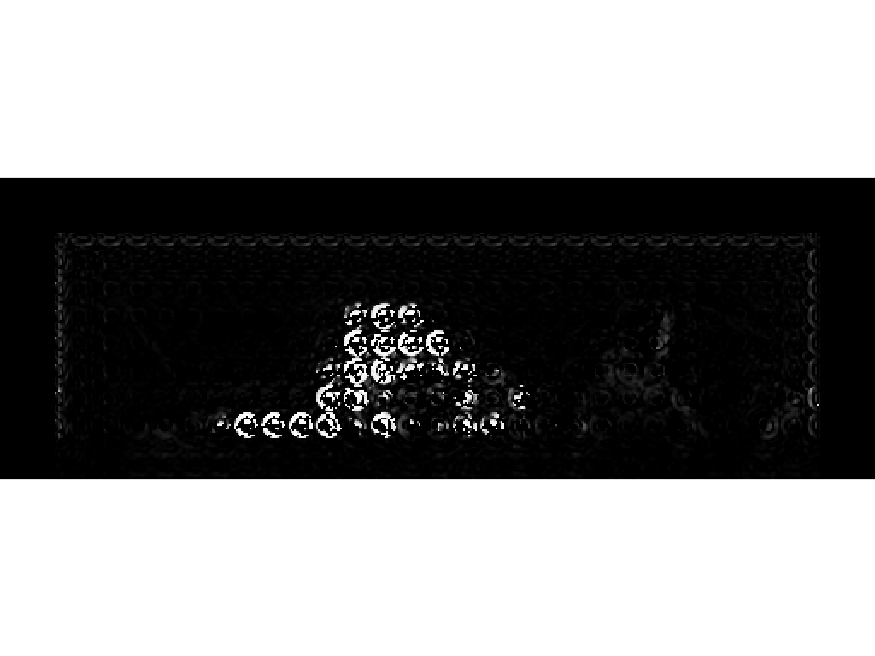}}
\subfigure{\includegraphics[width=.18\textwidth]{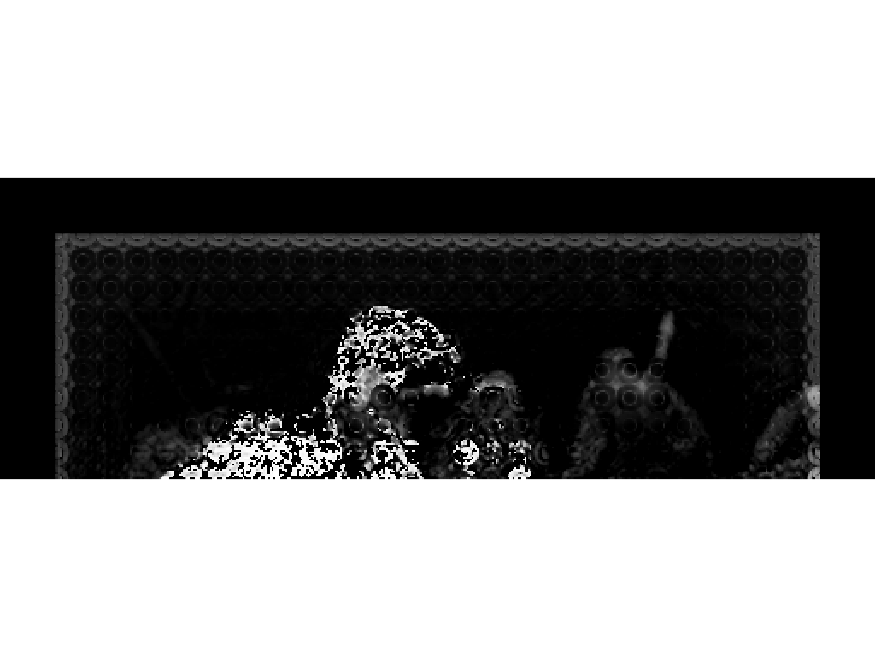}}
\subfigure{\includegraphics[width=.18\textwidth]{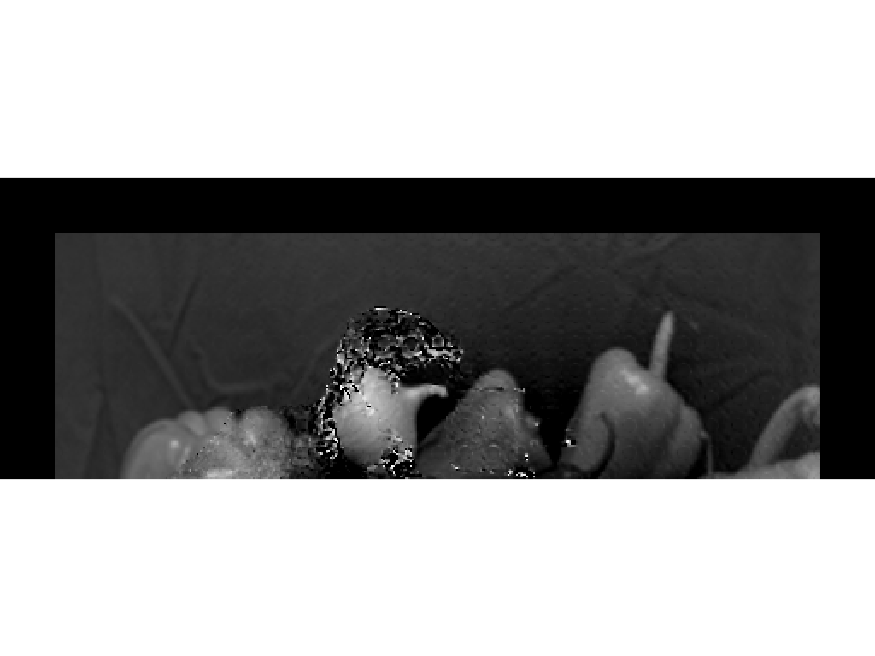}}
\subfigure{\includegraphics[width=.18\textwidth]{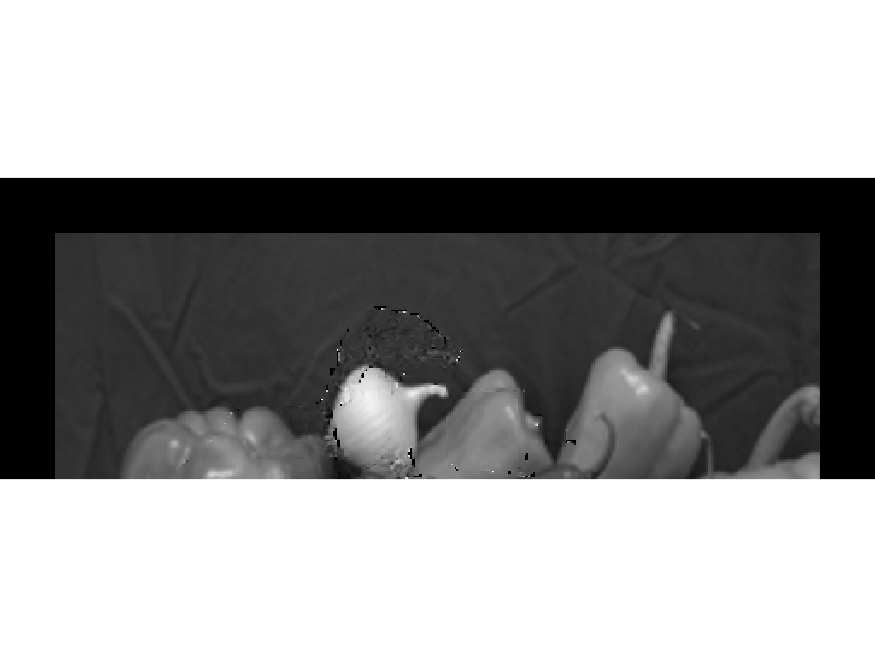}}
\vskip -.5in
\subfigure{\includegraphics[width=.18\textwidth]{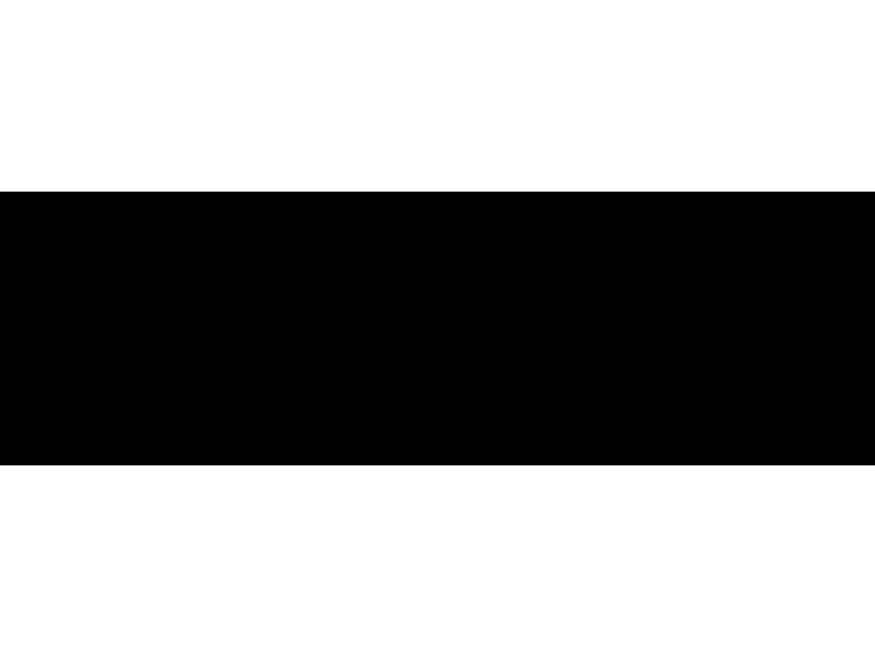}}
\subfigure{\includegraphics[width=.18\textwidth]{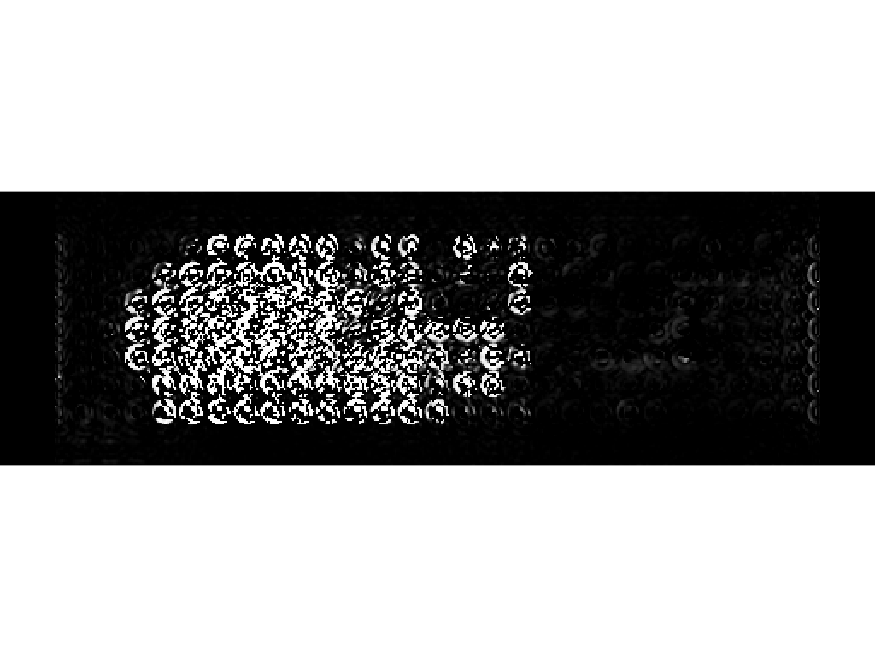}}
\subfigure{\includegraphics[width=.18\textwidth]{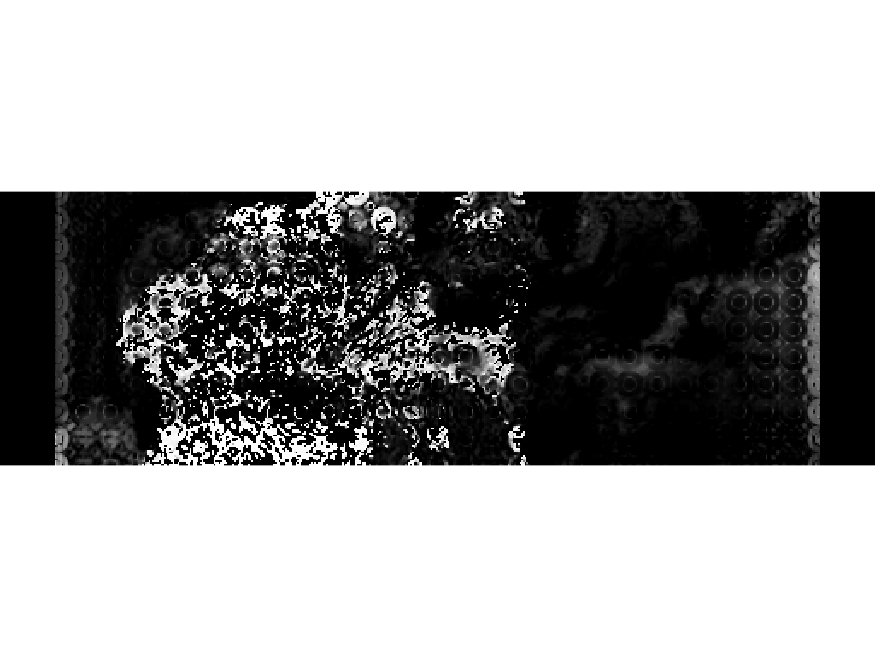}}
\subfigure{\includegraphics[width=.18\textwidth]{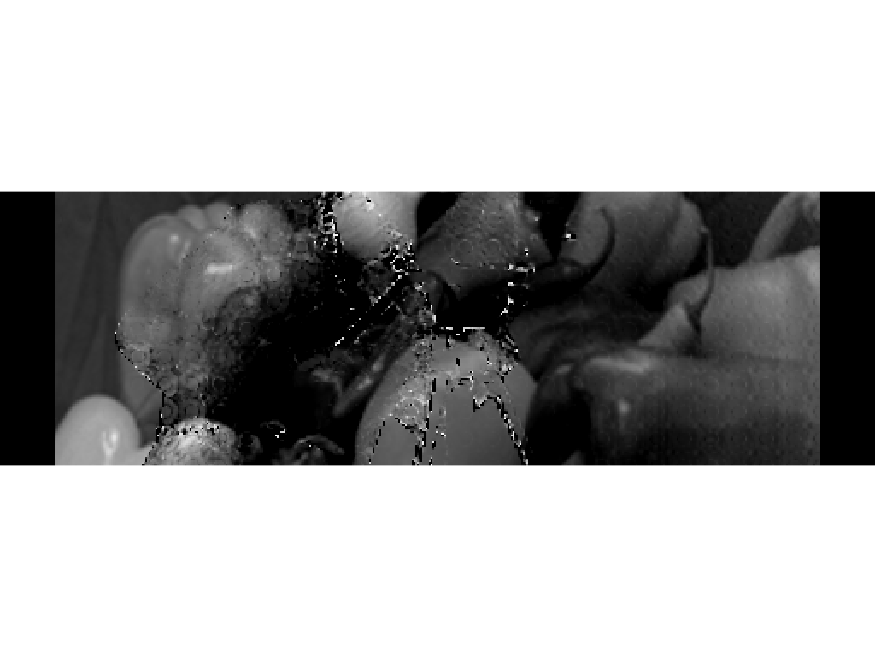}}
\subfigure{\includegraphics[width=.18\textwidth]{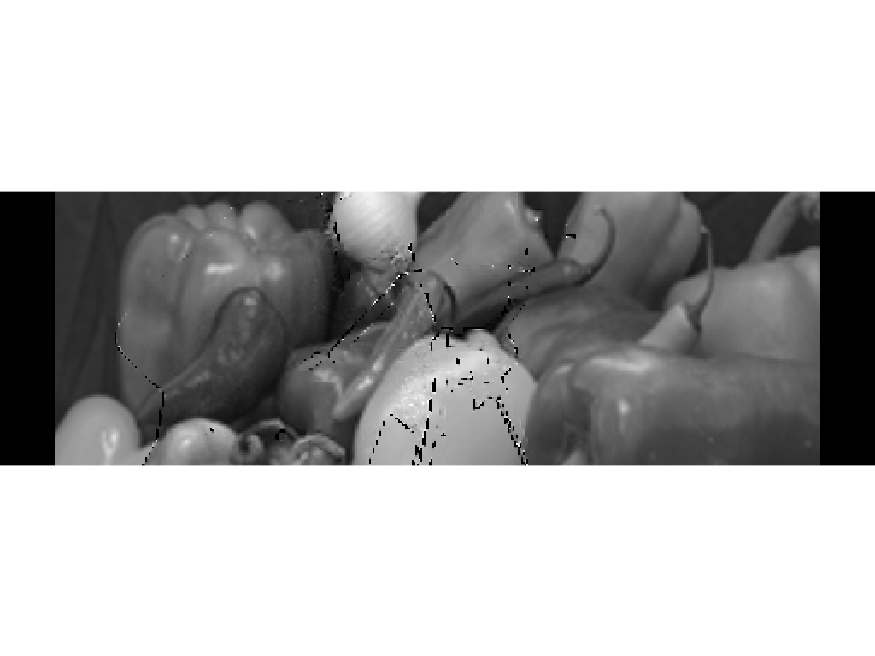}}
\vskip -.5in
\subfigure{\includegraphics[width=.18\textwidth]{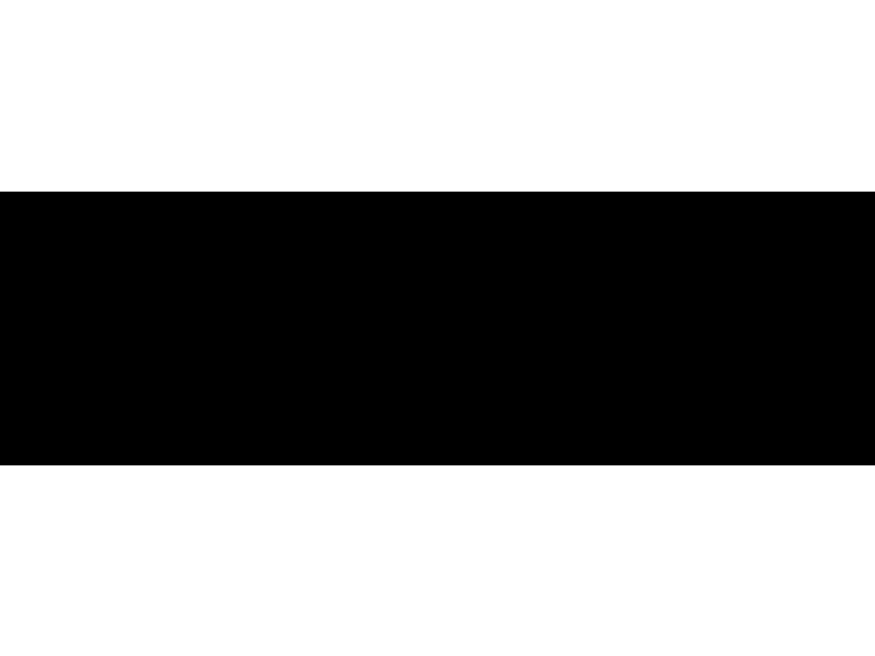}}
\subfigure{\includegraphics[width=.18\textwidth]{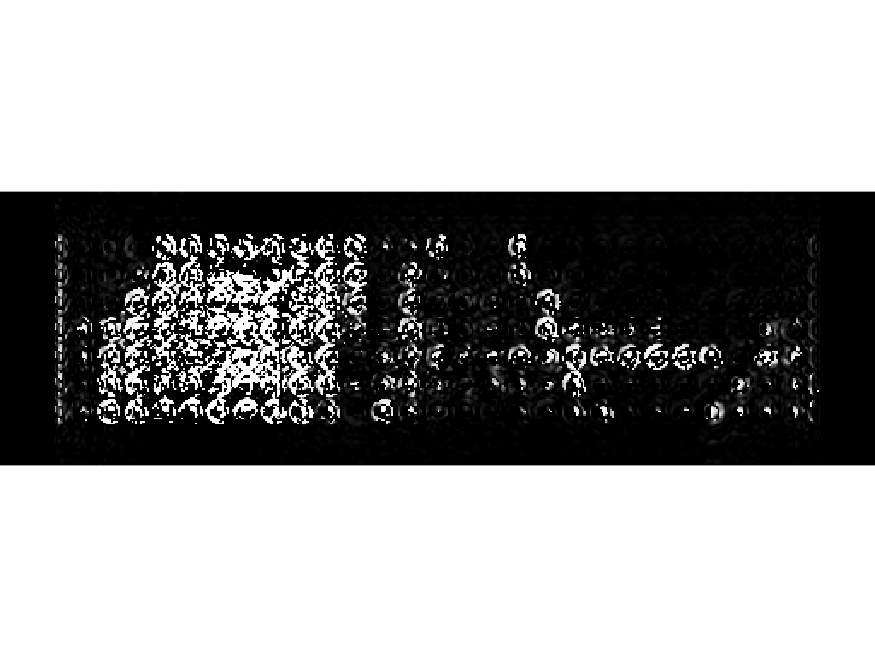}}
\subfigure{\includegraphics[width=.18\textwidth]{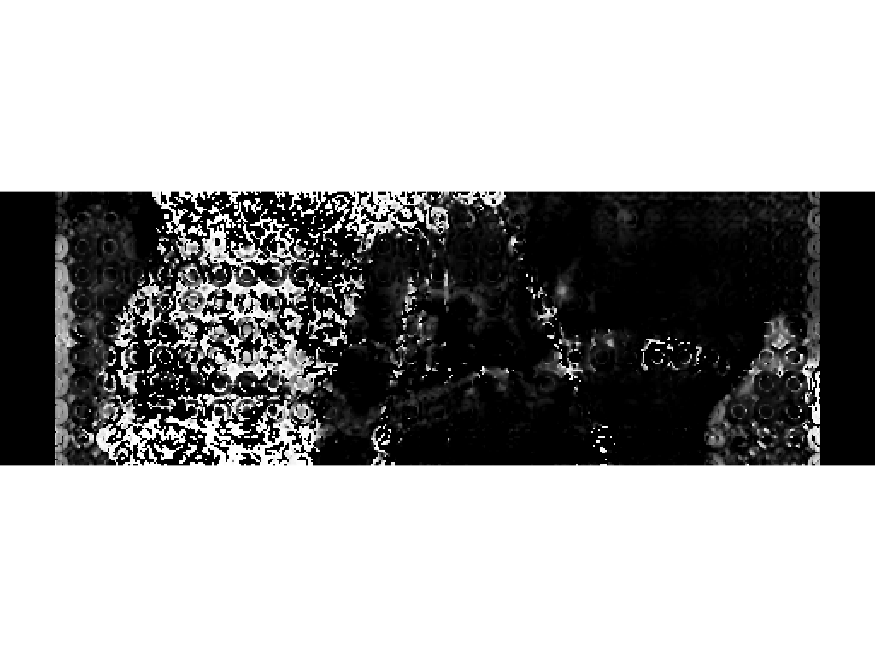}}
\subfigure{\includegraphics[width=.18\textwidth]{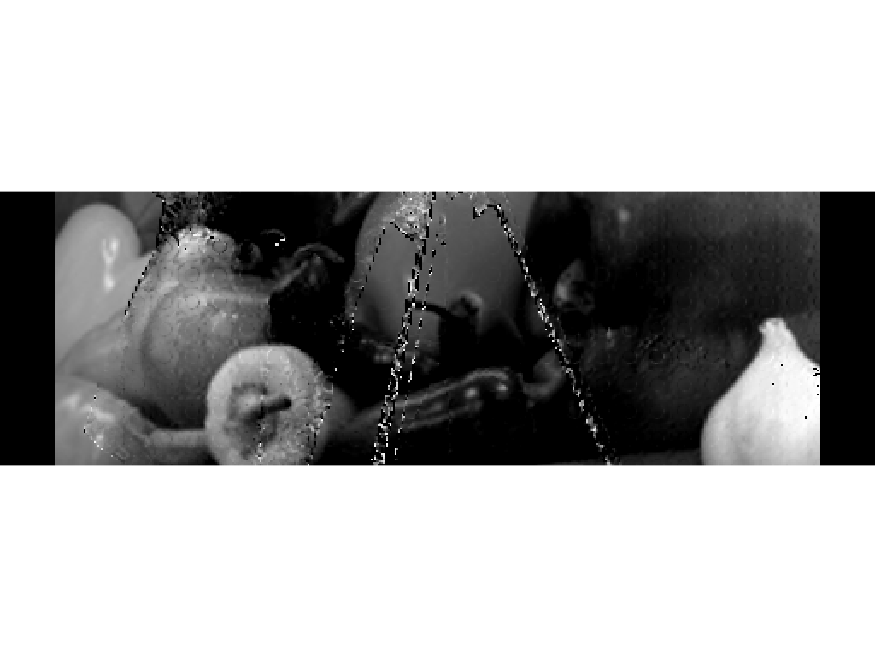}}
\subfigure{\includegraphics[width=.18\textwidth]{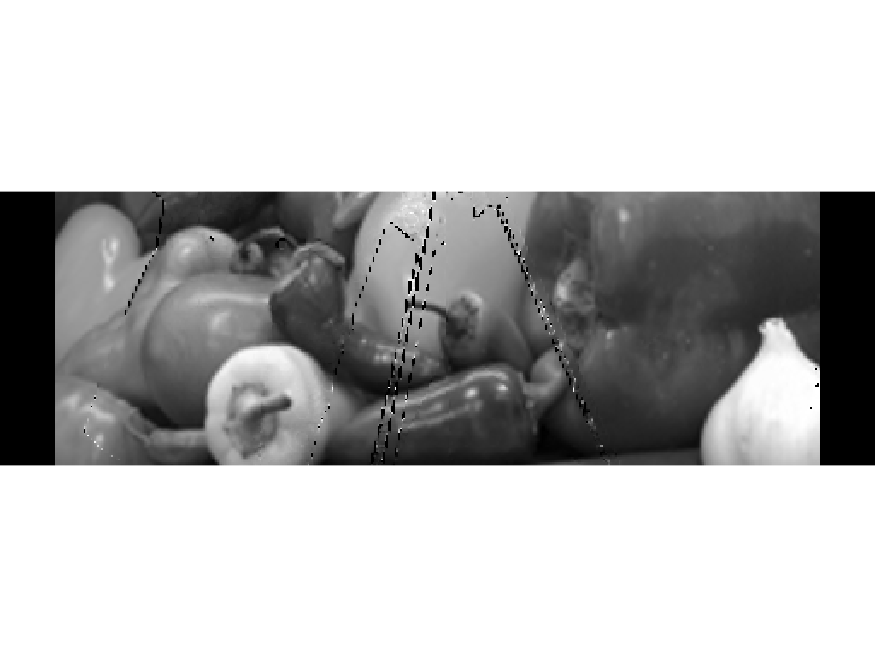}}
\vskip -.5in
\setcounter{subfigure}{0}
\subfigure[1st]{\includegraphics[width=.18\textwidth]{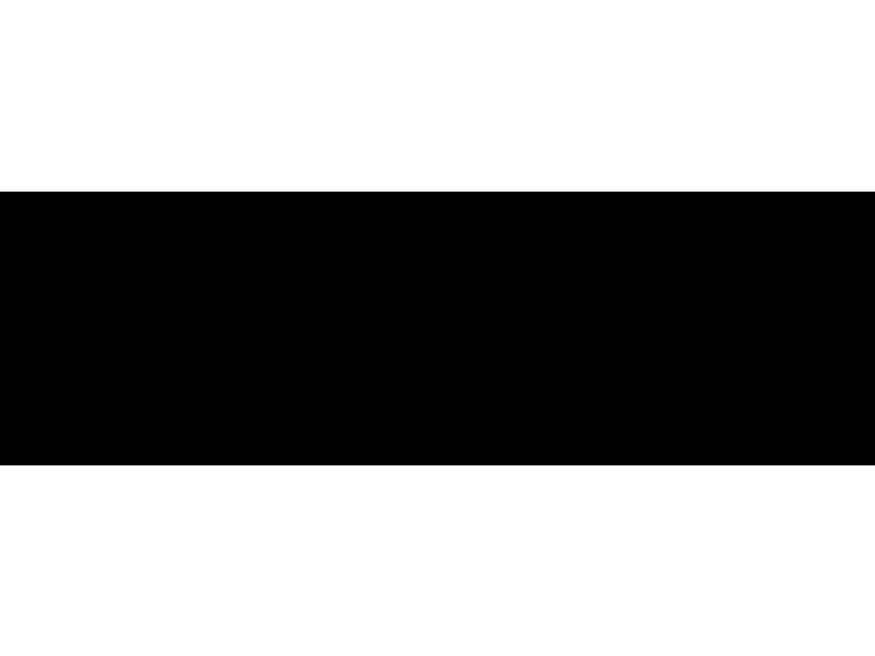}}
\subfigure[2nd]{\includegraphics[width=.18\textwidth]{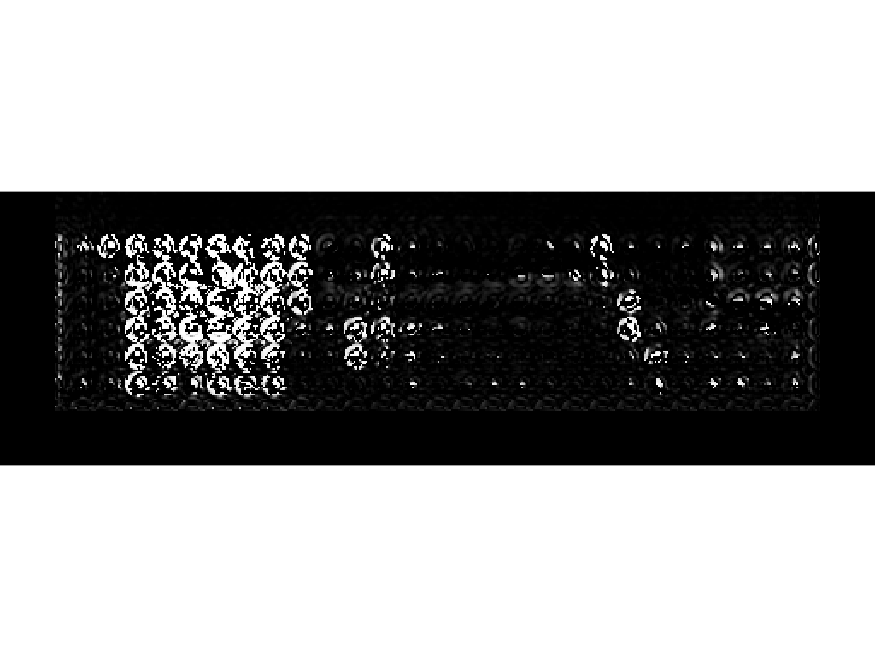}}
\subfigure[5th]{\includegraphics[width=.18\textwidth]{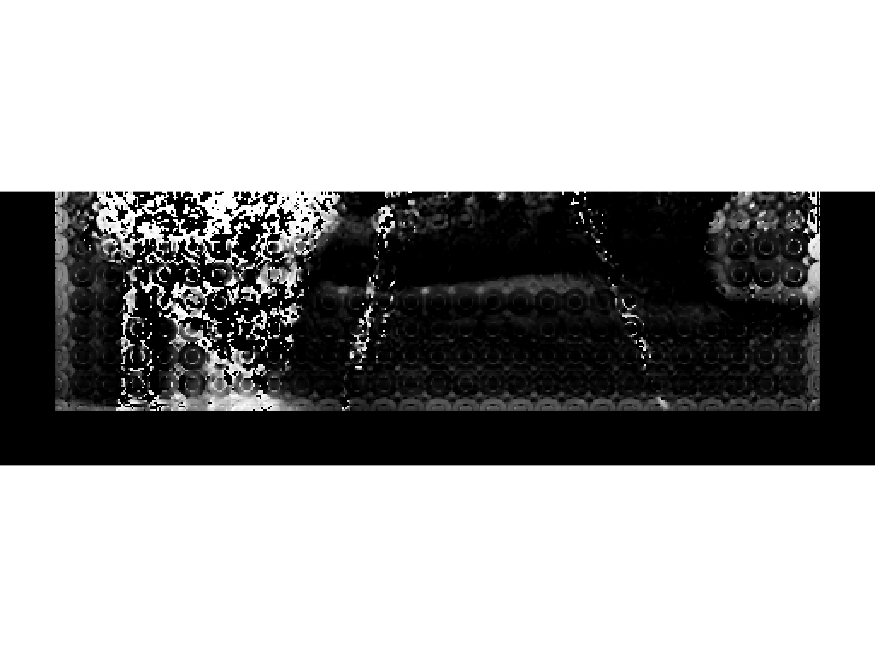}}
\subfigure[50th]{\includegraphics[width=.18\textwidth]{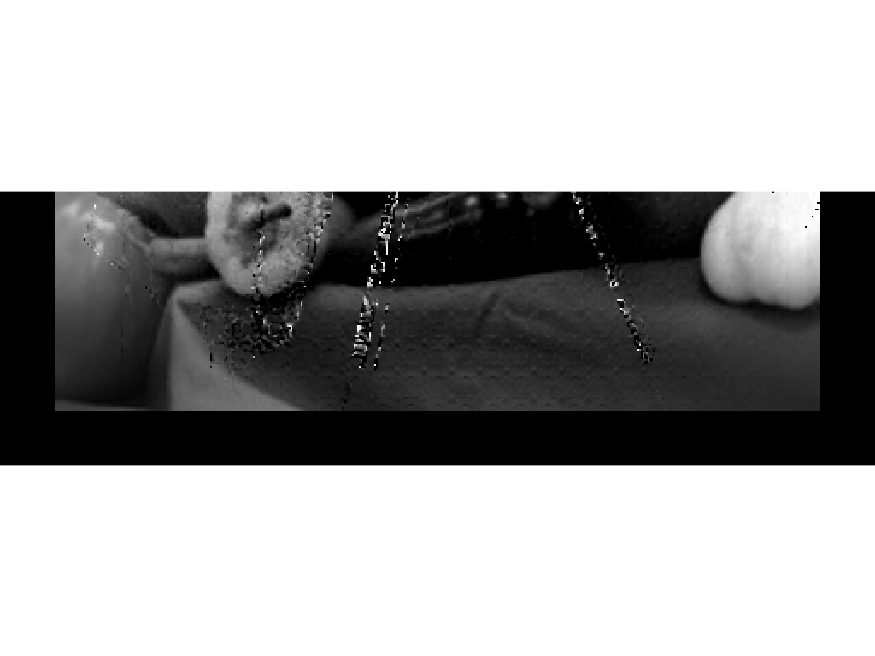}}
\subfigure[200th]{\includegraphics[width=.18\textwidth]{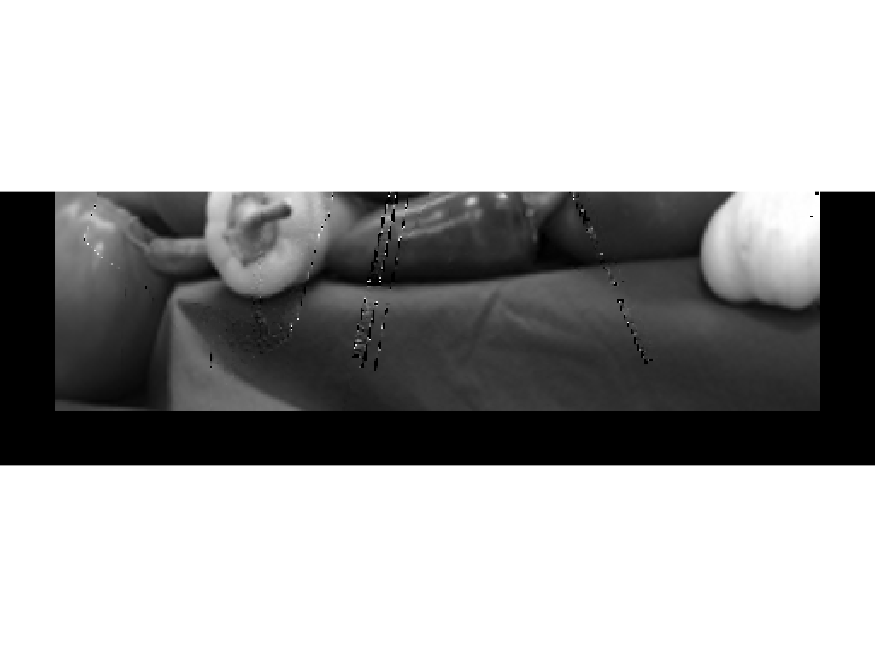}}
\end{center}
    \caption{Recovery results with $D=4$ subdomains at 1st, 2nd, 5th, 50th, 200th iterations: Recovered absolute parts (1st-4th rows) and phase parts (5th-8th rows) of the iterative solutions on the four subdomains, respectively. They are shown in the range of $[0,1]$
    and $[0,\pi]$ for the absolute and phase parts respectively.}
    \vskip -.2in
\label{fig4-2}
\end{figure}

\begin{figure}[]
\begin{center}
\subfigure[]{\includegraphics[width=.35\textwidth]{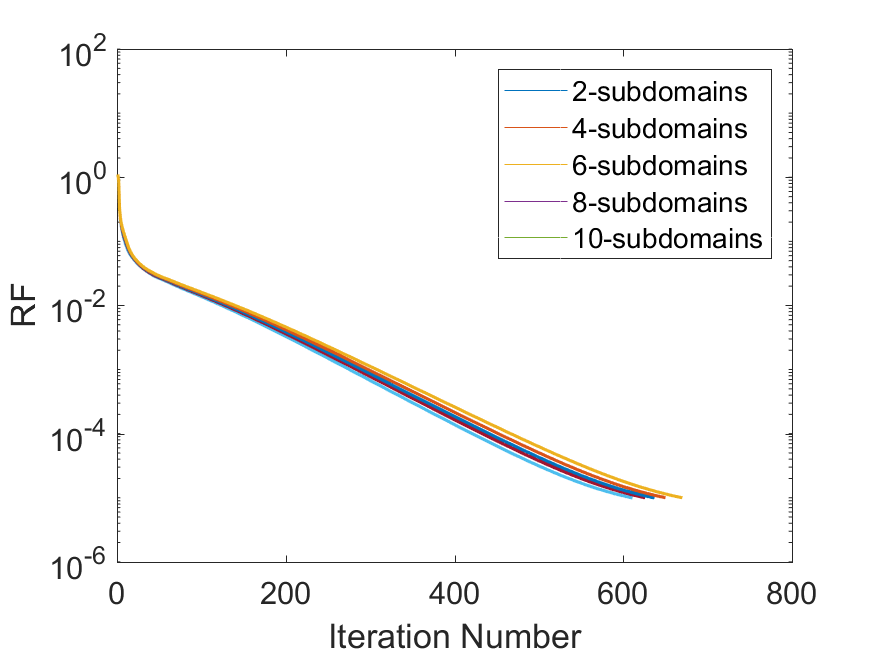}}
\subfigure[]{\includegraphics[width=.35\textwidth]{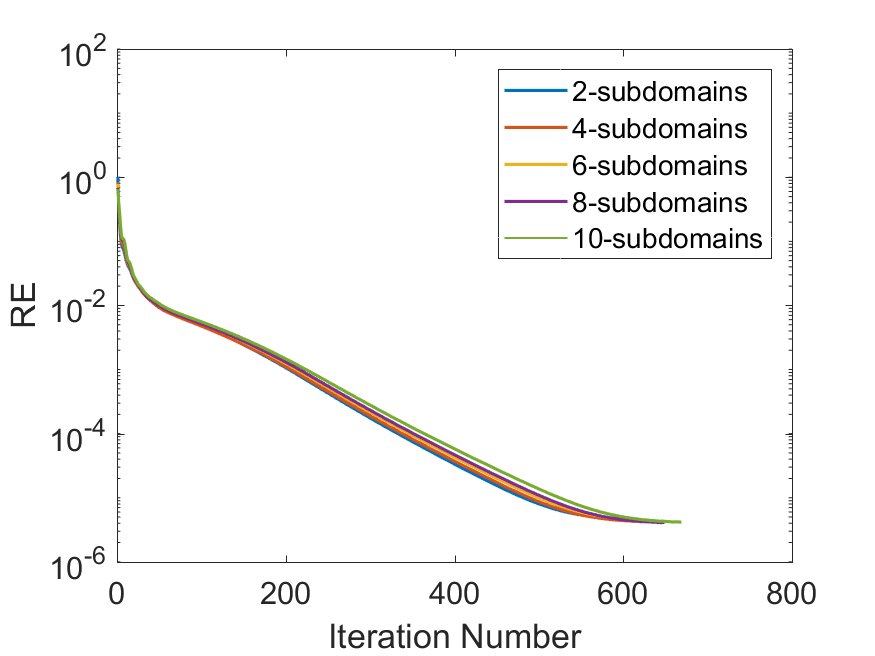}}
\end{center}
    \caption{Convergence curves w.r.t. different number of subdomains: (a) RF (R-factor), and  (b) RE (Relative error)  changes v.s. iterations. }
    \vskip -.2in
\label{fig4-3}
\end{figure}

\begin{table}[h!]
    \caption{Performances for the case of multiple subdomains. The speedup ratio is denoted by the ratio between the virtual wall-clock time and the runtime of ADMM algorithm \cite{chang2018Blind} without DD. The 2nd-3rd columns present the RF (R-factor), and RE (Relative error) for the proposed algorithms and 4th column presents the iter. no. (short for iteration number) when satisfying the stopping conditions.  }
    \centering
     \renewcommand\arraystretch{.8}
    \begin{tabular}{|c|c|c|c|c|}
    \hline
         D&RF(1E-5)&RE(1E-5)&Iter no.&\tabincell{c}{virtual wall-clock\\ time in seconds\\(Speedup ratio)}\\
         \hline
        1&0.997&0.482&587&601(1.00) \\
        \hline
        2&0.995&0.458&611&346(1.74)\\
        \hline
        4&0.998&0.426&626&167(3.60)\\
        \hline
        6&0.993&0.434&637&109(5.51)\\
        \hline
        8&0.999&0.409&651&88.0(6.83)\\
        \hline
        10& 0.999 &0.416&670 &72.0(8.35)\\
        \hline
    \end{tabular}
    \label{tab:my_label}
\end{table}

\subsection{Impact by parameters}\label{sec-4.3}

In order to evaluate the robustness of proposed  OD$^2$P w.r.t. the parameters, we conduct the experiments with noiseless data, where one parameter varies and meanwhile other parameters are kept unchanged.  Set $\eta_0=0.1, r_0=4.0\times 10^3$. The scan stepsize sets to 8 pixels. 
We show the performance impacted by  setting $\eta\in\{ 5^{-2}\times \eta_0, ~ 5^{-1}\times \eta_0,~ \eta_0,~ 5\times\eta_0,~ 5^2\times\eta_0\}$ (keeping $r=r_0$), and put the convergence curves in Fig. \ref{fig4-4} (a)-(b). Meanwhile, by setting $r\in\{ 5^{-2}\times r_0, ~ 5^{-1}\times r_0,~ r_0,~ 5\times r_0,~ 5^2\times r_0\}$ (keeping $\eta=\eta_0$), the convergence curves are put in Fig. \ref{fig4-4} (c)-(d).  Readily one knows that too smaller or bigger parameters would cause slow convergence or failure of convergence. In the future, an automatic parameter selection scheme should be developed, and we leave it as future work.
The parameter impacts to other proposed algorithms seem similar to that of OD$^2$P and  due to page limitation, we do not include them.  Here we also remark that our convergence guarantee requires that $\eta>1$ in order to make sure that the augmented Lagrangian of iterative sequence is lower bounded. However,  in the experiments, smaller values can give faster convergences, that motivates us to get much sharper estimate of the convergence condition of the parameters, and we also  leave it as future work.  

\begin{figure}[]
\begin{center}
\subfigure[]{\includegraphics[width=.24\textwidth]{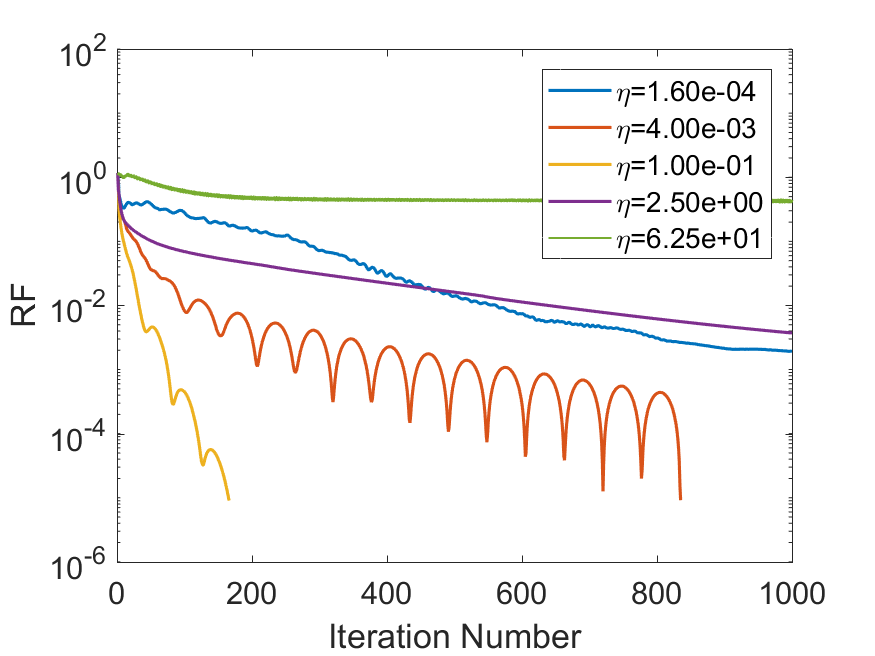}}
\subfigure[]{\includegraphics[width=.24\textwidth]{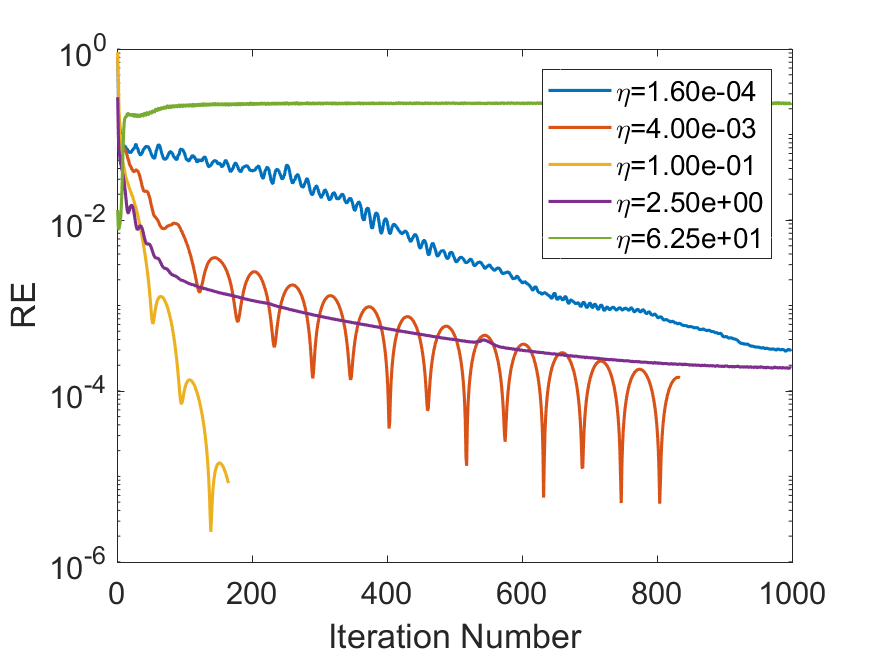}}
\subfigure[]{\includegraphics[width=.24\textwidth]{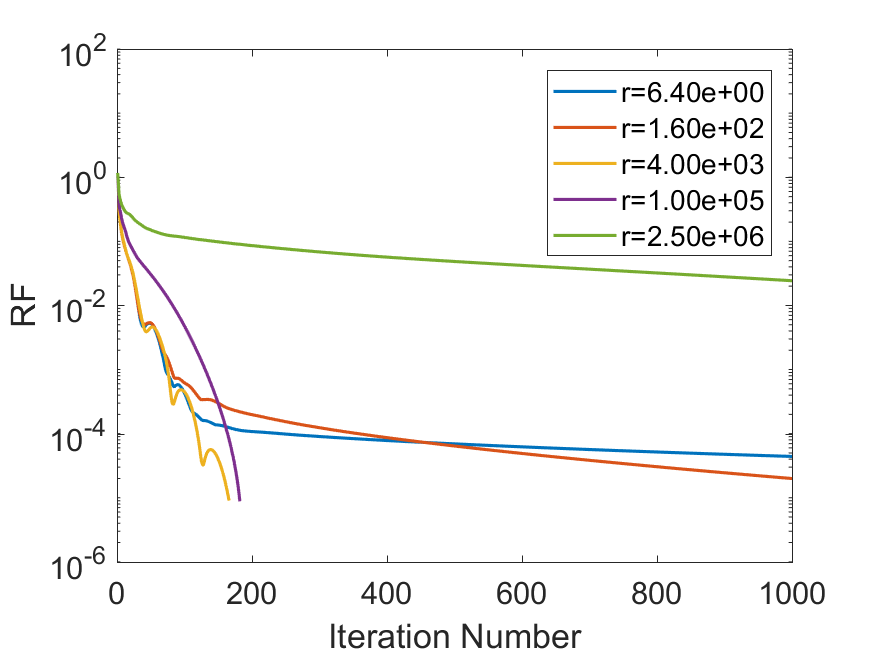}}
\subfigure[]{\includegraphics[width=.24\textwidth]{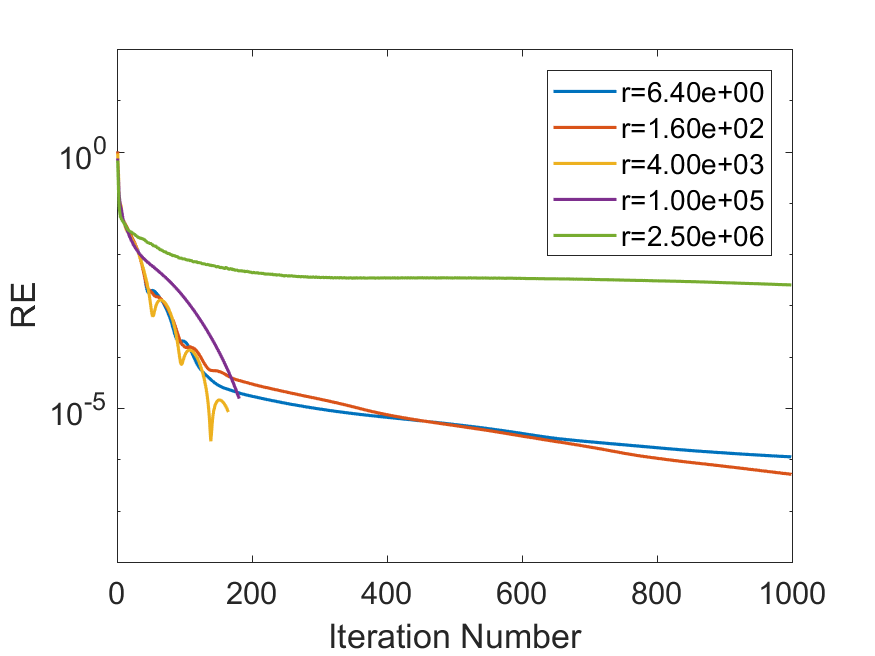}}
\end{center}
    \caption{Convergence curves w.r.t. different parameters: RF (R-factor) changes v.s. iterations with different $\eta$ in (a) and different $r$ in (c); 
    RE (Relative error)  changes v.s. iterations with different $\eta$ in (b) and different $r$ in (d);}
\label{fig4-4}
\vskip -.2in
\end{figure}

\section{Conclusions}
\label{sec5}

Overlapping DDMs have been successfully applied to ptychography reconstruction, that lead to  OD$^2$P for the two-subdomain nonblind recovery,  OD$^2$P$_m$ for multiple subdomains and OD$^2$BP for the blind recovery. With the newly-designed smooth truncated metric, these proposed algorithms are efficiently computed, since all subproblems have closed form solutions, and their convergences are well guaranteed under some mild conditions. Numerical experiments are further conducted to show the performance of proposed algorithms, demonstrating good convergence speed, robust to the noise. In the future, we will optimize the current algorithms on massively parallel processing computers, and explore more other applications including  Fourier ptychography \cite{zheng2013wide}, partial coherence analysis \cite{chang2018partially}, and high dimensional ptychographic imaging problems \cite{yu2018three}. 

\section*{Acknowledgments}
We would like to thank the two reviewers and the associate editor for
their valuable comments, which helped to improve the paper greatly. HC
acknowledges support from his hosts Professors Yang Wang and  James Sethian during the visits to the Dept. of Math. in Hong Kong University of Science and Technology and CAMERA in Lawrence Berkeley National Lab.

\bibliographystyle{siamplain}
\bibliography{rD}

\end{document}